\documentclass[12pt,fullpage]{article}
\usepackage{amsthm}
\usepackage{amsfonts}
\usepackage{xr}
\usepackage{graphicx}
\usepackage{amsmath}
\usepackage{epsfig,subfigure}
\usepackage{color}
\def\Z{\mathbb{Z}}
\def\R{\mathbb{R}}
\def\N{\mathbb{N}}
\def\C{\mathbb{C}}

\def\T{\mathbb{T}}
\def\epsilon{\varepsilon}
\def\hat{\widehat}
\def\tilde{\widetilde}

\newcommand{\me}{\mathrm{e}}

\newcommand{\SE}{\setcounter{equation}{0} \section}
\newcommand{\be}{\begin{equation}}
\newcommand{\ee}{\end{equation}}
\newcommand{\baa}{\begin{array}}
\newcommand{\eaa}{\end{array}}
\newcommand{\ba}{\begin{eqnarray}}
\newcommand{\ea}{\end{eqnarray}}

\newtheorem{theo}{\bf Theorem}[section]
\newtheorem{lem}[theo]{\bf Lemma}
\newtheorem{pro}[theo]{\bf Proposition}

\newtheorem{defi}[theo]{\bf Definition}
\newtheorem{rem}[theo]{\bf Remark}

\begin{document}
\date{}
\title{\bf{Bistable pulsating fronts for reaction-diffusion equations in a periodic habitat}}
\author{Weiwei Ding$^{\hbox{\small{ a,b}}}$, Fran{\c{c}}ois Hamel$^{\hbox{\small{ a,c}}}$ and Xiao-Qiang Zhao$^{\hbox{\small{ d}}}$\thanks{This work has been carried out in the framework of the Labex Archim\`ede (ANR-11-LABX-0033) and of the A*MIDEX project (ANR-11-IDEX-0001-02), funded by the ``Investissements d'Avenir" French Government program managed by the French National Research Agency (ANR). The research leading to these results has received funding from the European Research Council under the European Union's Seventh Framework Programme (FP/2007-2013) / ERC Grant Agreement n.321186 - ReaDi - Reaction-Diffusion Equations, Propagation and Modelling. Part of this work was carried out during visits by F.~Hamel to the Department of Mathematics and Statistics of Memorial University Newfoundland, whose hospitality is thankfully acknowledged.
X.-Q. Zhao's research is supported in part by the NSERC of Canada. W.~Ding is supported by the China Scholarship Council for 2 years of study at Aix Marseille Universit\'e.}\\
\\
\footnotesize{$^{\hbox{a }}$Aix Marseille Universit\'e, CNRS, Centrale Marseille, I2M, UMR 7373, 13453 Marseille, France}\\
\footnotesize{$^{\hbox{b }}$ School of Mathematical Sciences, University of Science and Technology of China,}\\
\footnotesize{Hefei, Anhui, 230026, P.R. China}\\
\footnotesize{$^{\hbox{c }}$ Institut Universitaire de France}\\
\footnotesize{$^{\hbox{d }}$ Department of Mathematics and Statistics, Memorial University of Newfoundland,}\\
\footnotesize{St. John's, NL A1C 5S7, Canada}}
\maketitle

\begin{abstract}
This paper is concerned with the existence and qualitative properties of pulsating fronts for spatially periodic reaction-diffusion equations with bistable  nonlinearities. We focus especially on the influence of the spatial period and, under various assumptions on the reaction terms and by using different types of arguments, we show several existence results when the spatial period is small or large. We also establish some properties of the set of periods for which there exist non-stationary fronts. Furthermore, we prove the existence of stationary fronts or non-stationary partial fronts at any period which is on the boundary of this set. Lastly, we characterize the sign of the front speeds and we show the global exponential stability of the non-stationary fronts for various classes of initial conditions.
\end{abstract}

%%%%%%%%%%%%%%%%%%%%%%%%%%%%%%%%%%%%%%%%
%%%%%%%%%%%%%%%%%%%%%%%%%%%%%%%%%%%%%%%%

\SE{Introduction and main results}

In this paper, we study the existence and qualitative properties of periodically propagating solutions of periodic reaction-diffusion equations of the type
\be\label{eqL}
u_t=(a_L(x)u_x)_x+f_L(x,u),\ \ t\in\R,\ x\in\R
\ee
with $L>0$, where $u_t$ stands for $u_t(t,x)=\partial_tu(t,x)=\partial u/\partial t(t,x)$, $u_x$ stands for $u_x(t,x)=\partial_xu(t,x)=\partial u/\partial x(t,x)$ and $(a_L(x)u_x)_x$ stands for $(a_L(x)u_x)_x(t,x)=\partial_x(a_Lu_x)(t,x)=\partial(a_Lu_x)/\partial x(t,x)$. The diffusion and reaction coefficients $a_L$ and $f_L$ are given by
$$a_L(x)=a(x/L)\ \hbox{ and }\ \ f_L(x,u)=f(x/L,u),$$
where the function $a:\R\to\R$ is positive, of class $C^{1,\alpha}(\R)$ (with $0<\alpha<1$) and $1$-periodic, that is $a(x+1)=a(x)$ for all $x\in\R$. Throughout the paper, the function $f:\R\times[0,1]\to\R,\ (x,u)\mapsto f(x,u)$ is continuous, $1$-periodic in $x$, of class $C^{0,\alpha}$ in $x$ uniformly in~$u\in[0,1]$, and of class $C^{1,1}$ in~$u$ uniformly in $x\in\R$ with $\partial_uf(\cdot,0)$ and $\partial_uf(\cdot,1)$ being continuous in $\R$. One assumes that, for every~$x\in\R$, the profile $f(x,\cdot)$ is bistable in~$[0,1]$, that is, there is $\theta_x\in(0,1)$ such that
\begin{equation}\label{bistable}
f(x,0)=f(x,1)=f(x,\theta_x)=0,\ f(x,\cdot)<0\hbox{ on }(0,\theta_x),\ f(x,\cdot)>0\hbox{ on }(\theta_x,1).
\end{equation}
One also assumes that $0$ and $1$ are uniformly (in $x$) stable zeroes of $f(x,\cdot)$, in the sense that there exist $\gamma>0$ and $\delta\in(0,1/2)$ such that
\begin{equation}\label{asspars}\left\{\baa{ll}
f(x,u)\le-\gamma u & \hbox{for all }(x,u)\in\R\times[0,\delta],\vspace{3pt}\\
f(x,u)\ge\gamma(1-u) & \hbox{for all }(x,u)\in\R\times[1-\delta,1].\eaa\right.
\end{equation}
Notice that this implies in particular that $\delta<\theta_x<1-\delta$ and $\max\big(\partial_uf(x,0),\partial_uf(x,1)\big)\le-\gamma$ for all $x\in\R$.

A particular case of such a function $f$ satisfying~\eqref{bistable} and~\eqref{asspars} is the cubic nonlinearity
\be\label{fcubic}
f(x,u)=u(1-u)(u-\theta_x),
\ee
where $x\mapsto\theta_x$ is a $1$-periodic $C^{0,\alpha}(\R)$ function ranging in $(0,1)$. Notice that in~\eqref{fcubic} or more generally in~\eqref{bistable}, the intermediate zero $\theta_x$ of $f(x,\cdot)$ is not assumed to be constant in general.

For mathematical purposes, the function $f$ is extended in $\R\times(\R\backslash[0,1])$ as follows: $f(x,u)=\partial_uf(x,0)u$ for $(x,u)\in\R\times(-\infty,0)$ and $f(x,u)=\partial_uf(x,1)(u-1)$ for $(x,u)\in\R\times(1,+\infty)$. Thus, $f$ is continuous in $\R\times\R$, $1$-periodic in $x$, $\min_{x\in\R}f(x,u)>0$ for all $u<0$ and $\max_{x\in\R}f(x,u)<0$ for all $u>1$, and $f(x,u)$, $\partial_uf(x,u)$ are globally Lipschitz-continuous in $u$ uniformly in $x\in\R$.

%%%%%%%%%%%%%%%%%%%%%%%%%%%%%%%%%%%%%%%%

\subsubsection*{Pulsating fronts}

For each $L>0$, the functions $a_L$ and $f_L(\cdot,u)$ (for all $u\in[0,1]$) are $L$-periodic (in $x$). One is especially interested in the paper in understanding the role of the spatial period $L$ on the existence of solutions connecting the two stable steady states $0$ and $1$ and propagating with a constant average speed, the so-called pulsating, or periodic, fronts. Namely, a pulsating front connecting~$0$ and~$1$ for~\eqref{eqL} is a solution $u:\R\times\R\to[0,1]$ such that there exist a real number~$c_L$ (the average speed) and a function $\phi:\R\times\R\to[0,1]$ such that
\begin{equation}\label{depulf}\left\{\baa{l}
u(t,x)=\phi(x-c_Lt,x/L)\hbox{ for all }(t,x)\in\R\times\R,\vspace{3pt}\\
\phi(\xi,y)\hbox{ is }1\hbox{-periodic in }y,\vspace{3pt}\\
\phi(-\infty,y)=1,\ \phi(+\infty,y)=0\hbox{ uniformly in }y\in\R.\eaa\right.
\end{equation}
If $c_L\neq0$, then the map $(t,x)\mapsto(x-c_Lt,x/L)$ is a bijection from $\R\times\R$ to $\R\times\R$ and~$\phi$ is uniquely determined by $u$. Furthermore, in this case, for every $x\in\R$, the function~$t\mapsto u(t,x+c_Lt)$ is~$L/c_L$-periodic (in time). The notion of pulsating front with nonzero speed was first given in~\cite{skt}. On the other hand, a pulsating front with a speed $c_L=0$ simply means a stationary solution $u(t,x)=\phi_0(x)$ of~\eqref{eqL} such that $\phi_0:\R\to[0,1]$, $\phi_0(-\infty)=1$ and~$\phi_0(+\infty)=0$. In this case, the function $u$ can be written as in~\eqref{depulf} where the function $\phi$, which is then not uniquely determined, can be defined as $\phi(\xi,y)=\phi_0(\xi)$ for all $(\xi,y)\in\R\times\R$.

In this paper, we will show the uniqueness of the speed $c_L$ of pulsating fronts for any $L>0$, and we will then give some conditions which guarantee the existence of pulsating fronts with nonzero or zero speed. In particular, we will focus on the dependence with respect to the spatial period~$L$ of the underlying medium and we will consider the limits as $L\to0^+$ (the homogenization limit) and $L\to+\infty$ (slowly varying media). We will finally discuss the global and exponential stability of pulsating fronts with nonzero speed.

%%%%%%%%%%%%%%%%%%%%%%%%%%%%%%%%%%%%%%%%

\subsubsection*{Some known existence results}

Before going further on, let us mention a very important case where the existence of fronts is known, that is the case of a homogeneous medium in the sense that the function $a$ is equal to a positive constant $d$ and the function $f$ does not depend on $x$. In this case, the function $f~:[0,1]\to\R$ is such that $f(0)=f(\theta)=f(1)=0$ for some $\theta\in(0,1)$, $f<0$ on $(0,\theta)$, $f>0$ on~$(\theta,1)$ and~$f'(0)<0$, $f'(1)<0$. It is well known~\cite{aw,fm} that for this homogeneous equation
\begin{equation}\label{homoin}
u_t=d\,u_{xx}+f(u),
\end{equation}
there exists a unique speed $c\in\R$ and a unique (up to shifts in $x$) front
\begin{equation}\label{twhomo}
u(t,x)=\phi(x-ct)\ \hbox{ such that }0<\phi<1\hbox{ in }\R,\ \ \phi(-\infty)=1\ \hbox{ and }\ \phi(+\infty)=0.
\end{equation}
Such a front is decreasing in $x$ and, for the equation~\eqref{eqL} with an arbitrary $L>0$, with $a_L=d$ and with $f_L(x,\cdot)=f$, it can also be viewed as a pulsating front with speed $c_L=c$. Furthermore, the speed $c$ has the same sign as the integral $\int_0^1f$ and the front is globally and exponentially stable~\cite{fm} (we will come back later to the precise notions of stability, directly in the framework of the periodic equation~\eqref{eqL}). In particular, if we fix $y\in\R$ and if in~\eqref{eqL} we freeze the coefficients~$a_L$ and~$f_L$ around the position~$Ly$, that is if we set $a^y(x)=a_L(Ly)=a(y)$ for all~$x\in\R$ and $f^y(u)=f_L(Ly,u)=f(y,u)$ for all~$u\in[0,1]$, then the homogeneous equation~$(u^y)_t=a^y(u^y)_{xx}+f^y(u^y)$ (with $(t,x)\in\R\times\R$) admits a unique front $u^y(t,x)=\phi^y(x-c^yt)$ and a unique speed $c^y\in\R$, with $0<\phi^y<1$ in $\R$, $\phi^y(-\infty)=1$ and~$\phi^y(+\infty)=0$. Furthermore, the speed $c^y$ has the sign of~$\int_0^1f^y(u)\,du=\int_0^1f(y,u)\,du$.

When the coefficients $a_L$ and $f_L$ of~\eqref{eqL} truly depend on $x$, the question of the existence of pulsating fronts is much more delicate. Actually, no explicit condition guaranteeing the existence or the non-existence has been known so far in general. Nevertheless, some particular cases have been dealt with and some more abstract conditions have been given. On the one hand, for a given $L>0$, the pulsating fronts are known to exist for~\eqref{eqL} if $f=f(u)$ does not depend on $x$ and if the coefficient~$a_L$ is close to a constant in some norms, see~\cite{fz,x2,x4,x5}. On the other hand, again for a given $L>0$, if the equation~\eqref{eqL} has no stable~$L$-periodic steady state $0<u(x)<1$ (see Definition~\ref{semistable} below for the precise meaning), then it admits pulsating fronts~\cite{fz}. Moreover,~\eqref{eqL} admits pulsating fronts with a positive speed under the additional assumption that at least some compactly supported initial conditions give rise to solutions converging to $1$ at $t\to+\infty$ locally uniformly in $x\in\R$, see~\cite{dgm}. However, pulsating fronts with nonzero speed~$c_L$ do not exist in general for some diffusions $a_L$ (not too close to their average) and some $x$-independent reactions~$f=f(u)=u(1-u)(u-\theta)$ with $\theta\simeq1/2$, see~\cite{x4,xz} and the comments following equation~\eqref{examc} below.

The bistability assumption~\eqref{bistable}, that is, the change of sign of $f(x,\cdot)$ in $(0,1)$ and the fact that~$\theta_x$ depends on $x$ in general, makes the questions of the existence of possible intermediate $L$-periodic steady states of~\eqref{eqL} and of the existence of pulsating fronts connecting~$0$ and~$1$ quite subtle. Many more existence results have actually been established for other types of nonlinearities~$f$ such as combustion, monostable or specific Fisher-KPP nonlinearities: for instance, for $1$-periodic (in~$x$) combustion nonlinearities $f$ for which $f\ge0$ in~$\R\times[0,1]$, $f=0$ in $\R\times\big([0,\theta]\cup\{1\}\big)$ for some $\theta\in(0,1)$ and $f(x,\cdot)$ is nonincreasing in~$[1-\delta,1]$ for some~$\delta>0$ independent of $x$, it is known that for every $L>0$, there is a unique (up to shifts in $t$) pulsating front~$u(t,x)=\phi(x-c_Lt,x/L)$ connecting~$0$ and~$1$, with $c_L>0$, see~\cite{bh2,x3}. This existence result holds whatever the size of~$a$ and~$f$ may be. On the other hand, when $f$ is positive on $\R\times(0,1)$, then for each~$L>0$, pulsating fronts $u(t,x)=\phi(x-ct,x/L)$ exist if and only if $c\ge c^*_L$, for some positive minimal speed~$c^*_L$, see~\cite{bh2,hz,lz1,lz2,w}. Furthermore, some variational formulas for~$c^*_L$ in the positive case and in the more specific KPP case have been derived in~\cite{bhn,bhr2,e} and some further qualitative and uniqueness properties for each speed $c\in[c^*_L,+\infty)$ have been given in~\cite{h,hr1,hr2}.

%%%%%%%%%%%%%%%%%%%%%%%%%%%%%%%%%%%%%%%%

\subsubsection*{Uniqueness of the speed and further qualitative properties of pulsating fronts}

In this paper, we first discuss the question of the uniqueness of the speed of pulsating fronts for the bistable equation~\eqref{eqL}, under assumptions~\eqref{bistable} and~\eqref{asspars}, as well as the monotonicity and the uniqueness of pulsating fronts with nonzero speed. In~\cite{bh3}, qualitative properties of transition fronts, which are more general than pulsating fronts, were investigated in unstructured heterogeneous media. Some results of~\cite{bh3} can be applied to the pulsating fronts of the periodic equation~\eqref{eqL}, provided that the propagation speeds are not zero. More precisely,  \cite[Theorems~1.11 and~1.14]{bh3} (see also~\cite{x1} for $x$-independent reactions $f=f(u)$) lead to the uniqueness of the speed and of the fronts (up to shifts in time) in the class of pulsating fronts with nonzero speed, as well as the negativity of $\partial_{\xi}\phi$ for a pulsating front $\phi(x-ct,x/L)$ with $c\neq0$ and $\xi=x-ct$. In the present paper, we prove the uniqueness of the speed in a more general class of pulsating fronts with zero or nonzero speed. Furthermore, we show that if the speed is not equal to $0$, then it has a well determined sign, as in the homogeneous case. Throughout the paper, we denote $\overline{f}$ the arithmetic mean of the function $f$ with respect to the variable $x$, defined by
$$\overline{f}(u)=\int_0^1f(x,u)\,dx\ \hbox{ for }u\in[0,1].$$

\begin{theo}\label{thqual}
For any fixed $L>0$, the speed of pulsating fronts for~\eqref{eqL} is unique in the sense that if  $u_L(t,x)=\phi_L(x-c_Lt,x/L)$ and $\tilde{u}_L(t,x)=\tilde{\phi}_L(x-\tilde{c}_Lt,x/L)$ are two pulsating fronts, then~$\tilde{c}_L=c_L$. Furthermore, if $c_L\neq0$, then it has the sign of~$\int_0^1\overline{f}(u)\,du$, the functions $u_L$ and $\tilde{u}_L$ are increasing in $t$ if $c_L>0$, resp. decreasing in $t$ if $c_L<0$, and the front is unique up to shifts in $t$, that is, there is~$\tau\in\R$ such that $\tilde{u}_L(t,x)=u_L(t+\tau,x)$ for all~$(t,x)\in\R^2$.
\end{theo}

When $c_L\neq0$, given the uniqueness of the speed stated in the first part of Theorem~\ref{thqual}, the monotonicity and uniqueness of $u_L$ up to shifts in $t$ can then be viewed as a consequence of~\cite{bh3}. Here, we will see that the uniqueness follows from the global stability of the pulsating fronts with nonzero speed (see Theorem~\ref{gstability} below and the end of Section~\ref{sec41} for the proof of the uniqueness in Theorem~\ref{thqual}). Theorem~\ref{thqual} also shows that the sign of the front speed, if not zero, is that of~$\int_0^1\overline{f}(u)du$ and therefore, it does not depend on $L$. But we point out that this property is only valid under the condition $c_L\neq 0$. In other words, one can not say that if $\int_0^1\overline{f}(u)du\neq 0$ then~$c_L\neq 0$. Consider for instance the equation
\begin{equation}\label{examc}
u_t=(d(x)u_x)_x+\mu^2 g(u),\ \ t\in\R,\ x\in\R,
\end{equation}
where $d(x)=1+\delta \lambda \sin(2\pi x)$, $\mu\in\R$, $\delta\in\R$, $\lambda\in\R$, and $g(u)=u(1-u)(u-1/2+\delta)$. Xin \cite{x4} proved that there are $\mu>0$, $\delta\in(0,1/2)$ and $\lambda\neq 0$ such that equation~\eqref{examc} admits a stationary front, that is, a pulsating front with speed $c=0$, whereas $\int_0^1g(u)du\neq 0$. In this case, for which pulsating fronts with nonzero speed do not exist by Theorem~\ref{thqual}, it follows from~\cite{x2,x4} that the variation of $d$ with respect to its mean value can actually not be too small.

%%%%%%%%%%%%%%%%%%%%%%%%%%%%%%%%%%%%%%%%

\subsubsection*{Existence of pulsating fronts for small $L>0$}

In this section and the following two ones, we present new results on the existence of pulsating fronts for the bistable equation~\eqref{eqL}. We first begin with the case of rapidly oscillating environments. For combustion-type and $x$-independent nonlinearities $f=f(u)$, Heinze proved in~\cite{he1} that pulsating fronts for equation~\eqref{eqL}, which exist in this case for every $L>0$, can be homogenized as $L\to~0^+$. He also showed in~\cite{he2} that the homogenization process still held for semilinear higher-dimensional reaction-diffusion equations of the type $u_t=\Delta u+f(u)$ with~$x$-independent bistable nonlinearities~$f=f(u)$ in periodically perforated domains (see the beginning of Section~\ref{smalllhomo}). By using variational principles for the speeds, the asymptotic expansions of the speeds for both models were established in~\cite{hps}. In the case of periodic Fisher-KPP type of nonlinearity, the convergence of the minimal speeds~$c^*_L$ to the homogenization limit was proved in~\cite{ehr}.

Inspired by the aforementioned homogenization results and the methods which were used, mostly in~\cite{he2}, we will show here that, under some assumptions guaranteeing their existence, pulsating fronts $u_L(t,x)=\phi_L(x-c_Lt,x/L)$ of~\eqref{eqL} converge as $L\to 0^+$, in a sense which will be made clear in Lemma~\ref{continuem}, to the following limit:
\begin{equation}\label{homogenized}
\left\{\baa{l}
a_{H}\phi_0''+c_0\phi_0'+\overline{f}(\phi_0)=0 \hbox{ in } \R,\vspace{3pt}\\
\phi_0(-\infty)=1,\ \ \phi_0(+\infty)=0,\eaa\right.
\end{equation}
where $a_H>0$ denotes the harmonic mean of the function $a$, defined by
\be\label{defaH}
a_H=\Big(\int_0^1a(x)^{-1}dx\Big)^{-1}.
\ee

\begin{theo}\label{thhomo}
Assume that there is a unique $($up to shifts$)$ front $\phi_0$, with speed $c_0\neq 0$, for the homogenized equation~\eqref{homogenized}. Then there is $L_*>0$ such that for any $0<L<L_*$, equation~\eqref{eqL} admits a unique $($up to shifts in time$)$ pulsating front $u_L(t,x)=\phi_L(x-c_Lt,x/L)$, with speed~$c_L\neq 0$, and $c_L\to c_0$ as $L\to0^+$. Lastly, up to translation of $\phi_0$, there holds $\phi_L(\xi,y)-\phi_0(\xi)\to0$ in~$H^1(\R\times(0,1))$ as~$L\to0^+$.
\end{theo}

This theorem does not only give the existence and uniqueness of pulsating fronts at small~$L>~\!\!0$, but it also provides the convergence of the speeds $c_L$ to $c_0$ as $L\to0^+$. In particular, for small~$L~\!\!>~\!\!0$, the speeds $c_L$ have the same sign as the speed~$c_0$ of the homogenized equation~\eqref{homogenized}, that is, the same sign as the integral~$\int_0^1\overline{f}(u)\,du$ (notice that this sign property could also be viewed as a consequence of Theorem~\ref{thqual}).

\begin{rem}{\rm
By the assumptions~\eqref{bistable} and~\eqref{asspars}, the function $\overline{f}$ is a~$C^1([0,1])$ function such that $\overline{f}(0)=\overline{f}(1)=0$, $\overline{f}'(0)<0$ and $\overline{f}'(1)<0$. In addition, if one assumes that there is a {\it unique} real number $\overline{\theta}\in(0,1)$ such that $\overline{f}(\overline{\theta})=0$, then, as already mentioned, equation~\eqref{homogenized} admits a unique solution $(\phi_0,c_0)$ with $0<\phi_0<1$ in $\R$, and $c_0$ has the sign of~$\int_0^1\overline{f}(u)\,du$. However, it is easy to see that there are  examples of nonlinearities $f(x,u)$ which satisfy~\eqref{bistable} and~\eqref{asspars}, but for which $\overline{f}$ has more than one zero in $(0,1)$. The possible multiple oscillations of~$\overline{f}$ on $(0,1)$ do not exclude the existence of fronts for the homogenized equation~\eqref{homogenized}, although the existence does not hold in general, see~\cite{fm}.}
\end{rem}

Recently, Fang and Zhao~\cite{fz} considered the propagation property for the following reaction-diffusion equation
\begin{equation}\label{eqfz}
u_t=(d(x)\,u_x)_x+f(u),\ \ t\in\R,\ x\in\R,
\end{equation}
where $d(x)$ is a positive $C^1$-continuous periodic function and $f(u)=u(1-u)(u-\theta)$ with~$\theta\in(0,1)$. Under an abstract setting, they first established the existence of bistable traveling waves for monotone spatially periodic semiflows, and then applied the abstract results to the semiflow generated by the solution operator associated with equation~\eqref{eqfz}. By studying the stability of the intermediate (i.e., ranging in $(0,1)$) periodic steady states, they proved the existence of pulsating fronts provided that $d(x)$ is sufficiently close to a positive constant in $C^0$-norm. Actually, we will show in the present paper that some results of~\cite{fz} can be used to prove the existence of pulsating fronts for~\eqref{eqL} when $L$ is small enough, under some assumptions on $\overline{f}$ but for general diffusion coefficients $a$ (we will consider later the case of large periods~$L$). More precisely, we have the following result:

\begin{theo}\label{thhomobis}
Assume that there is $\overline{\theta}\in(0,1)$ such that
\be\label{hypoverf}
\overline{f}<0\hbox{ on }(0,\overline{\theta}),\ \overline{f}>0\hbox{ on }(\overline{\theta},1),\hbox{ and }\overline{f}'(\overline{\theta})>0.
\ee
Then there is $\tilde{L}_*>0$ such that for all $0<L<\tilde{L}_*$, equation~\eqref{eqL} admits a pulsating front~$u_L(t,x)=\phi_L(x-c_Lt,x/L)$ with speed $c_L$, and $c_L\to c_0$ as $L\to0^+$, where $c_0$ is the unique speed for the homogenized equation~\eqref{homogenized}. Furthermore, $c_L=c_0=0$ for all $0<L<\tilde{L}_*$ if $c_0=0$.
\end{theo}

The results stated in Theorems~\ref{thhomo} and~\ref{thhomobis} and the methods used to prove them are different. On the one hand, the proof of Theorem~\ref{thhomo} relies on a perturbation argument and on the application of the implicit function theorem in some suitable Banach spaces. It provides the local uniqueness (up to shifts in time) of the pulsating fronts~$u_L(t,x)=\phi_L(x-c_Lt,x/L)$, as well as the uniqueness and the nonzero sign of the speeds~$c_L$. The proof uses as an essential ingredient the fact that the front speed of the homogenized equation~\eqref{homogenized} is not zero. On the other hand, the main tool in the proof of Theorem~\ref{thhomobis} is to show that equation~\eqref{eqL} does not admit any semistable $L$-periodic steady state ranging in $(0,1)$, and this strategy may well give rise to pulsating fronts~$u_L(t,x)=\phi_L(x-c_Lt,x/L)$ with speed $c_L=0$. Consider for instance the equation~\eqref{examc} again, which admits some stationary fronts for some parameters $d(x)$, $\mu>0$, $g(u)=u(1-u)(u-1/2+\delta)$ and $\delta\in(0,1/2)$. For the equation~\eqref{eqL} with $a_L(x)=d(x/L)$ and~$f_L(x,u)=f(u)=\mu^2g(u)$, Theorem~\ref{thhomo} can be applied since~$\overline{f}(u)=\mu^2u(1-u)(u-1/2+\delta)$ and the speed $c_0$ associated with~\eqref{homogenized} is positive (it has the sign of~$\int_0^1\overline{f}$). As a consequence, the period $L_*$ given in Theorem~\ref{thhomo} satisfies $L_*\le1$, since the interval $(0,L_*)$ is an interval of existence (and uniqueness) of pulsating fronts with {\it nonzero} speeds. Theorem~\ref{thhomobis} can also be applied in this case since $\overline{f}$ satisfies~\eqref{hypoverf} with~$\overline{\theta}=1/2-\delta$. We believe that in this case, for every~$L>0$, equation~\eqref{eqL} has no semistable~$L$-periodic steady state ranging in $(0,1)$. If that were true, then the method used in the proof of Theorem~\ref{thhomobis} would imply that the threshold~$\tilde{L}_*$ would actually be infinite. More generally speaking, Theorem~\ref{thhomo} provides an interval of existence and uniqueness of pulsating fronts with nonzero speeds, while the fronts given in Theorem~\ref{thhomobis} may be stationary in general.

Finally, we point out that even if Theorem~\ref{thhomo} can cover the case of functions $\overline{f}$ with multiple oscillations on the interval $[0,1]$, it does not hold if $\int_0^1\overline{f}(u)du=0$, while Theorem~\ref{thhomobis} does, under the additional assumption~\eqref{hypoverf}. As an example, fix an $x$-independent function $f=f(u)$ satis\-fying~\eqref{bistable} and~\eqref{asspars} and fix a positive constant $d>0$. It follows from~\cite{fz} that there is $\eta>0$ such that for every $L>0$ and for every $a$ satisfying the general assumptions of the present paper together with~$\|a-d\|_{L^{\infty}(\R)}\le\eta$, equation~\eqref{eqL} admits a pulsating front. From Theorem~\ref{thqual}, this pulsating front has to be stationary since $\int_0^1\overline{f}(u)\,du=\int_0^1f(u)\,du=0$. The existence of these stationary fronts cannot be covered by Theorem~\ref{thhomo}, but it can by Theorem~\ref{thhomobis}, for small~$L>0$ (and actually, for all $L>0$ since all $L$-periodic stationary states $0<\bar{u}(x)<1$ are unstable if~$\|a-d\|_{L^{\infty}(\R)}\le\eta$, and the proof of Theorem~\ref{thhomobis} is based on this property).

%%%%%%%%%%%%%%%%%%%%%%%%%%%%%%%%%%%%%%%%

\subsubsection*{Existence of pulsating fronts for large $L>0$}

The method used in the proof of Theorem~\ref{thhomobis} can also be applied to prove the existence of pulsating fronts of equation~\eqref{eqL} with large $L>0$.

\begin{theo}\label{thlarge}
Assume that
\begin{equation}\label{conlarge}
\int_0^1\!\!f(x,u)\,du>0\ \hbox{ and }\ \frac{\partial f}{\partial u}(x,\theta_x)>0\ \hbox{ for all }x\in\R.
\end{equation}
Then there is $L^*>0$ such that for all $L>L^*$, equation~\eqref{eqL} admits a pulsating front $u_L(t,x)=\phi_L(x-c_Lt,x/L)$ with speed $c_L>0$.
\end{theo}

Notice that, similarly, equation~\eqref{eqL} admits a pulsating front with negative speed $c_L<0$ for large~$L$ large if~\eqref{conlarge} is replaced by~$\int_0^1f(x,u)\,du<0$ for all $x\in\R$. On the other hand, if $\min_{x\in\R}\int_0^1f(x,u)du<0<\max_{x\in\R}\int_0^1f(x,u)du$, then there is in general no pulsating front with nonzero speed for large $L$, but there are stationary fronts (see our upcoming paper~\cite{dhz2}).

Theorem~\ref{thlarge} shows the existence of pulsating fronts with speed $c_L>0$ for large $L$. Actually, in the proof of Theorem~\ref{thlarge}, we first prove the instability of all intermediate $L$-periodic steady states to get the existence of pulsating fronts with $c_L\ge0$ for $L$ large under the assumption~\eqref{conlarge} and we exclude the case $c_L=0$ for large $L$ by using again~\eqref{conlarge}. We also point out that, as in Theorems~\ref{thhomo} and~\ref{thhomobis}, the existence of pulsating fronts stated in Theorem~\ref{thlarge} does not require that the coefficients of~\eqref{eqL} be close to their spatial average. Thus, one can say that Theorems~\ref{thhomo},~\ref{thhomobis} and~\ref{thlarge} are of a different spirit from the aforementioned existence results of~\cite{fz,x2,x4}.

\begin{rem}{\rm From the proofs of Theorems~\ref{thhomobis} and~\ref{thlarge} in Sections~\ref{smalllhomobis} and~\ref{seclarge}, one will see that the conditions~\eqref{hypoverf} and~\eqref{conlarge} are only used to show the instability of the $L$-periodic steady states of equation~\eqref{eqL}, and this stability property is invariant by changing $x$ into $-x$. Thus, these conditions do not only imply the existence of fronts $u(t,x)=\phi(x-c_Lt,x/L)$ satisfying~\eqref{depulf}, but they also provide the existence of pulsating fronts of the type $\tilde{u}(t,x)=\tilde{\phi}(x-\tilde{c}_Lt,x/L)$ where~$\tilde{\phi}$ is $1$-periodic in its second variable and satisfies reversed limiting conditions:
$$\tilde{\phi}(-\infty,y)=0\,\,\,\hbox{and}\,\,\, \tilde{\phi}(+\infty,y)=1,\ \ \hbox{uniformly for }y\in\R.$$
Moreover, if both speeds $c_L$ and $\tilde{c}_L$ are nonzero, then they must have the same sign, which is that of $\int_0^1\overline{f}(u)du$, as a consequence of Theorem~\ref{thqual}. But whether $\tilde{c}_L=c_L$ for general coefficients $a$ and $f$ is not clear in general. We thank Dr.~X.~Liang for mentioning this question.}
\end{rem}

\begin{rem}{\rm Let $a$ and $f$ satisfy the general assumptions of the paper, in particular, ~\eqref{bistable},~\eqref{asspars} as well as the $1$-periodicity in $x$, and consider now the equation
\be\label{eqM}
u_t=(a(x)u_x)_x+M\,f(x,u),\ \ t\in\R,\ x\in\R,
\ee
where the positive parameter $M$ measures the amplitude of the reaction. A solution $u$ of this $1$-periodic equation is a pulsating front with speed $\sigma$ if and only if the function $v(t,x)~\!\!=~\!\!u(t/M,x/\sqrt{M})$ is a pulsating front for the equation~\eqref{eqL} with $L=\sqrt{M}$ and speed~$\sigma/\sqrt{M}$. Therefore, under the assumption of Theorem~\ref{thhomo}, there is $M_*>0$ such that for all~$0<M<M_*$, equation~\eqref{eqM} admits a unique (up to shifts in time) pulsating front, with speed~$\sigma_M\neq0$, and~$\sigma_M\sim\sqrt{M}\,c_0$ as $M\to0^+$. Similarly, under the assumption~\eqref{hypoverf} of Theorem~\ref{thhomobis}, there is~$\tilde{M}_*>0$ such that for all $0<M<\tilde{M}_*$, equation~\eqref{eqM} admits a pulsating front with speed~$\sigma_M$, and $\sigma_M\sim\sqrt{M}\,c_0$ as $M\to0^+$ if the unique speed $c_0$ of the homogenized equation~\eqref{homogenized} is not zero, while $\sigma_M=0$ for all $0<M<\tilde{M}_*$ if $c_0=0$. Lastly, under the assumption~\eqref{conlarge} of Theorem~\ref{thlarge}, there is $M^*>0$ such that for all $M>M^*$, equation~\eqref{eqM} admits a pulsating front with speed~$\sigma_M>0$.}
\end{rem}

%%%%%%%%%%%%%%%%%%%%%%%%%%%%%%%%%%%%%%%%

\subsubsection*{Set of periods $L$ for which pulsating fronts with nonzero speed exist}

After establishing some conditions for the existence of pulsating fronts for small or large periods, we derive some properties of the set of periods $L$ for which~\eqref{eqL} admits pulsating fronts with nonzero speeds. We had already emphasized the particular role played by the stationary fronts and we focus in this section on the fronts with nonzero speeds. We define
\be\label{defE}
E=\Big\{L>0,\ \ (\ref{eqL})\hbox{ admits a pulsating front with speed }c_L\neq 0\Big\}
\ee
and we investigate the properties of $E$ under the assumption $\int_0^1\overline{f}(u)\,du\neq0$. Indeed, if~$\int_0^1\overline{f}(u)\,du=0$, then Theorem~\ref{thqual} excludes the existence of pulsating fronts with nonzero speeds for any period $L$: in other words, $E$ is empty in this case. It will also follow from~\cite{dhz2} that there is a constant~$C$ which only depends on $a$ and $f$ such that
\be\label{bounded}
 |c_L|\le C,\quad \forall\,L\in E,
\ee
that is, the front speeds are globally bounded independently of the period~$L$. As we will see in~\cite{dhz2}, the same property actually holds for the broader notion of global mean speeds of generalized transition fronts. In the present paper, we do not prove this global property~\eqref{bounded} and we deal with local properties of the set $E$. In particular, it will follow from Theorem~\ref{thE1} below that the speeds $c_L$ with $L\in E$ are locally bounded around any point $L_0\in E$ and we will also prove in Lemma~\ref{lemspeeds} below that the speeds $c_L$ with $L\in E$ are locally bounded around any boundary point~$L_0\in\partial E\cap(0,+\infty)$. Motivated by the implicit function theorem used in the homogenization process in Theorem~\ref{thhomo}, we will first prove the following result.

\begin{theo}\label{thE1}
The set $E$ is open and for any $L_0\in E$, one has $c_L\to c_{L_0}$ as $L\to L_0$ and $L\in E$.
\end{theo}

Under the assumptions of Theorems~\ref{thhomobis} and~\ref{thlarge}, that is, under conditions~\eqref{hypoverf} and~\eqref{conlarge}, it is natural to wonder whether $E=(0,+\infty)$, namely, whether there exist pulsating fronts with nonzero speed for all $L>0$.  As a matter of fact, the answer is no in general, since quenching may occur even for some $x$-independent nonlinearities $f=f(u)$ (see again the example given with~\eqref{examc}), that is, stationary fronts connecting $0$ and $1$ may exist. In view of Theorem~\ref{thqual}, the periods $L$ for which~\eqref{eqL} admits stationary fronts cannot belong to $E$, but they may appear at the boundary of $E$. Hence, it is of interest to investigate the question of the solutions of equation~\eqref{eqL} when $L\in(0,+\infty)$ is a boundary point of $E$.

To do so, we first need to define precisely the notion of stability of periodic steady states and, actually, that of general steady states. Namely, let $L>0$ and let~$\bar{u}:\R\to[0,1]$ denote a steady state of~\eqref{eqL}. For any $R>0$, let~$\lambda_{1,R}(L,\bar{u})$ be the unique real number $\lambda$ such that there exists a~$C^2([-R,R])$ function $\psi$ satisfying
\begin{equation}\label{truncted}\left\{\baa{l}
(a_L(x)\psi')'+\partial_uf_L(x,\bar{u}(x))\psi=\lambda\psi\hbox{ in }[-R,R],\vspace{3pt}\\
\psi>0\hbox{ in }(-R,R),\ \ \psi(-R)=\psi(R)=0.\eaa\right.
\end{equation}
The real number $\lambda_{1,R}(L,\bar{u})$ is the principal eigenvalue of equation~\eqref{truncted}, and $\psi$ is the (unique up to multiplication) corresponding eigenfunction. It is well known that $\lambda_{1,R}(L,\bar{u})$ exists  uniquely, and that $\lambda_{1,R}(L,\bar{u})$ is increasing in $R$.

\begin{defi}\label{semistable}
Let $\lambda_1(L,\bar{u})=\lim_{R\to+\infty}\lambda_{1,R}(L,\bar{u})$. One says that $\bar{u}$ is  unstable if $\lambda_1(L,\bar{u})>0$, stable if $\lambda_1(L,\bar{u})<0$, and semistable if $\lambda_1(L,\bar{u})\leq0$.
\end{defi}

By comparison, it is immediate to see that $\lambda_1(L,\bar{u})\le\|\partial_uf_L(\cdot,\bar{u}(\cdot))\|_{L^{\infty}(\R)}$. It also follows from~\cite{bhr1} that if $\bar{u}$ is $L$-periodic, then $\lambda_1(L,\bar{u})$ is the principal eigenvalue of the periodic eigenvalue problem
\begin{equation}\label{prineigen}\left\{\baa{l}
(a_L(x)\varphi')'+\partial_uf_L(x,\bar{u}(x))\varphi=\lambda\varphi\hbox{ in }\R,\vspace{3pt}\\
\varphi>0\hbox{ in }\R,\ \ \varphi\hbox{ is }L\hbox{-periodic}.\eaa\right.
\end{equation}

The following theorem gives some information about the existence of steady states or other pulsating fronts of~\eqref{eqL} at a positive boundary point of the set~$E$. We first point out that if $L\in\partial E\cap(0,+\infty)$, then~\eqref{eqL} cannot admit a pulsating front with a nonzero speed, as a consequence of Theorem~\ref{thE1}.

\begin{theo}\label{thE2}
If $\int_0^1\overline{f}(u)du\neq0$ and $L\in\partial E\cap(0,+\infty)$, then one of the following cases occurs:
\begin{itemize}
\item [(i)] either there is $c>0$ such that equation~\eqref{eqL} admits some $L$-periodic steady states $0<\bar{u}(x),\,\bar{v}(x)<1$, with $\bar{u}$ being semistable, and some pulsating fronts
\be\label{upm}\left\{\baa{l}
\displaystyle0<\phi(x-ct,x/L)=u(t,x)<v(t,x)=\psi(x-ct,x/L)<1,\vspace{3pt}\\
\phi(\xi,y)\hbox{ and }\psi(\xi,y)\hbox{ are }1\hbox{-periodic in }y\eaa\right.
\ee
with speed $c$ and with limiting values
\be\label{upm2}\left\{\baa{ll}
\phi(-\infty,y)=\bar{u}(Ly), & \phi(+\infty,y)=0,\vspace{3pt}\\
\psi(-\infty,y)=1, & \psi(+\infty,y)=\bar{v}(Ly),\eaa\hbox{ uniformly in }y\in\R;\right.
\ee
\item [(ii)] or there is $c<0$ such that equation~\eqref{eqL} admits some $L$-periodic steady states $0<\bar{u}(x),\,\bar{v}(x)<1$, with $\bar{v}$ being semistable, and some pulsating fronts $0<u(t,x)<v(t,x)<1$, with speed $c$, satisfying~\eqref{upm} and~\eqref{upm2};
\item [(iii)] or equation~\eqref{eqL} admits a semistable $L$-periodic steady state $0<\bar{u}\le1$ and a semistable steady state $u$ such that $0<u<\bar{u}$,~$u(\cdot+L)<u$ with the limiting values
$$u(x)-\bar{u}(x)\to0\hbox{ as }x\to-\infty\ \hbox{ and }u(x)\to0\hbox{ as }x\to+\infty;$$
\item [(iv)] or equation~\eqref{eqL} admits a semistable $L$-periodic steady state $0\le\bar{v}<1$ and a semistable steady state $v$ such that $\bar{v}<v<1$,~$v(\cdot+L)<v$ with the limiting values $v(x)\to1$ as~$x\to-\infty$ and $v(x)-\bar{v}(x)\to0$ as $x\to+\infty$.
\end{itemize}
Furthermore, if $\int_0^1\overline{f}(u)du>0$, then only cases~(i) and~(iii) can occur, while only cases~(ii) and~(iv) can occur if $\int_0^1\overline{f}(u)du<0$.
\end{theo}

\begin{rem}\label{remfn}{\rm It follows in particular that if $L\in\partial E\cap(0,+\infty)$, then equation~\eqref{eqL} admits either at least one semistable $L$-periodic steady state $0<\bar{u}<1$ or a stationary front connecting~$0$ and~$1$. Cases~(i),~(ii),~(iii) and~(iv) have some similarities to the limiting behavior of homogeneous equations of the type
\be\label{eqfn}
u_t=u_{xx}+f_n(u),
\ee
where $(f_n)_{n\in\N}$ is a family of $C^1([0,1])$ functions satisfying~\eqref{asspars} uniformly in $n$ with~$f_n(0)=f_n(1)=0$, and converging uniformly in $[0,1]$ to a $C^1([0,1])$ function~$f$. On the one hand, if, for instance, each function $f_n$ has a unique zero $\theta_n$ in $(0,1)$, $\theta_n\to\theta\in(0,1)$, $\int_0^1f_n\neq0$, $f<0$ in~$(0,\theta)$, $f>0$ in $(\theta,1)$ and $\int_0^1f=0$, then each equation~\eqref{eqfn} admits a traveling front~$\phi_n(x-c_nt)$ connecting $0$ and $1$, with $c_n\neq0$, while the limiting equation $u_t=u_{xx}+f(u)$ admits a stationary front but does not admit any non-stationary front connecting~$0$ and $1$. The conclusion would then be in some sense similar to that of cases~(iii) and~(iv) in Theorem~\ref{thE2}. On the other hand, assume now that there are $0<\theta_{1,n}<\theta_{2,n}<\theta_{3,n}<1$ such that~$f_n(\theta_{1,n})=f_n(\theta_{2,n})=f_n(\theta_{3,n})=0$, $f_n<0$ in $(0,\theta_{1,n})\cup(\theta_{2,n},\theta_{3,n})$, $f_n>0$ in $(\theta_{1,n},\theta_{2,n})\cup(\theta_{3,n},1)$ and let~$c'_n$ and~$c''_n$ be the speeds of the traveling fronts of~\eqref{eqfn} connecting $0$ and $\theta_{2,n}$, and $\theta_{2,n}$ and $1$, respectively. If $c'_n<c''_n$, then~\eqref{eqfn} admits a traveling front connecting $0$ and $1$, with a speed~$c_n$ such that $c'_n<c_n<c''_n$, see~\cite{fm}. Now, if the real numbers $0<\theta_{1,n}<\theta_{2,n}<\theta_{3,n}<1$ converge to $0<\theta_1<\theta_2<\theta_3<1$, if $f<0$ in $(0,\theta_1)\cup(\theta_2,\theta_3)$, $f>0$ in $(\theta_1,\theta_2)\cup(\theta_3,1)$, and if~$c'_n$ and~$c''_n$ converge to the same real number $c$, then the limiting equation $u_t=u_{xx}+f(u)$ does not admit any traveling front connecting $0$ and $1$~\cite{fm}, but it admits some traveling fronts with speed $c$ connecting~$0$ and $\theta_2$, and $\theta_2$ and $1$, respectively. Furthermore, $\theta_2\in(0,1)$ is necessarily a semistable steady state of the limiting equation in the sense that $f'(\theta_2)\le0$. If $c>0$ or $c<0$, then the conclusion for this limiting equation is similar to that of cases~(i) or~(ii) in Theorem~\ref{thE2}.}
\end{rem}

As already emphasized, $E=(0,+\infty)$ for equation~\eqref{eqL} under some additional assumptions on the coefficients~$a$ and $f$, see~\cite{fz,fm,x2,x4,x5} and the comments after equation~\eqref{eqfz}. When~$E=(0,+\infty)$, it is interesting to investigate the effect of environmental fragmentation on the speed of pulsating fronts: from Theorem~\ref{thE1}, the map~$L\mapsto c_L$ is continuous, but can one say that it is monotone? As known in~\cite{n} for equations with periodic Fisher-KPP type nonlinearities $f$, the minimal wave speed~$c^*_L$ of pulsating fronts, which is well defined for all $L>0$, is nondecreasing with respect to the period $L>0$. The limits of $c^*_L$ as $L\to0^+$ and $L\to+\infty$ have been determined in~\cite{ehr,hfr,hnr}, and the proofs use as an essential tool a variational formula for~$c^*_L$, which only involves the derivative~$\partial_uf(\cdot,0)$ of $f$ at $u=0$. For the bistable equation~\eqref{eqL} under assumptions~\eqref{bistable} and~\eqref{asspars}, Theorems~\ref{thhomo} and~\ref{thhomobis} give the limit of $c_L$ as $L\to0^+$, but the determination of the limit, in any, of $c_L$ as $L\to+\infty$ under the assumptions of Theorem~\ref{thlarge} is still open, as is the question of the monotonicity of $c_L$ with respect to $L$ on the connected components of the set $E$.

More generally speaking, for general diffusion and reaction coefficients $a(x)$ and $f(x,u)$ satis\-fying~\eqref{bistable} and~\eqref{asspars}, the question of the existence of pulsating fronts with zero or nonzero speed~$c_L$ for the $L$-periodic equation~\eqref{eqL} is very challenging. We conjecture that the existence does not hold in general, but we leave this open question for further investigations. The possible presence of multiple ordered steady states could prevent the existence of fronts with zero or nonzero speed in general, as in the homogeneous case $f=f(u)$ (see~\cite{fm} and Remark~\ref{remfn} above), but no example has been known for functions $f$ satisfying~\eqref{bistable} and~\eqref{asspars} (apart from a non-existence result of fronts in straight inifinite cylinders with non-convex sections, see~\cite{bh1}). So far, the related ``non-existence" results have been concerned with the non-existence of pulsating fronts with nonzero speed, for some specific equations such as~\eqref{examc} (see~\cite{x4,xz}) or in the case of large periods (see~\cite{dhz2}). We mention that the existence of stationary solutions (preventing the existence of truly propagating solutions) has also been investigated for spatially discrete models~\cite{agn,bc,cmv,cs,k,mp}, for some non-periodic equations~\cite{amo,lk,n2,p} or in some higher-dimensional situations~\cite{bbc,cg}.

%%%%%%%%%%%%%%%%%%%%%%%%%%%%%%%%%%%%%%%%

\subsubsection*{Exponential stability of pulsating fronts}

The last main result of the paper is concerned with the global and exponential stability of the pulsating fronts with nonzero speed. In this section, we fix $L\in(0,\infty)$ and we assume that equation~\eqref{eqL} admits a pulsating front $\phi_L(x-c_Lt,x/L)$ with nonzero speed $c_L\neq0$. We study the asymptotic behavior as $t\to+\infty$ of the solutions of
\begin{equation}\label{inieqL}\left\{\baa{ll}
u_t=(a_L(x)u_x)_x+f_L(x,u), & t>0,\ x\in\R,\vspace{3pt}\\
u(0,x)=g(x), & x\in\R,\eaa\right.
\end{equation}
for the class of ``front-like" initial conditions $g\in L^{\infty}(\R,[0,1])$ (the initial condition $u(0,x)=g(x)$ is understood for a.e. $x\in\R$). From~\cite{fm}, it is well known that for the homogeneous equation~\eqref{homoin}, if the initial value $g$ is above $1-\delta$ at $-\infty$ and below $\delta$ at $+\infty$, where $\delta>0$ is as in~\eqref{asspars}, then the solution of associated initial value problem converges at large time to a translate of the traveling wave solution~\eqref{twhomo}, and this convergence is exponential in time. For scalar reaction-diffusion equations with bistable time-periodic nonlinearities, such exponential stability of time-periodic traveling waves was first proved in~\cite{abc}, and then a dynamical systems approach to these results was presented in \cite[Section~10]{zh}. For a special class of equations in periodic habitat with $x$-independent bistable reaction $f=f(u)$, only the local stability of pulsating fronts had been shown, see~\cite{x2} (see also~\cite{s} for some results on the local stability of fronts for time almost-periodic bistable equations). In the current paper, we show the global and exponential stability of  pulsating fronts for the more general equation~\eqref{eqL}.

\begin{theo}\label{gstability}
Assume that equation~\eqref{eqL} admits a pulsating front $u_L(t,x)=\phi_L(x-c_Lt,x/L)$ with speed $c_L\neq 0$. Then there exists a positive constant $\mu>0$ such that for every $g\in L^{\infty}(\R,[0,1])$ satisfying
\begin{equation}\label{initialv}
\liminf_{x\to-\infty}g(x)>1-\delta\ \hbox{ and }\ \limsup_{x\to+\infty} g(x)<\delta,
\end{equation}
where $\delta$ is the constant given in~\eqref{asspars}, the solution $u(t,x)$ of~\eqref{inieqL} satisfies
\be\label{convergence}
|u(t,x)-u_L(t+\tau_g,x)|=|u(t,x)-\phi_L(x-c_Lt-c_L\tau_g,x/L)  |\leq C_g\,\me^{-\mu t}\hbox{ for all }t\ge0,\ x\in\R,
\ee
for some constants $\tau_g\in\R$ and $C_g>0$.
\end{theo}

This theorem implies in particular that for the large class of initial values satisfying~\eqref{initialv}, the solutions of~\eqref{inieqL} have the same profile and the same wave speed at large time. Furthermore, Theorem~\ref{gstability} also immediately provides the uniqueness of the speed of pulsating fronts as well as the uniqueness of the pulsating fronts up to shifts in time in the case where the speed is not zero (as consequences of Theorem~\ref{gstability}, these uniqueness properties stated in Theorem~\ref{thqual} are proved at the end of Section~\ref{sec41}). More generally speaking, the global stability of pulsating fronts will be used in~\cite{dhz2} to prove a uniqueness result in the larger class of generalized transition fronts.

In Theorem~\ref{gstability}, the assumption $c_L\neq0$ is essential. Namely, there are equations of the type~\eqref{eqL} which admit stationary fronts (with speed $c_L=0$) that are not stable: in~\cite{dhz2} we construct generalized transition fronts which connect unstable stationary fronts to stable ones.

In Theorem~\ref{gstability}, the initial conditions are assumed to be close enough to $0$ and $1$ at $\pm\infty$. Actually, the convergence holds under other types of assumptions for the initial conditions, as the following result shows.

\begin{theo}\label{gstability2}
Assume that equation~\eqref{eqL} admits a pulsating front $u_L(t,x)=\phi_L(x-c_Lt,x/L)$ with speed $c_L\neq 0$ and assume that the $L$-periodic stationary states $0<\bar{u}(x)<1$ of~\eqref{eqL} are all unstable. Then for any $L$-periodic stationary states $0<\bar{u}_{\pm}(x)<1$ of~\eqref{eqL} and for any $g\in L^{\infty}(\R,[0,1])$ satisfying
\begin{equation}\label{initialv2}
\liminf_{x\to-\infty}\big(g(x)-\bar{u}_-(x)\big)>0\ \hbox{ and }\ \limsup_{x\to+\infty}\big(g(x)-\bar{u}_+(x)\big)<0,
\end{equation}
the solution $u(t,x)$ of~\eqref{inieqL} satisfies~\eqref{convergence}.
\end{theo}

Under the assumptions of Theorem~\ref{gstability2} (we thank Dr.~J.~Fang for mentioning initial conditions of the type~\eqref{initialv2}), it follows that any two $L$-periodic stationary states $0<\bar{u}_1(x)<1$ and $0<\bar{u}_2(x)<1$ of~\eqref{eqL} are either identically equal, or unordered in the sense that $\min_{\R}(\bar{u}_1-\bar{u}_2)<0<\max_{\R}(\bar{u}_1-\bar{u}_2)$. Given this fact, the proof of Theorem~\ref{gstability2} is then based on the following argument. On the one hand, if a function $v$ solves the Cauchy problem~\eqref{inieqL} with an~$L$-periodic initial condition $g\in L^{\infty}(\R,[0,1])$ such that $g>\bar{u}_1$ in $\R$ (resp. $g<\bar{u}_2$ in $\R$), then $v(t,x)\to1$ (resp.~$v(t,x)\to0$) as $t\to+\infty$ uniformly in $x\in\R$. On the other hand, the solution $u$ of~\eqref{inieqL} with an initial condition $g$ satisfying~\eqref{initialv2} can be estimated from below or above as $x\to\pm\infty$ by solutions~$v$ of the above type. Therefore, $u(T,\cdot)$ satisfies the limiting conditions~\eqref{initialv} for some time $T>0$ large enough and Theorem~\ref{gstability} can be applied to $u(T+t,x)$ and leads to the conclusion.

 However, we point out that even if Theorem~\ref{gstability} is used in the proof of Theorem~\ref{gstability2}, the assumption~\eqref{initialv2} does not imply~\eqref{initialv} in general, so Theorem~\ref{gstability2} cannot be viewed as a direct corollary of Theorem~\ref{gstability}. To see it, let us consider the homogeneous equation~\eqref{homoin} with an $x$-independent function $f=f(u)$ satisfying~\eqref{bistable}, with $\theta_x=\theta$, together with~\eqref{asspars} and~$\int_0^1f(u)\,du>0$. The  assumptions of Theorems~\ref{gstability} and~\ref{gstability2} are fulfilled, and any initial condition $g\in L^{\infty}(\R,[0,1])$ such that $\liminf_{x\to-\infty}g(x)>\theta$ and $\limsup_{x\to+\infty}g(x)<\theta$ satisfies~\eqref{initialv2} but not~\eqref{initialv} in general.

In Theorems~\ref{gstability} and~\ref{gstability2}, the initial conditions $g$ are front-like in the sense that $g$ is not too small at $-\infty$ and not too large at $+\infty$. We mention that the Cauchy problem with initial conditions which are compactly supported or at least not too large at $\pm\infty$ has been extensively studied in the homogeneous and periodic cases, see, e.g., \cite{aw,dm,fm,g,x4,zl}.

\hfill\break
\noindent{\bf{Outline of the paper.}} Section~\ref{sec2} is devoted to the proof of the existence results for small and large periods $L$, that is, Theorems~\ref{thhomo},~\ref{thhomobis} and~\ref{thlarge}. We also show part of Theorem~\ref{thqual} on the sign property of the speed of non-stationary pulsating fronts. In Section~\ref{sec3}, we prove Theorems~\ref{thE1} and~\ref{thE2} on the properties of the set of periods of non-stationary fronts. Lastly, Section~\ref{sec4} is devoted to the proof of the stability results, Theorems~\ref{gstability} and~\ref{gstability2}, while the Appendix (Section~\ref{sec5}) is concerned with the proof of some auxiliary lemmas used in the proofs of Theorems~\ref{thhomo} and~\ref{thE1}.

%%%%%%%%%%%%%%%%%%%%%%%%%%%%%%%%%%%%%%%%
%%%%%%%%%%%%%%%%%%%%%%%%%%%%%%%%%%%%%%%%

\SE{Existence of pulsating fronts}\label{sec2}

As already emphasized, the proofs of Theorems~\ref{thhomo} and~\ref{thhomobis} on the existence of pulsating fronts for small periods~$L$ use different methods. They are carried out in Sections~\ref{smalllhomo} and~\ref{smalllhomobis}, whereas Section~\ref{seclarge} is devoted to the proof of Theorem~\ref{thlarge} on the existence of pulsating fronts for large periods~$L$. In Section~\ref{secsign}, we show that the sign of non-stationary pulsating fronts is that of the integral $\int_0^1\overline{f}(u)\,du$.

%%%%%%%%%%%%%%%%%%%%%%%%%%%%%%%%%%%%%%%%

\subsection{Small periods $L$: the implicit function theorem}\label{smalllhomo}

This section is devoted to the proof of Theorem~\ref{thhomo}. The strategy is similar to that used by Heinze in~\cite{he2}. There, the homogeneous process for the following equation as $\varepsilon\to 0^+$ was considered
\begin{equation}\label{heihomo}
\left\{\baa{rcll}
u_t & = & d\,\Delta u+f(u), & t\in\R,\,\, x\in\Omega_{\varepsilon},\vspace{3pt}\\
u_{\nu} & = & 0, & t\in\R,\,\, x\in\partial \Omega_{\varepsilon},\eaa\right.
\end{equation}
where $d$ is a positive constant, $f:\R\to\R$ is a bistable function of class $C^2$, $u_{\nu}=\frac{\partial u}{\partial\nu}$, $\nu$ denotes the exterior normal vector to $\Omega_{\epsilon}$ and $\Omega_{\varepsilon}=\varepsilon\,\Omega$, with $\Omega$ being a smooth open connected set of~$\R^n$ which is 1-periodic in all directions $x_i$ ($1\le i\le n$). Assuming the existence of a traveling wave for a homogenized problem and then using the implicit function theorem in an appropriate function space, Heinze obtained a unique (up to shift in time) pulsating front for the equation~\eqref{heihomo} at small~$\varepsilon$. Although a portion of the arguments of Theorem~\ref{thhomo} follows the same lines as those used in~\cite{he2} for problem~\eqref{heihomo}, the oscillations in the diffusion coefficient $a_L$ and the nonlinearity~$f_L$ in equation~\eqref{eqL} make the analysis different and more complicated. In addition, we provide a different approach in the convergence to the homogeneous process (see Lemma~\ref{continuem} below). As a matter of fact, the approach we use here allows us to prove some continuity results at any $L>0$ and not only at $L=0^+$ (see Lemma~\ref{continueg} below). This strategy can then later be applied to the proof of Theorem~\ref{thE1} in Section~\ref{sec3}, where a pulsating front for equation~\eqref{eqL} is assumed to exist at a fixed $L_0>0$. Thus, for the sake of completeness of the proof here and also for convenience of that of Theorem~\ref{thE1}, we include the details as follows.

We assume that the homogenized equation~\eqref{homogenized} has a (unique up to shifts) front $\phi_0$, with speed~$c_0\neq0$. Without loss of generality, one may assume that the speed satisfies $c_0>0$ throughout this subsection. Indeed, if $c_0<0$, then the function $\psi_0(x)=1-\phi_0(-x)$ solves~\eqref{homogenized} with speed~$-c_0\,(>0)$ instead of $c_0$ and with $\overline{g}(u)=-\overline{f}(1-u)$ instead of $\overline{f}$, where $g(x,u)=-f(-x,1-u)$. Furthermore, if $u(t,x)=\phi(x-c_Lt,x/L)$ is a pulsating front for~\eqref{eqL} with speed~$c_L$, then $v(t,x)=1-\phi(-x-c_Lt,-x/L)$ is a pulsating front with speed $-c_L$ for the equation $v_t=(\tilde{a}_Lv_x)_x+g_L(x,v)$ with $\tilde{a}(x)=a(-x)$ (and $\tilde{a}_H=a_H$). Therefore, even if it means changing~$a$ into~$\tilde{a}$ and $f$ into $g$, one can assume here that $c_0>0$.

Define the new variables
\begin{equation}\label{changev}
\xi=x-c_Lt\quad\hbox{and}\quad y=\frac{x}{L}.
\end{equation}
For a given $L>0$, finding pulsating fronts $u(t,x)=\phi_L(x-c_Lt,x/L)$ of~\eqref{eqL} with a speed $c_L\neq0$ amounts to finding solutions~$\phi_{L}$ of the following problem:
\begin{equation}\label{pulsolu}\left\{\baa{l}
\tilde{\partial}(a(y)\tilde{\partial}_L\phi_{L})+c_L\partial_{\xi}\phi_{L}+f(y,\phi_{L})=0 \hbox{ for all }(\xi,y)\in\R\times\R,\vspace{3pt}\\
\phi_{L}(\xi,y)\hbox{ is } 1\hbox{-periodic in }y,\vspace{3pt}\\
\phi_{L}(-\infty,y)=1,\ \phi_{L}(+\infty,y)=0\hbox{ uniformly in }y\in\R,\eaa\right.
\end{equation}
where
$$\tilde{\partial}_L=\partial_{\xi}+\frac{1}{L}\partial_y.$$
By the periodicity condition, equation~\eqref{pulsolu} can be restricted in $y\in\mathbb{T}:=\R/\Z$. Let $L^2(\R\times \T)$ and $H^1(\R\times \T)$ be the Banach spaces defined by
$$\left\{\baa{rcl}
L^2(\R\times \T) & = & \big\{v\in L^2_{loc}(\R\times \R)\,\big|\, v\in L^2(\R\times (0,1)) \,\,\hbox{and}\,\, v(\xi,y+1)=v(\xi,y)\,\,\hbox{a.e.\,\,in}\,\, \R^2 \big\},\vspace{3pt}\\
H^1(\R\times \T) & = & \big\{v\in H^1_{loc}(\R\times \R)\,\big|\, v\in H^1(\R\times (0,1)) \,\,\hbox{and}\,\, v(\xi,y+1)=v(\xi,y)\,\,\hbox{a.e.\,\,in}\,\, \R^2 \big\},\eaa\right.$$
embedded with the norms $\| v\|_{L^2(\R\times \T)}=\| v\|_{L^2(\R\times (0,1))}$ and $\| v\|_{H^1(\R\times \T)}=\| v\|_{H^1(\R\times (0,1))}=\big(\|v\|_{L^2(\R\times\T)}+\|\partial_{\xi}v\|_{L^2(\R\times\T)}+\|\partial_yv\|_{L^2(\R\times\T)}\big)^{1/2}$, respectively.

For the homogenization limit as $L\to0^+$, we introduce some auxiliary operators. Namely, fix a real $\beta>0$ and for any $c>0$, define
\begin{equation}\label{linhomo}\baa{rcll}
\!\!\!\!M_{c,L}(v) & \!\!\!\!=\!\!\!\! & \tilde{\partial}_L(a\tilde{\partial}_Lv)\!+\!c\partial_{\xi}v\!-\!\beta v, & \!\!v\in\mathcal{D}_L\!=\!\{v\in H^1(\R\!\times\!\T) \, |\, \tilde{\partial}_L(a\tilde{\partial}_Lv)\!\in\!L^2(\R\!\times\!\T)\},\,L\!\in\!\R^*,\vspace{3pt}\\
\!\!\!\!M_{c,0}(v) & \!\!\!\!=\!\!\!\! & a_Hv''+cv'-\beta v, & \!\!v\in D(M_{c,0})=H^2(\R),\eaa
\end{equation}
where $a_H>0$ is the harmonic mean of $a$ defined in~\eqref{defaH}. The fact that $v\in\mathcal{D}_L$ means that $v\in H^1(\R\times\T)$ and there is a constant $C>0$ such that~$\big|\int_{\R\times\T}a\tilde{\partial}_Lv\tilde{\partial}_L\varphi\big|\!\le\!C\|\varphi\|_{L^2(\R\times\T)}$ and $\int_{\R\times\T}a\tilde{\partial}_Lv\tilde{\partial}_L\varphi\!=\!-\int_{\R\times\T}\tilde{\partial}_L(a\tilde{\partial}_Lv)\varphi$ for all $\varphi\in H^1(\R\times\T)$. Notice that the operators $M_{c,L}$ are also defined for negative values of $L$, that the domain $\mathcal{D}_L$ is dense in~$H^1(\R\times \T)$ and that $M_{c,L}(v)\in L^2(\R\times \T)$ for all $v\in \mathcal{D}_L$ and $L\in\R^*$. Furthermore, the domain~$D(M_{c,0})=H^2(\R)$ of~$M_{c,0}$ is dense in $H^1(\R)$ and $M_{c,0}(v)\in L^2(\R)$ for all $v\in D(M_{c,0})$. We first state in the following three lemmas some of the basic properties of the operators~$M_{c,L}$,~$M_{c,0}$ and their inverses.

\begin{lem}\label{invertible}
The operators $M_{c,0}: H^2(\R)\, \rightarrow \, L^2(\R)$ and $M_{c,L}: \mathcal{D}_L\, \rightarrow \, L^2(\R\times \T)$ for $L\neq0$ are invertible for every $c>0$. Furthermore, for every $r_1> 0$ and $r_2>0$, there is a constant $C=C(r_1,r_2,\beta,a)$ such that for all $c\geq r_1$, $|L|\leq r_2$, $g\in L^2(\R\times\T)$ and $\varphi\in L^2(\R)$,
$$\left\{\baa{lcll}
\|M_{c,L}^{-1}(g)\|_{H^1(\R\times\T)} & \leq & C\|g\|_{L^2(\R\times\T)} & \hbox{if }L\neq 0,\vspace{3pt}\\
\|M_{c,0}^{-1}(\varphi)\|_{H^1(\R)} & \leq & C\|\varphi\|_{L^2(\R)}.\eaa\right.$$
\end{lem}

The following lemma deals with the convergence of $M_{c_n,L_n}^{-1}$ to $M_{c,0}^{-1}$ when $L_n\to0$ with $L_n\in\R^*$, $c_n\to c>0$ and the operators $M_{c_n,L}^{-1}$ are applied to $g_n$ with $g_n\to g$ in $L^2(\R\times\T)$.

\begin{lem}\label{continuem}
Fix $\beta>0$ and $c>0$. Then for every $g\in L^2(\R\times\T)$,
\be\label{convinverse}\left\{\baa{lcll}
M_{c_n,L_n}^{-1}(g_n) & \to & M_{c,0}^{-1}(\overline{g}) & \hbox{in }H^1(\R\times\T),\vspace{3pt}\\
M_{c_n,0}^{-1}(\varphi_n) & \to & M_{c,0}^{-1}(\overline{g}) & \hbox{in }H^1(\R)\eaa\right.
\ee
as $n\to+\infty$ for every sequences $(g_n)_{n\in\N}$ in $L^2(\R\times\T)$, $(\varphi_n)_{n\in\N}$ in $L^2(\R)$, $(c_n)_{n\in\N}$ in $(0,+\infty)$ and~$(L_n)_{n\in\N}$ in $\R^*$ such that $\|g_n-g\|_{L^2(\R\times\T)}\to0$, $\|\varphi_n-\overline{g}\|_{L^2(\R)}\to0$, $c_n\to c$ and $L_n\to0$ as~$n\to+\infty$, where $\overline{g}\in L^2(\R)$ is defined as
\begin{equation}\label{ovlineg}
\overline{g}(\xi)=\int_{\T}g(\xi,y)\, dy\,\,\,\,\, \hbox{for } \xi\in \R,
\end{equation}
and $M_{c,0}^{-1}(\overline{g})\in H^2(\R)$ is viewed as an $H^2(\R\times\T)$ function by extending it trivially in the $y$-variable. Furthermore, the limits~\eqref{convinverse} are uniform in the ball $B_A=\big\{g\in L^2(\R\times\T)\,|\,\|g\|_{L^2(\R\times\T)}\le A\big\}$ for every $A>0$.
\end{lem}

The following lemma is the analogue of Lemma~\ref{continuem}, when $L_n\to L\neq0$ as $n\to+\infty$.

\begin{lem}\label{continuem2}
Fix $\beta>0$, $c>0$ and $L\in\R^*$. Then for every $g\in L^2(\R\times\T)$ and every sequences $(g_n)_{n\in\N}$ in $L^2(\R\times\T)$, $(c_n)_{n\in\N}$ in $(0,+\infty)$ and~$(L_n)_{n\in\N}$ in $\R^*$ such that $\|g_n-g\|_{L^2(\R\times\T)}\to0$, $c_n\to c$ and $L_n\to L$ as $n\to+\infty$, there holds $M_{c_n,L_n}^{-1}(g_n)\to M_{c,L}^{-1}(g)$ in $H^1(\R\times\T)$ as $n\to+\infty$. Furthermore, the limit is uniform with respect to $g\in B_A$ for every $A>0$.
\end{lem}

In order not to lengthen the course of the proof of Theorem~\ref{thhomo}, the proofs of these auxiliary lemmas are postponed in the Appendix (Section~\ref{sec5}).

Coming back to the solution $(\phi_0,c_0)$ of the homogeneous equation~\eqref{homogenized}, it is well known that there are some positive constants~$A_1$ and~$A_2$ such that $\phi_0(\xi)\sim A_1\me^{-\lambda_1\xi}$ as $\xi\to+\infty$ and $1-\phi_0(\xi)\sim A_2\me^{\lambda_2\xi}$ as $\xi\to-\infty$, with $\lambda_1=(c_0+(c_0^2-4a_H\overline{f}'(0))^{1/2})/(2a_H)>0$ and~$\lambda_2=(-c_0+(c_0^2-4a_H\overline{f}'(1))^{1/2})/(2a_H)>0$. Now, in order to apply an implicit function theorem around the homogeneous front $(\phi_0,c_0)$ to get the existence of pulsating fronts for small $L$, we need to introduce a few more notations. Firstly, let $\chi\in C^2(\R)$ be a solution of the equation
\begin{equation}\label{dualcell}
(a(\chi'+1))'=0\hbox{ in }\R\ \hbox{ and }\ \chi(y+1)=\chi(y)\hbox{for all } y\in \R.
\end{equation}
The function $\chi$ is unique up to a constant, and it satisfies $a(y)(\chi'(y)+1)=a_H$ for all $y\in\R$. Secondly, for any $v\in H^1(\R\times\T)$, $c>0$ and $L\in\R^*$, define
\begin{equation*}
\begin{aligned}
K(v,c,L)(\xi,y)=&\big(a_H+(a\chi)'(y)\big)\phi_{0}''(\xi)+La(y)\chi(y)\phi_{0}^{(3)}(\xi)+c\big(\phi_{0}'(\xi)+L\chi(y)\phi_{0}''(\xi)\big)\vspace{3pt}\\
&+f\big(y,v(\xi,y)+\phi_{0}(\xi)+L\chi(y)\phi_{0}'(\xi)\big)+\beta v(\xi,y).
\end{aligned}
\end{equation*}
Since $\phi_0\in L^2(0,+\infty)$, $1-\phi_0\in L^2(-\infty,0)$, $\phi_0'\in H^2(\R)$ and $f(y,u)$ is globally Lipschitz-continuous in $u$ uniformly for $y\in\T$, it follows that $K(v,c,L)\in L^2(\R\times\T)$ for any $v\in H^1(\R\times\T)$ (and even for any $v\in L^2(\R\times\T)$). For $L=0$, define
$$K(v,c,0)(\xi,y)=\big(a_H+(a\chi)'(y)\big)\phi_{0}''(\xi)+c\phi_{0}'(\xi)+f\big(y,v(\xi,y)+\phi_{0}(\xi)\big)+\beta v(\xi,y).$$
Similarly, $K(v,c,0)\in L^2(\R\times\T)$ for any $v\in H^1(\R\times\T)$. Finally, for $v\in H^1(\R\times\T)$, $c>0$ and~$L\in\R$, let us set $G(v,c,L)=(G_1,G_2)(v,c,L)$ with
\be\label{defG}\left\{\baa{rcl}
G_1(v,c,L) & = & \left\{\baa{ll}
v+M_{c,L}^{-1}\big(K(v,c,L)\big) & \hbox{if }L\neq0,\vspace{3pt}\\
v+M_{c,0}^{-1}\big(\overline{K(v,c,0)}\big) & \hbox{if }L=0,\eaa\right.\vspace{3pt}\\
G_2(v,c,L) & = & \displaystyle\int_{\R^+\times \T}\big(\phi_{0}(\xi)+v(\xi,y)+L\chi(y)\phi_{0}'(\xi)\big)^2-\int_{\R^+}\phi_0^2,\eaa\right.
\ee
where $\overline{K(v,c,0)}\in L^2(\R)$ is defined as $\overline{g}$ in~\eqref{ovlineg}.  In view of Lemma~\ref{invertible}, the function $G$ maps $H^1(\R\times\T)\times(0,+\infty)\times\R$ into $H^1(\R\times\T)\times\R$. Note that $G(0,c_0,0)=(0,0)$ by the definition of $(\phi_0,c_0)$ and~$\overline{f}$. Moreover, it is also straightforward to  check, using in particular~\eqref{dualcell} and $a\,(\chi'+1)=a_H$, that a pair~$(\phi_L,c_L)\in\big(\phi_0+H^1(\R\times\T)\big)\times(0,+\infty)$ solves the first two lines of~\eqref{pulsolu} for $L\neq 0$ with
\begin{equation}\label{normizep}
\int_{\R^+\times \T} \phi_L^2=\int_{\R^+} \phi_0^2,
\end{equation}
if and only if $G(\phi_L-\phi_0-L\chi\phi_0',c_L,L)=(0,0)$.

The general strategy of the proof of Theorem~\ref{thhomo} is to use the implicit function theorem for the existence and uniqueness of a solution of~\eqref{pulsolu} and~\eqref{normizep} for small $L$. For this, we use some continuity and differentiability properties of $G$.

\begin{lem}\label{continueg}
The function $G:H^1(\R\times\T)\times(0,+\infty)\times\R\to H^1(\R\times\T)\times\R$ is continuous, and it is continuously Fr\'echet differentiable with respect to $(v,c)$. Furthermore, the operator $Q=\partial_{(v,c)}G(0,c_0,0):H^1(\R\times\T)\times \R\to H^1(\R\times\T)\times \R$ is invertible.
\end{lem}

The proof of Lemma~\ref{continueg} is quite lengthy and is therefore postponed in the Appendix (Section~\ref{sec5}). We just point out here that the proof of the invertibility of the operator $Q$ uses as key-points some properties of the linearization of~\eqref{homogenized} at $\phi_0$. Namely, denoting
\be\label{defH}
H(u)= a_Hu''+c_0u'+ \overline{f}'(\phi_{0})u=M_{c_0,0}(u)+\beta\,u+\overline{f}'(\phi_0)\,u\ \hbox{ for }u\in H^2(\R)
\ee
and the adjoint operator $H^*$ being given by $H^*(u)=a_Hu''-c_0u'+ \overline{f}'(\phi_{0})u$ for $u\in H^2(\R)$ in such a way that $(H^*(v),u)_{L^2(\R)}=(v,H(u))_{L^2(\R)}$ for all $u,v\in H^2(\R)$, it follows from Section 4 of~\cite{he2} that the operators $H$ and $H^*$ have algebraically simple eigenvalue 0 and that the range of $H$ is closed in~$L^2(\R)$. Furthermore, the kernel $\ker(H)$ of $H$ is equal to $\ker(H)=\R\phi_0'$. We will see in the proof of Theorem~\ref{thE1} that similar properties hold for the linearization of the equation~\eqref{pulsolu} around a pulsating front for a period $L_0>0$.

\begin{proof}[End of the proof of Theorem~\ref{thhomo}]
From Lemma~\ref{continueg}, one can apply the implicit function theorem for the function $G: H^1(\R\times\T) \times (0,+\infty)\times\R\to H^1(\R\times\T) \times\R$ (see, e.g., the remark of Theorem~15.1 in~\cite{d}). Then there exists $L_*>0$ such that for any $0<L<L_*$, there is a unique $(v_L,c_L)\in H^1(\R\times\T)\times (0,+\infty)$ such that $G(v_L,c_L,L)=(0,0)$ and $(v_L,c_L)\to (0,c_0)$ as $L\to 0$. Let $\phi_L(\xi,y)=\phi_0(\xi)+v_L(\xi,y)+L\chi(y)\phi_0'(\xi)$ for $(\xi,y)\in\R\times \T$. It follows in particular that $\phi_L-\phi_0\to0$ in $H^1(\R\times\T)$ as $L\to0^+$. According to the definition of~$G$, for every~$L\in(0,L_*)$,~$(\phi_L,c_L)$ is a weak solution and then, by parabolic regularity, a bounded classical solution of the equation~\eqref{pulsolu} which satisfies, in particular, the limiting conditions in~\eqref{pulsolu} since $\partial_{\xi}\phi_L$ and $\partial_y\phi_L$ actually belong to $L^{\infty}(\R\times\R)$. The strong maximum principle together with~$f(y,u)>0$ (resp. $f(y,u)<0$) for all $(y,u)\in\R\times(-\infty,0)$ (resp. $(y,u)\in\R\times(1,+\infty)$) implies that $\phi_L$ ranges in $(0,1)$.

Lastly, for any given $L\in(0,L_*)$, if $\tilde{u}_L(t,x)=\tilde{\phi}_L(x-\tilde{c}_Lt,x/L)$ is a pulsating front for~\eqref{eqL}, then $\tilde{c}_L=c_L>0$ by Theorem~\ref{thqual} (whose proof is independent of the present one), while $\tilde{v}_L(\xi,y):=\tilde{\phi}_L(\xi,y)-\phi_0(\xi)-L\chi(y)\phi'_0(\xi)\in H^1(\R\times\T)$ from the general exponential decay estimates of $\tilde{\phi}_L$  and $1-\tilde{\phi}_L$ as $\xi\to\pm\infty$ (see Lemma~\ref{exponential} below). By continuity, there is $\tilde{\xi}\in\R$ such that $\int_{\R^+\times\T}\tilde{\phi}_L^2(\xi+\tilde{\xi},y)=\int_{\R^+}\phi_0^2$, whence $G(\tilde{v}_L(\cdot+\tilde{\xi},\cdot),c_L,L)=(0,0)$ and $\tilde{\phi}_L(\xi+\tilde{\xi},y)=\phi_L(\xi,y)$ for all $(\xi,y)\in\R\times\T$ by uniqueness of $v_L$. The proof of Theorem~\ref{thhomo} is thereby complete.
\end{proof}

%%%%%%%%%%%%%%%%%%%%%%%%%%%%%%%%%%%%%%%%

\subsection{The sign of the speeds of non-stationary pulsating fronts}\label{secsign}

The notations and tools used in the previous section enable us to prove the part of Theorem~\ref{thqual} which is concerned with the sign of the speed of non-stationary pulsating fronts. We state this result (see Lemma~\ref{sign} below) as an independent lemma, since it will be used later and is also of interest in its own. Before doing so, we first establish some exponential bounds for the pulsating fronts when they approach their limiting states, whether they be stationary or non-stationary. The proof is similar to that of the exponential decay of pulsating fronts for combustion nonlinearities in~\cite{x3}. We include it here because its strategy is useful.

\begin{lem}\label{exponential}
For a fixed $L>0$, if $u(t,x)=\phi_L(x-c_Lt,x/L)$ is a pulsating front of equation~\eqref{eqL} with speed $c_L\in\R$, then there exist $A_1$, $A_2\in\R$, $\mu_1>0$, $\mu_2>0$, $C_1>0$, $C_2>0$ such that
\begin{equation}\label{exponential1}
0<u(t,x)=\phi_L(x-c_Lt,x/L)\leq C_1\me^{-\mu_1(x-c_Lt)}\ \hbox{ if }x-c_Lt \geq A_1,
\end{equation}
\begin{equation}\label{exponential2}
0<1-u(t,x)=1-\phi_L(x-c_Lt,x/L)\leq C_2\me^{\mu_2(x-c_Lt)}\ \hbox{ if }x-c_Lt \leq A_2.
\end{equation}
\end{lem}

\begin{proof}
From the strong parabolic maximum principle, we know that $0<u(t,x)<1$ for all $(t,x)\in\R^2$. According to the sign of $c_L$, two cases may occur.

{\it Case 1:} $c_L=0$. In this case, $u(t,x)=u(x)$ is a stationary solution of equation~\eqref{eqL} with limiting conditions $\lim_{x\to-\infty}u(x)=1$ and $\lim_{x\to+\infty}u(x)=0$. Then there exists $A_1\in\R$ such that~$u(x)\leq \delta$ for all $x\geq A_1$, where $\delta>0$ is as in~\eqref{asspars}. It follows from~\eqref{asspars} that $(a_Lu')'(x)-\gamma u(x) \geq 0$ for all~$x\geq A_1$. Choose~$\mu_1>0$ small enough so that $a_L(x)\mu_1^2-a_L'(x)\mu_1-\gamma\leq 0$ for all $x\in\R$, and define $\omega(x)=C_1\me^{-\mu_1x}$, where $C_1$ is a positive constant such that $C_1\me^{-\mu_1A_1}\geq u(A_1)$. But $(a_L\omega')'(x)-\gamma\omega(x)=C_1(a_L(x)\mu_1^2-a_L'(x)\mu_1-\gamma)\me^{-\mu_1x}\leq 0$ for all $x\in\R$. Then the inequality~\eqref{exponential1} follows from the elliptic weak maximum principle, and~\eqref{exponential2} can be obtained in the same way, using this time the second line of~\eqref{asspars}.

{\it Case 2:} $c_L\neq0$. In this case, consider the variables $(\xi,y)$ as in~\eqref{changev}. Upon substitution,~$\phi_L(\xi,y)$ satisfies equation~\eqref{pulsolu}. For any $\mu\in\R$, let $\mathcal{T}_{\mu}$  be the linear operator defined by
$$\mathcal{T}_{\mu}(\psi)=L^{-2}(a\psi')'-2L^{-1}\mu a\psi'+\big(-L^{-1}\mu a'-c_L\mu+a\mu^2-\gamma\big)\psi$$
for all $\psi\in C^2(\T)$. The Krein-Rutman theory provides the existence and uniqueness, for any $\mu\in\R$, of the principal eigenvalue $\lambda_1(\mu)$ of $\mathcal{T}_{\mu}$, associated with a (unique up to multiplication) positive eigenfunction $\psi_{\mu}$. By uniqueness, $\lambda_1(0)=-\gamma<0$. Moreover, there are $\alpha>0$ and $\beta>0$ such that $-\mu\,a'(y)/L-c_L\mu+a(y)\mu^2-\gamma\ge\alpha\mu^2-\beta$ for all $\mu\in\R$ and $y\in\R$, whence $\lambda_1(\mu)\ge\alpha\mu^2-\beta$ and $\lambda_1(\mu)>0$ for $|\mu|$ large enough. Since $\mu\mapsto\lambda_1(\mu)$ is continuous, there is $\mu_1>0$ such that $\lambda_1(\mu_1)=0$. A direct computation shows that
\be\label{eqpsimu1}
\tilde{\partial}_L\big(a(y)\tilde{\partial}_L(\me^{-\mu_1\xi}\psi_{\mu_1}(y))\big)+c_L\partial_{\xi}\big(\me^{-\mu_1\xi}\psi_{\mu_1}(y)\big)-\gamma\me^{-\mu_1\xi}\psi_{\mu_1}(y)=0\ \hbox{ for all }(\xi,y)\in\R\times\T,
\ee
that is, $\overline{u}_t-(a_L(x)\overline{u}_x)_x+\gamma\overline{u}=0$ with $\overline{u}(t,x)=\me^{-\mu_1(x-c_Lt)}\psi_{\mu_1}(x/L)$. On the other hand, since~$\lim_{\xi\to +\infty}\phi_L(\xi,y)=0$ uniformly for $y\in\T$, there is $A_1\in\R$ such that $0<\phi_L(\xi,y)\leq \delta$ for all~$\xi\geq A_1$ and $y\in \T$. It then follows from~\eqref{asspars} that
\be\label{phiL}
\tilde{\partial}_L(a(y)\tilde{\partial}_L\phi_L)+c_L\partial_{\xi}\phi_L-\gamma\phi_L\geq 0\ \hbox{ for all }\xi\geq A_1\hbox{ and }y\in\T.
\ee
One can normalize $\psi_{\mu_1}$ in such a way that $ \me^{-\mu_1A_1}\psi_{\mu_1}(y)>\phi_L(A_1,y)$ for all $y\in\T$. Define $\epsilon^*=\inf\big\{\epsilon\ge0\ |\ u(t,x)-\epsilon\le\overline{u}(t,x)$ for all $x-c_Lt\ge A_1\big\}$. The real number $\epsilon^*$ is well defined and $u(t,x)-\epsilon^*\le\overline{u}(t,x)$ for all $x-c_Lt\ge A_1$. Assume by contradiction that $\epsilon^*>0$. Since~$u(t+L/c_L,x+L)=u(t,x)$ in~$\R^2$ and $u(t,x)\to0$ as $x-c_Lt\to+\infty$, there is then~$(t^*,x^*)\!\in\!\R^2$ such that $x^*-c_Lt^*\ge A_1$ and $u(t^*,x^*)-\epsilon^*=\overline{u}(t^*,x^*)$, whence $x^*-c_Lt^*>A_1$ from the normalization of~$\psi_{\mu_1}$. Define $z=\overline{u}-u$. From~\eqref{eqpsimu1} and~\eqref{phiL}, one has $0\le z_t-(a_L(x)z_x)_x+\gamma z$ for all $x-c_Lt\ge A_1$. But $z$ has a minimum at the point $(t^*,x^*)$ with $x^*-c_Lt^*>A_1$ and~$z(t^*,x^*)=-\epsilon^*<0$. Hence $0\le-\gamma\,\epsilon^*$, which is a contradiction. Thus, $\epsilon^*=0$, that is, ~\eqref{exponential1} holds $C_1=\max_{\T}\psi_{\mu_1}>0$.

Similarly, there is $A_2\in\R$ such that $0<1-\phi_L(\xi,y)\leq \delta$ for all $\xi\leq A_2$ and $y\in \T$. As above, by working this time with the function $1-\phi_L(\xi,y)$ on the set~$(-\infty,A_2)\times \T$,~\eqref{exponential2} follows.
\end{proof}

\begin{lem}\label{sign}
If equation~\eqref{eqL} admits a pulsating front $u(t,x)=\phi_L(x-c_Lt,x/L)$ with $c_L\neq 0$, then~$c_L$ has the sign of $\int_0^1\overline{f}(s)ds$.
\end{lem}

\begin{proof}
Use the variables $(\xi,y)$ as in~\eqref{changev} and notice that, since $c_L\neq0$, standard parabolic estimates applied to $u$ and $u_t$ imply that the function $\phi_L(\xi,y)$ is a classical solution of the equation~\eqref{pulsolu} and, by Lemma~\ref{exponential}, all functions $\partial_{\xi}\phi_L$, $\partial_y\phi_L$, $\partial_{\xi\xi}\phi_L$, $\partial_{\xi y}\phi_L$ and $\partial_{yy}\phi_L$ converge to~$0$ exponentially as $\xi\to \pm\infty$. Integrating by parts the first equation of~\eqref{pulsolu} against $\partial_{\xi}\phi_L$ yields
$$\int_{\R\times \T} a(y)\tilde{\partial}_L\phi_L\tilde{\partial}_L(\partial_{\xi}\phi_L)-c_L(\partial_{\xi}\phi_L)^2-f(y,\phi_L)\partial_{\xi}\phi_L=0.$$
Since $\int_{\R\times \T} a(y)\tilde{\partial}_L\phi_L\tilde{\partial}_L(\partial_{\xi}\phi_L)=(1/2)\int_{\R\times \T} a(y)\partial_{\xi}\big(\tilde{\partial}_L\phi_L\big)^2=0$, one infers that
$$c_L\int_{\R\times \T}(\partial_{\xi}\phi_L)^2=-\int_{\R\times\T}f(y,\phi_L)\partial_{\xi}\phi_L=-\int_{\R\times \T}  \partial_{\xi}\Big(\int_0^{\phi_L(\xi,y)} f(y,s)ds\Big)=\int_{0}^1\overline{f}(s)ds.$$
Hence, $c_L$ has the sign of $\int_{0}^1\overline{f}(s)ds$.
\end{proof}

%%%%%%%%%%%%%%%%%%%%%%%%%%%%%%%%%%%%%%%%

\subsection{Small periods $L$: the instability of $L$-periodic steady states}\label{smalllhomobis}

In this section, we do the proof of Theorem~\ref{thhomobis}. To obtain the existence result, we employ the abstract theory in~\cite{fz} by checking that the semiflow generated by the equation~\eqref{inieqL} satisfies the general assumptions for bistable monotone semiflows in that paper. Namely, we first prove a series of lemmas to verify the general conditions in~\cite{fz}.

To do so, we first recall that for any~$L~\!\!>~\!\!0$ and any $L$-periodic steady state $\bar{u}:\R\to[0,1]$ of~\eqref{eqL},~$\lambda_1(L,\bar{u})$ denotes the principal eigenvalue of the problem~\eqref{prineigen} and that, by Definition~\ref{semistable},~$\bar{u}$ is unstable if $\lambda_1(L,\bar{u})>0$. By comparison, one has $\min_{x\in\R}\big(\partial_uf_L(x,\bar{u}(x))\big)\leq \lambda_1(L,\bar{u}) \leq \max_{x\in\R}\big(\partial_uf_L(x,\bar{u}(x))\big)$. In particular, the steady state~$\bar{u}~\!\!\equiv~\!\!0$ is stable since $\partial_uf(x,0)\le-\gamma$. Similarly, the steady state $\bar{u}\equiv1$ is stable. Then the Dancer-Hess connecting orbit theorem (see, e.g., \cite[Proposition 9.1]{he}) implies that there exists at least one~$L$-periodic steady state $\bar{u}$ such that $0 < \bar{u} < 1$ in $\R$. In addition, it turns out that for small $L>0$, all such intermediate $L$-periodic steady states are unstable under the assumption~\eqref{hypoverf}, as the following lemma shows.

\begin{lem}\label{unstable}
Under the assumption~\eqref{hypoverf}, there is $\tilde{L}_*>0$ such that $\lambda_1(L,\bar{u})>0$ for every $0<L<\tilde{L}_*$ and for every $L$-periodic steady state $\bar{u}$ of~\eqref{eqL} with $0 < \bar{u} < 1$.
\end{lem}

\begin{proof}
Assume by contradiction that there are some sequences $(L_n)_{n\in\N}$ in $(0,+\infty)$, $(\bar{u}_n)_{n\in\N}$ and~$(\psi_n)_{n\in\N}$ in $C^2(\R)$ such that $L_n\to 0^+$ as $n\to+\infty$ and, for each $n\in\N$,~$\bar{u}_n$ satisfies
\begin{equation}\label{equn}
\left\{\baa{l}
(a_{L_n}\bar{u}_n')'+f_{L_n}(x,\bar{u}_n)=0\ \hbox{ in }\R,\vspace{3pt}\\
\bar{u}_n\hbox{ is }L_n\hbox{-periodic, }\ 0<\bar{u}_n <1\hbox{ in }\R,\eaa\right.
\end{equation}
and $\psi_n$ satisfies
\begin{equation}\label{eqpsin}
\left\{\baa{l}
(a_{L_n}\psi_n')'+\partial_uf_{L_n}(x,\bar{u}_n)\psi_n=\lambda_1(L_n,\bar{u}_n)\psi_n\ \hbox{ in }\R,\vspace{3pt}\\
\psi_n\hbox{ is }L_n\hbox{-periodic},\ \ \psi_n>0\hbox{ in }\R,\eaa\right.
\end{equation}
with principal eigenvalue $\lambda_1(L_n,\bar{u}_n)\leq 0$. Since
\begin{equation}\label{prinbound}
\min_{x\in\R,\,u\in [0,1]}\big(\partial_uf_{L_n}(x,u)\big)\leq \lambda_1({L_n},\bar{u}_n) \leq \max_{x\in\R,\,u\in [0,1]}\big(\partial_uf_{L_n}(x,u)\big),
\end{equation}
the sequence $\big(\lambda_1(L_n,\bar{u}_n)\big)_{n\in\N}$ is then bounded. Up to extraction of some subsequence, there is a real number $\tilde{\lambda}_1\leq 0$ such that $\lambda_1(L_n,\bar{u}_n)\to\tilde{\lambda}_1$ as $n\to+\infty$. Now, denote $v_n(y)=\bar{u}_n(L_ny)$. Each function $v_n$ is $1$-periodic and obeys $a(y)v_n''(y)+a'(y)v_n'(y)+L_n^2f(y,v_n(y))=0$ for all $y\in\R$. It then follows from standard elliptic estimates that there is a $1$-periodic $C^2(\R)$ function $0\le v_{\infty}\le 1$ such that, up to extraction of some subsequence, $v_n\to v_{\infty}$ in $C^2(\R)$ as~$n\to+\infty$, and the function~$v_{\infty}$ solves $av_{\infty}''+a'v_{\infty}'=0$ in $\R$. Thus, $av_{\infty}'$ is a constant function, and then $v_{\infty}'$ has a sign. Since~$v_{\infty}$ is $1$-periodic, $v_{\infty}$ is then a constant function.

Next, one shows that $v_{\infty}=\overline{\theta}$. Integrating $(a_{L_n}\bar{u}_n')'+f_{L_n}(x,\bar{u}_n)=0$ over $[0,L_n]$ yields
$$0=\frac{1}{L_n}\int_0^{L_n}f_{L_n}(x,\bar{u}_n(x))dx=\int_0^{1}f(y,v_n(y))dy \to \overline{f}(v_\infty)\,\,\, \hbox{as}\,\,n\to+\infty.$$
Therefore, $\overline{f}(v_\infty)=0$. If $v_\infty=0$, then the assumption~\eqref{asspars} would imply that $0<v_n\le\delta$ in~$\R$ for~$n$ large enough, whence $f(y,v_n(y))\leq -\gamma v_n(y)< 0$ and $(av_n')'=-L_n^2f(y,v_n)>0$ in~$\R$ for~$n$ large enough, which contradicts the fact that $av_n'$ is $1$-periodic. Similarly, one obtains that~$v_{\infty}\neq 1$. It then follows from the assumption~\eqref{hypoverf} that~$v_{\infty}=\overline{\theta}$.

Finally, for any $n\in\N$, multiply the equation~\eqref{eqpsin} by~$\psi_n^{-1}$ and integrate by parts over $[0,L_n]$. It then follows that
$$\int_0^{L_n}\frac{a_{L_n}(x)(\psi_n'(x))^2}{(\psi_n(x))^2}dx+\int_0^{L_n} \partial_uf_{L_n}(x,\bar{u}_n(x))dx=L_n\lambda_1(L_n,\bar{u}_n).$$
As a consequence,
$$\lambda_1(L_n,\bar{u}_n)\geq \frac{1}{L_n}\int_0^{L_n} \partial_uf_{L_n}(x,\bar{u}_n(x))dx=\int_0^{1} \partial_uf(y,v_n(y))dy.$$
Taking the limit as $n\to+\infty$ yields that $\tilde{\lambda}_1\geq \overline{f}'(\overline{\theta})>0$, which contradicts the assumption that~$\tilde{\lambda}_1\leq0$. The proof of Lemma~\ref{unstable} is thereby complete.
\end{proof}

A consequence of Lemma~\ref{unstable} is the following non-existence result. Before doing so, we first introduce an important notation: in the sequel, for any $u_0\in C(\R,[0,1])$, $u(t,x;u_0)$ denotes the unique solution of equation~\eqref{eqL} with initial value $u(0,x;u_0)=u_0(x)$.

\begin{lem}\label{noexstate}
For every $0<L<\tilde{L}_*$ and for every $L$-periodic steady state $0<\bar{u}<1$ of~\eqref{eqL}, there is no steady state $v$ of~\eqref{eqL} such that $0<v<\bar{u}$ and there is no steady state~$w$ of~\eqref{eqL} such that $\bar{u}<w<1$.
\end{lem}

\begin{proof} We only prove the first conclusion, since the proof of the second one is similar. Thus, let $0<L<\tilde{L}_*$, let $0<\bar{u}<1$ be an $L$-periodic steady state of~\eqref{eqL} and let $v$ be a steady state of~\eqref{eqL} such that $0\le v<\bar{u}$. Our goal is to show that $v\equiv0$ in $\R$.

{\it Step 1: $\sup_{x\in\R}\big(v(x)-\bar{u}(x)\big)<0$.} First, since $\lambda_{1}(L,\bar{u})>0$ by Lemma~\ref{unstable}, it follows from Definition~\ref{semistable} that there is~$R>L/2$ large enough such that~$\lambda_{1,R}(L,\bar{u})>0$, where~$\lambda_{1,R}(L,\bar{u})$ is the principal eigenvalue of~\eqref{truncted}. For any $\epsilon >0$, define
\be\label{defveps}v_{\epsilon}(x)=\left\{
\begin{array}{lll}
\bar{u}(x)-\epsilon\psi_R(x)\; & {\rm if}\,\,\,  |x|<R,\vspace{3pt}\\
\bar{u}(x)\; & {\rm if} \,\,\,|x|\geq R,
\end{array}\right.
\ee
where $\psi_R$ is a fixed positive eigenfunction of~\eqref{truncted} corresponding to $\lambda_{1,R}(L,\bar{u})$. Choose $\varepsilon_0>0$ small enough such that, for all $0<\epsilon\le\epsilon_0$, $0<v_{\epsilon}\le\bar{u}$ in $\R$ and
$$f_L(x,v_{\epsilon})-f_L(x,\bar{u})\leq -\partial_uf_L(x,\bar{u})\times\varepsilon\psi_R+\frac{\lambda_{1,R}(L,\bar{u})}{2}\times \varepsilon\psi_R\ \hbox{ in }(-R,R).$$
It then follows that
\begin{equation}\label{supersol}
\begin{split}
(a_L\,(v_{\epsilon})_x)_x+f_L(x,v_{\epsilon}) & =(a_L\bar{u}_x)_x-(a_L\,(\varepsilon\psi_R)_x)_x+f_L(x,v_{\epsilon})-f_L(x,\bar{u})+f_L(x,\bar{u})\vspace{3pt}\\
&\leq -\lambda_{1,R}(L,\bar{u})\times\varepsilon\psi_R+\frac{\lambda_{1,R}(L,\bar{u})}{2}\times\varepsilon\psi_R<0\ \hbox{ in }(-R,R)
\end{split}
\end{equation}
for all $0<\epsilon\le\epsilon_0$. This together with the fact that $\bar{u}$ is a stationary solution of equation~\eqref{eqL} implies that $v_{\epsilon}$ is a supersolution of equation~\eqref{eqL}.

Now, for any $k\in\Z$ and for any $0<\epsilon\le\epsilon_0$, the function $v_{\epsilon}(\cdot-kL)$ is also a supersolution of~\eqref{eqL}. Since $v<u$ in $\R$, it follows then from the strong elliptic maximum principle that, for every $0<\epsilon\le\epsilon_0$, there holds $v(x)<v_{\epsilon}(x-kL)$ for all $x\in(kL-R,kL+R)$ and for all~$k\in\Z$. Since $R>L/2$ and $\psi_R$ is continuous and positive in $(-R,R)$, one infers that $\sup_{x\in\R}\big(v(x)-\bar{u}(x)\big)<0$.

{\it Step 2: $v\equiv 0$ in $\R$.} Finally, let $\varphi$ be a principal eigenfunction of the periodic problem~\eqref{prineigen}, associated with the principal eigenvalue~$\lambda_1(L,\bar{u})$. With similar calculation as above, there is~$\eta_0>0$ such that for all~$0<\eta\le\eta_0$, the $L$-periodic function $\bar{u}-\eta\varphi$ satisfies $v<\bar{u}-\eta\varphi<\bar{u}$ and
\be\label{supersol2}
(a_L\,(\bar{u}-\eta\varphi)_x)_x+f_L(x,\bar{u}-\eta\varphi)<0\ \hbox{ in }\R,
\ee
that is $\bar{u}-\eta\varphi$ is a strict supersolution of~\eqref{eqL}. As a consequence, the solution $u(t,x;\bar{u}-\eta_0\varphi)$ of~\eqref{eqL} with initial condition $\bar{u}-\eta_0\varphi$ is decreasing in $t>0$ and, from standard parabolic estimates, it converges as $t\to+\infty$ to an~$L$-periodic steady state $u_{\infty}(x)$ of~\eqref{eqL} such that
$$0\le v\le u_{\infty}<\bar{u}-\eta_0\varphi<\bar{u}<1\hbox{ in }\R.$$
If $u_{\infty}\not\equiv0$ in $\R$, then $0<u_{\infty}<1$ from the strong maximum principle, whence $u_{\infty}$ is unstable from Lemma~\ref{unstable}, in the sense that $\lambda_1(L,u_{\infty})>0$. Therefore, as above, by calling $\psi$ a principal periodic eigenfunction of the periodic problem~\eqref{prineigen} associated with $u_{\infty}$, it follows that the functions~$u_{\infty}+\kappa\psi$ are subsolutions of~\eqref{eqL} for all $\kappa>0$ small enough. In particular, since $u_{\infty}<\bar{u}-\eta_0\varphi$ in~$\R$ and both functions are $L$-periodic and continuous, there is $\kappa_0>0$ such that $u_{\infty}+\kappa_0\psi$ is a subsolution of~\eqref{eqL} and $u_{\infty}+\kappa_0\psi<\bar{u}-\eta_0\varphi$ in $\R$, whence $u_{\infty}+\kappa_0\psi<u(t,\cdot;\bar{u}-\eta_0\varphi)$ in $\R$ for all $t>0$, from the maximum principle. Finally, passing to the limit as $t\to+\infty$ gives $u_{\infty}+\kappa_0\psi\le u_{\infty}$ in $\R$, which is impossible. Hence, $u_{\infty}\equiv 0$ and $v\equiv 0$ in $\R$, and the proof of Lemma~\ref{noexstate} is complete.
\end{proof}

As a consequence of Lemma~\ref{noexstate}, for every $0<L<\tilde{L}_*$ and for every $L$-periodic steady state $0<\bar{u}<1$, equation~\eqref{eqL} restricted to $E_1=\{u\in C(\R,[0,1])\,|\,0\leq u\leq \bar{u}\hbox{ in }\R\}$, and to~$E_2=\{u\in C(\R,[0,1])\,|\, \bar{u}\leq u\leq 1\hbox{ in }\R\}$ respectively, has a monostable structure. In order to prove the existence of pulsating fronts for~\eqref{eqL}, one will verify a counter-propagation condition on the spreading speeds of these subsystems, as defined in~\cite{fz}. To do so, denote
$$\mathcal{C}^-(0,\bar{u})=\big\{u\in C(\R,[0,1])\,\big|\, 0\le u\leq \bar{u},\, \limsup_{x\to+\infty}(u(x)-\bar{u}(x))<0\hbox{ and }u(x)=\bar{u}(x)\,\,\hbox{for}\,\, x \ll -1\big\},$$
$$\mathcal{C}^+(\bar{u},1)=\big\{u\in C(\R,[0,1])\,\big|\, \bar{u}\leq u \le1,\, \liminf_{x\to-\infty}(u(x)-\bar{u}(x))>0\hbox{ and }u(x)=\bar{u}(x)\,\,\hbox{for}\,\, x \gg 1\big\}.$$

\begin{lem}\label{spreadingspeed}
For every $0<L<\tilde{L}_*$ and for every $L$-periodic steady state $0<\bar{u}<1$ of~\eqref{eqL}, there are some real numbers $c^+>0$ and $c^->0$ such that
\begin{equation}\label{rightspeed}\left\{\baa{ll}
\displaystyle\mathop{\limsup}_{t\to+\infty,\, x\geq -c^-t} u(t,x;u_0)=0 & \hbox{for all }u_0\in\mathcal{C}^-(0,\bar{u}),\vspace{3pt}\\
\displaystyle\mathop{\liminf}_{t\to+\infty,\, x\leq c^+t} u(t,x;\tilde{u}_0)=1 & \hbox{for all }\tilde{u}_0\in\mathcal{C}^+(\bar{u},1).\eaa\right.
\end{equation}
\end{lem}

\begin{proof}
We only give the proof of the first assertion~\eqref{rightspeed}, since the arguments for the other one are similar. First, one claims that for any $u_0\in\mathcal{C}^-(0,\bar{u})$ and any constant $C\in\R$, there holds
\begin{equation}\label{convhalf}
\lim_{t\to+\infty} u(t,x;u_0)=0\quad \hbox{uniformly for}\,\,x\in[C,+\infty).
\end{equation}
So, fix any $u_0\in\mathcal{C}^-(0,\bar{u})$, any real number $C$, any $k_0\in\N$ such that $C\ge-k_0L$, and let~$\eta>0$ be arbitrary. There are then $\epsilon>0$ small enough and $n_0\in\N$ large enough such that
$$u_0(\cdot+nL)\leq v_{\epsilon}\ \hbox{ for all }n\ge n_0,\ n\in\N,$$
where $v_{\epsilon}$ is defined in~\eqref{defveps} and is a strict supersolution of~\eqref{eqL}, in the sense of~\eqref{supersol}. It follows from the parabolic maximum principle that $u(t,x;v_{\epsilon})<v_{\epsilon}(x)$ and $u(t,x;v_{\epsilon})$ is decreasing in $t>0$. By standard parabolic estimates, $u(t,x;v_{\epsilon})$ converges as $t\to+\infty$ locally uniformly in~$x\in\R$ to a stationary solution $v_{\epsilon,\infty}$ of equation~\eqref{eqL} with $0\leq v_{\epsilon,\infty}<\bar{u}$. Lemma~\ref{noexstate} and the strong elliptic maximum principle imply that~$v_{\epsilon,\infty}\equiv 0$. Therefore, there is $T>0$ such that
$$0\le u(t,y;v_{\epsilon})\le\eta\ \hbox{ for all }t\ge T\hbox{ and }|y|\le(k_0+n_0+1)L.$$
For any $x\ge C\,(\ge-k_0L)$, there is $l_x\in\Z$ such that $l_x\ge-k_0$ and $l_xL\le x\le(l_x+1)L$. With $n_x=k_0+n_0+l_x\ge n_0$, one has $|x-n_xL|\le|x-l_xL|+(k_0+n_0)L\le(k_0+n_0+1)L$. Hence, from the maximum principle and the periodicity of~\eqref{eqL}, it follows that, for all $t\ge T$,
$$0\le u(t,x;u_0)=u(t,x-n_xL;u_0(\cdot+n_xL))\le u(t,x-n_xL;v_{\epsilon})\le\eta.$$
The claim~\eqref{convhalf} is thereby proved.

Next, we fix a real number $\sigma$ such that $0<\sigma<\min_{\R}\bar{u}$ and a function $w_0\in\mathcal{C}^-(0,\bar{u})$ such that $\sigma\le w_0\le\bar{u}$ in $\R$ and $w_0=\sigma$ in $\R^+$. From~\eqref{convhalf} applied to $w_0$, and since $0\le u(t,x;w_0)\le\bar{u}(x)$ for all $t\ge0$ and $x\in\R$, there is a time $t_1>0$ such that $0\le u(t_1,x;w_0)\le w_0(x+L)$ for all~$x\in\R$. From the maximum principle, it follows by immediate induction that
\be\label{nt1}
0\le u(nt_1,x;w_0)\le w_0(x+nL)\ \hbox{ for all }n\in\N\hbox{ and }x\in\R.
\ee

Finally, one shows that the first assertion in~\eqref{rightspeed} holds with any positive constant $c^-$ such that $0<c^-<L/t_1$. Fix any function $u_0\in\mathcal{C}^-(0,\bar{u})$. By~\eqref{convhalf} and $u(t,\cdot;u_0)\le\bar{u}$, there is $T>0$ such that $0\le u(T,\cdot;u_0)\le w_0$, whence
\be\label{nt1bis}
0\le u(T+nt_1,x;u_0)\le u(nt_1,x;w_0)\le w_0(x+nL)\ \hbox{ for all }n\in\N\hbox{ and }x\in\R
\ee
by~\eqref{nt1} and the maximum principle. Let us now argue by contradiction and assume that ${\limsup}_{t\to+\infty,\, x\ge-c^-t} u(t,x;u_0)>0$. Then there are some sequences $(\tau_k)_{k\in\N}$ in $(0,+\infty)$ and $(x_k)_{k\in\N}$ in $\R$ such that $x_k\ge-c^-\tau_k$ for all $k\in\N$, $\tau_k\to+\infty$ as $k\to+\infty$ and $\liminf_{k\to+\infty}u(\tau_k,x_k;u_0)>0$. For $k$ large enough, one can write $\tau_k=T+n_kt_1+\tilde{\tau}_k$ with $n_k\in\N$, $0\le\tilde{\tau}_k\le t_1$ and $n_k\to+\infty$ as $k\to+\infty$. Write also $x_k=x'_k+x''_k$ with $x'_k\in L\Z$ and $-L\le x''_k\le 0$. Up to extraction of a subsequence, one can assume that $\tilde{\tau}_k\to\tau\in\R$ and $x''_k\to y\in\R$ as $k\to+\infty$. For $k$ large enough, denote
$$u_k(t,x)=u(t+\tau_k,x+x'_k;u_0)\ \hbox{ for }t\ge-\tau_k,\ x\in\R.$$
From standard parabolic estimates, the functions $u_k$ converge locally uniformly in $\R^2$, up to extraction of a subsequence, to a solution $u_{\infty}(t,x)$ of~\eqref{eqL} defined for all $(t,x)\in\R^2$ and such that $0\le u_{\infty}(t,x)\le\bar{u}(x)$ for all $(t,x)\in\R^2$, while $u_{\infty}(0,y)>0$. Furthermore, for any given~$m\in\Z$ and $x\in\R$, one has, for all $k$ large enough,
\be\label{ukmt1}
0\le u_k(-mt_1-\tilde{\tau}_k,x)=u(T+n_kt_1-mt_1,x+x'_k;u_0)\le w_0(x+x'_k+(n_k-m)L)
\ee
by~\eqref{nt1bis}. But $x'_k=x_k-x''_k\ge x_k\ge-c^-\tau_k\ge-c^-(T+n_kt_1+t_1)$, whence $x'_k+n_kL\to+\infty$ as~$k\to+\infty$ since $c^-<L/t_1$ and $n_k\to+\infty$. As a consequence, it follows from~\eqref{ukmt1} and the definitions of~$u_{\infty}$ and~$w_0$ that $0\le u_{\infty}(-mt_1-\tau,x)\le\sigma\le w_0(x)$ for all $m\in\Z$ and $x\in\R$. One infers that $u_{\infty}\equiv 0$ in $\R^2$. Indeed, for any $(t,x)\in\R^2$, one has, for all $m\in\N$ large enough,
$$0\le u_{\infty}(t,x)=u(t+mt_1+\tau,x;u_{\infty}(-mt_1-\tau,\cdot))\le u(t+mt_1+\tau,x;w_0).$$
The property~\eqref{convhalf} applied with $w_0$ implies that $u(t+mt_1+\tau,x;w_0)\to0$ as $m\to+\infty$ whence~$u_{\infty}(t,x)=0$ for all $(t,x)\in\R^2$, which contradicts $u_{\infty}(0,y)>0$. Therefore, the first assertion of~\eqref{rightspeed} is shown and the proof of Lemma~\ref{spreadingspeed} is complete.
\end{proof}

Based on the above preparations, one is ready to prove Theorem~\ref{thhomobis}.

\begin{proof}[Proof of Theorem~\ref{thhomobis}] Fix a period $L$ such that $0<L<\tilde{L}_*$. For any $t\geq 0$, define $Q_t: C(\R,[0,1])\to C(\R,[0,1])$ by
\begin{equation}\label{semiflow}
Q_t[u_0]=u(t,\cdot;u_0).
\end{equation}
By classical parabolic theory, together with Lemmas~\ref{unstable} and~\ref{spreadingspeed}, the semiflow $(Q_t)_{t\geq 0}$ satisfies the following properties:
\begin{itemize}
\item[(A1)](Periodicity) $T_y\big[Q_t[\varphi]\big]=Q_t\big[T_y[\varphi]\big]$ for all $\varphi\in C(\R,[0,1])$, $t>0$ and $y\in L\Z$, where $T_y:C(\R,[0,1])\to C(\R,[0,1])$ is the translation operator defined by $T_y[\psi]=\psi(\cdot-y)$.
\item[(A2)](Continuity) For any $t>0$, $Q_t$ is continuous with respect to the compact open topology.
\item[(A3)](Monotonicity) For any $t>0$, $Q_t$ is order preserving in the sense that $Q_t[\varphi_1]\geq Q_t[\varphi_2]$ whenever $\varphi_1\ge\varphi_2$ in $C(\R,[0,1])$.
\item[(A4)](Compactness) For any $t>0$, $Q_t$ is compact with respect to the compact open topology.
\item[(A5)](Bistability) Let $\mathcal{C}_{per}$ be the set of $L$-periodic functions in $C(\R,[0,1])$. For any $t~\!>~\!0$,~$Q_t$ maps~$\mathcal{C}_{per}$ to itself and is strongly monotone on $\mathcal{C}_{per}$ in the sense that $\inf_{x\in\R}\big(Q_t[\varphi_1](x)-Q_t[\varphi_2](x)\big)>0$ whenever $\varphi_1\ge\varphi_2$ in $\mathcal{C}_{per}$ with $\varphi_1\not\equiv\varphi_2$. Furthermore, the constant functions~$0$ and~$1$ ($\in\mathcal{C}_{per}$) are stationary solutions of~\eqref{eqL} and they are strongly stable from above and below, respectively, in the sense of ~\cite{fz}, namely, for every~$t>0$ there is~$\epsilon_0>0$ such that~$\sup_{x\in\R}\big(Q_t[\epsilon](x)-\epsilon\big)<0$ and $\inf_{x\in\R}\big(Q_t[1-\epsilon](x)-(1-\epsilon)\big)>0$ for all~$0<\epsilon\le\epsilon_0$, which follows from the assumption ~\eqref{bistable}. Lastly, any stationary solution~$0<\bar{u}<1$ in~$\mathcal{C}_{per}$ is strongly unstable from above and below in the sense of~\cite{fz} since for every $t>0$, there is~$\epsilon_0>0$ such that $\inf_{x\in\R}\big(Q_t[\overline{u}+\epsilon\varphi](x)-(\overline{u}(x)+\epsilon\varphi(x))\big)>0$ and $\sup_{x\in\R}\big(Q_t[\overline{u}-\epsilon\varphi](x)-(\overline{u}(x)-\epsilon\varphi(x))\big)<0$ in $\R$ for all $0<\epsilon\le\epsilon_0$, where $\varphi$ denotes the periodic principal eigenfunction of~\eqref{prineigen} with~$\lambda=\lambda_1(L,\bar{u})>0$.  Indeed, the inequalities~$Q_t[\overline{u}+\epsilon\varphi]>\overline{u}+\epsilon\varphi$ and $Q_t[\overline{u}-\epsilon\varphi]<\overline{u}-\epsilon\varphi$ in~$\R$ for all $0<\epsilon\le\epsilon_0$ follow from the fact that the $L$-periodic functions $\overline{u}+\epsilon\varphi$ and $\overline{u}-\epsilon\varphi$ are respectively strict sub- and supersolutions of the elliptic equation associated with ~\eqref{eqL},  which can be verified by calculations similar to~\eqref{supersol2}.
\item[(A6)](Counter-propagation)  For each stationary solution $\bar{u}\in \mathcal{C}_{per}$ with $0<\bar{u}<1$, one has~$c^-_*(0,\bar{u})+c^+_*(\bar{u},1)>0$, where $c^-_*(0,\bar{u})$ and $c^+_*(0,\bar{u})$ are the spreading speeds defined by
$$\begin{aligned}
c^-_*(0,\bar{u}) & = \sup\big\{c\in\R\,\big|\, \limsup_{t\to+\infty,\,x\geq-ct} u(t,x;u_0)=0\hbox{ for all }u_0\in\mathcal{C}^-(0,\bar{u})\big\},\vspace{3pt}\\
c^+_*(\bar{u},1) & = \sup\big\{c\in\R\,\big|\, \liminf_{t\to+\infty,\,x\leq ct} u(t,x;\tilde{u}_0)=1\hbox{ for all }\tilde{u}_0\in\mathcal{C}^+(\bar{u},1)\big\}.
\end{aligned}$$
Indeed, Lemma~\ref{spreadingspeed} implies that $c^-_*(0,\bar{u})\ge c^->0$ and $c^+_*(\bar{u},1) \ge c^+>0$, with the notations of Lemma~\ref{spreadingspeed}. Following~\cite{fz}, $c^-_*(0,\bar{u})$ is called the leftward spreading speed of equation~\eqref{eqL} on $\mathcal{C}^-(0,\bar{u})$, and $c^+_*(\bar{u},1)$ the rightward spreading speed of equation~\eqref{eqL} on~$\mathcal{C}^+(\bar{u},1)$ (Lemma~\ref{spreadingspeed} is then stronger than the counter-propagation condition given in~\cite{fz}, which is just defined as the positivity of the sum of these spreading speeds).
\end{itemize}

Having in hand the properties (A1)-(A6),  we then see from \cite[Proposition 3.1 and Theorems~3.4 and~4.1]{fz} that for any $0<L<\tilde{L}_*$, equation~\eqref{eqL} admits a pulsating front~$u_L(t,x)=\phi_L(x-c_Lt,x/L)$ with speed $c_L\in\R$ such that~$\phi_L(\xi,x)$ is nonincreasing in~$\xi$.

Lastly, the assumption~\eqref{hypoverf} yields the existence (and uniqueness) of a front $(\phi_0,c_0)$ for the homogenized equation~\eqref{homogenized}. If $c_0\neq0$, then Theorem~\ref{thhomo} implies that the speeds $c_L$ of the pulsating fronts given in the previous paragraph, which exist for all $0<L<\tilde{L}_*$, are such that~$c_L\to c_0$ as $L\to0^+$. On the other hand, if $c_0=0$, then $\int_0^1\overline{f}(u)du=0$ and Lemma~\ref{sign} implies that $c_L=0$ for all $0<L<\tilde{L}_*$. Hence, the proof of Theorem~\ref{thhomobis} is complete.
\end{proof}

%%%%%%%%%%%%%%%%%%%%%%%%%%%%%%%%%%%%%%%%

\subsection{The case of large periods $L$}\label{seclarge}

This section is devoted to the proof of Theorem~\ref{thlarge}. That is, under the assumption~\eqref{conlarge}, we show that the equation~\eqref{eqL}, for any period $L>0$ large enough, admits a pulsating front with positive speed. Firstly, as for the proof of Theorem~\ref{thhomobis}, we will show that for $L$ large enough, any $L$-periodic intermediate steady state $0<\bar{u}<1$ of~\eqref{eqL} is unstable and, applying the abstract results in ~\cite{fz}, we will then obtain the existence of a pulsating front with nonnegative speed. To complete the proof, we need to exclude the case of pulsating fronts with zero speed (stationary fronts), at least for $L$ large enough. This proof, as well as that of the instability of the intermediate steady states of equation~\eqref{eqL}, will use a passage to the limit as $L\to+\infty$ and the properties of the solutions to
\begin{equation}\label{eqfreezed}
a(y)(u^y)''(x)+f^y(u^y(x))=0\ \hbox{ and }\ 0<u^y(x)<1\ \hbox{ for all }x\in\R,
\end{equation}
where $y$ is any real number and $f^y(u)=f(y,u)$.

We first begin with the instability of all intermediate steady states of equation~\eqref{eqL} at large~$L$.

\begin{lem}\label{unstablelar}
Under the assumption~\eqref{conlarge}, there is $L^*>0$ such that $\lambda_1(L,\bar{u})>0$ for every $L>L^*$ and for every $L$-periodic steady state $\bar{u}$ of~\eqref{eqL} with $0<\bar{u}<1$, where $\lambda_1(L,\bar{u})$ is the principal eigenvalue defined in~\eqref{prineigen}.
\end{lem}

\begin{proof}
Assume by contradiction that there are some sequences $(L_n)_{n\in\N}$ in $(0,+\infty)$, $(\bar{u}_n)_{n\in\N}$ and~$(\psi_n)_{n\in\N}$ in $C^2(\R)$ such that $L_n\to+\infty$ as $n\to+\infty$ and, for each $n\in\N$,~$\bar{u}_n$ and $\psi_n$ satisfy~\eqref{equn} and~\eqref{eqpsin} with principal eigenvalue $\lambda_1(L_n,\bar{u}_n)\leq 0$. By~\eqref{prinbound}, one can assume that, up to extraction of some subsequence, $\lambda_1(L_n,\bar{u}_n)\to\tilde{\lambda}_1\in(-\infty,0]$ as $n\to+\infty$.

One first observes from the assumption~\eqref{bistable} that, for any $n\in\N$, there is $x_n\in [0,L_n]$ such that $\bar{u}_n(x_n)=\theta_{x_n/L_n}$. Otherwise, by continuity and $L_n$-periodicity of $\bar{u}_n$, one would have either~$\theta_{x/L_n}<\bar{u}_n(x)<1$ for all $x\in\R$ or $0<\bar{u}_n(x)<\theta_{x/L_n}$ for all $x\in\R$ (notice also that, by~\eqref{bistable} and~\eqref{asspars}, the function $x\mapsto\theta_x$ is continuous), whence the function
$(a_{L_n}\bar{u}_n')'$ would have a fixed strict sign; this last property would contradict the fact that $a_{L_n}\bar{u}_n'$ is an $L_n$-periodic function.

Define now $p_n(x)= \bar{u}_n(x+x_n)$ for $x\in\R$ and $n\in\N$. Each function $p_n$ is a solution to~$(a_{L_n}(x+x_n)p_n')'+f_{L_n}(x+x_n,p_n)=0$ in $\R$ with $p_n(0)=\theta_{x_n/L_n}$ and $0<p_n <1$ in $\R$. Up to extraction of some subsequence, one can assume that $x_n / L_n\to x_{\infty}\in[0,1]$ as $n\to+\infty$ and that, from standard elliptic estimates, there is a $C^2(\R)$ function $0\le p_{\infty}\le 1$ such that~$p_n\to p_{\infty}$ in~$C^2_{loc}(\R)$ as $n\to+\infty$. Moreover, $p_\infty$ solves
\begin{equation}\label{largelim}
\left\{\baa{l}
a(x_{\infty})\,p_{\infty}''+f(x_{\infty},p_{\infty})=0\ \hbox{ in }\R,\vspace{3pt}\\
p_{\infty}(0)=\theta_{x_\infty}\ \hbox{ and }\ 0< p_{\infty} < 1\hbox{ in }\R,\eaa\right.
\end{equation}
the strict inequalities following from the strong maximum principle. Similarly, by normalizing~$\psi_n$ in such a way that $\psi_n(x_n)=1$, there is a nonnegative $C^2(\R)$ function $\psi_{\infty}$ such that, up to extraction of some subsequence, $\psi_n(\cdot+x_n)\to\psi_{\infty}$ in $C^2_{loc}(\R)$ as $n\to+\infty$, and $\psi_{\infty}$ solves
\be\label{eqpsiinfty}\left\{\baa{l}
a(x_{\infty})\,\psi_{\infty}''+\partial_u f(x_{\infty},p_{\infty})\,\psi_{\infty}=\tilde{\lambda}_1\psi_{\infty}\hbox{ in }\R,\vspace{3pt}\\
\psi_{\infty}(0)=1\ \hbox{ and }\  \psi_{\infty} > 0\hbox{ in }\R\eaa\right.
\ee
(notice here that the function $\psi_{\infty}$ may not be bounded or periodic, since the convergence is only local as $L_n\to+\infty$).

By~\eqref{bistable} and~\eqref{conlarge}, according to phase diagrams of equation~\eqref{largelim}, the solution $p_{\infty}$ can only be of one of the following three types: either a constant function, or a non-constant periodic function, or a ground state solution such that $p_{\infty}(\pm\infty)=0$. In what follows, one will get a contradiction in each of these three cases.

{\it Case 1: $p_{\infty}$ is a constant solution, that is $p_{\infty}\equiv \theta_{x_{\infty}}$ in $\R$}. In this case, $\psi_{\infty}$ obeys the linear equation $\psi_{\infty}''+\beta\psi_{\infty}=0$ in $\R$ with $\beta=(\partial_uf(x_{\infty},\theta_{\infty})-\tilde{\lambda}_1)/a(x_{\infty})$. Since $\partial_uf(x_{\infty},\theta_{\infty})>0$ and~$\tilde{\lambda}_1\leq 0$, it follows that $\beta>0$ and that the positive function $\psi_{\infty}$ is strictly concave in $\R$, which is impossible. Hence, Case~1 is ruled out.

{\it Case 2: $p_{\infty}$ is a non-constant $\tilde{L}$-periodic solution with $\tilde{L}>0$}. In this case, $p_{\infty}'$ is a non-signed~$\tilde{L}$-periodic function satisfying
\be\label{eqpinfty'}
a(x_{\infty})(p_{\infty}')''+\partial_u f(x_{\infty},p_{\infty})p_{\infty}'=0\ \hbox{ in }\R,
\ee
whereas $\psi_{\infty}$ solves $a(x_{\infty})\psi_{\infty}''+\big(\partial_u f(x_{\infty},p_{\infty})-\tilde{\lambda}_1\big)\psi_{\infty} =0$ in $\R$. Since $\tilde{\lambda}_1\leq 0$, it follows from Sturm comparison theorem that $\psi_{\infty}$ must vanish somewhere, which is impossible since $\psi_{\infty}>0$ in $\R$. Hence, Case~2 is ruled out too.

{\it Case 3: $p_{\infty}$ is a non-periodic solution and $\lim_{x\to\pm\infty}p_{\infty}(x)=0$}. Denote $F(s)=\int_0^sf(x_{\infty},u)du$ for all $s\in[0,1]$. From the assumptions~\eqref{bistable} and $\int_0^1f(x_{\infty},u)\,du>0$, there is a real number $\bar{s}\in (\theta_{x_{\infty}},1)$ such that $F(0)=F(\bar{s})=0$, $F(s)<0$ for all $0<s<\bar{s}$ and $F(s)>0$ for all $\bar{s}<s\le 1$.  It then follows that there is $\bar{x}\in\R$ such that $p_{\infty}(\bar{x})=\bar{s}$, $p_{\infty}'(\bar{x})=0$, $p_{\infty}'>0$ in~$(-\infty,\bar{x})$ and~$p_{\infty}'<0$ in $[\bar{x},+\infty)$. Notice also by~\eqref{bistable} that
$$p_{\infty}''(\bar{x})=-\frac{f(x_{\infty},p_{\infty}(\bar{x}))}{a(x_{\infty})}=-\frac{f(x_{\infty},\bar{s})}{a(x_{\infty})}<0$$
and that there is $\underline{x}<\bar{x}$ such that $p_{\infty}''(x)=-f(x_{\infty},p_{\infty}(x))/a(x_{\infty})>0$ for all $x\le\underline{x}$. Furthermore, $\lim_{x\to-\infty}p_{\infty}''(x)=\lim_{x\to-\infty}p_{\infty}'(x)=0$. Denote
$$q(x)=\psi_{\infty}'(x)p_{\infty}'(x)-\psi_{\infty}(x)p_{\infty}''(x)\ \hbox{ for }x\in\R.$$
It follows from~\eqref{eqpsiinfty} and~\eqref{eqpinfty'} that $q'(x)=\tilde{\lambda}_1\psi_{\infty}(x)p_{\infty}'(x)/a(x_{\infty})\le0$ for all $x\le\bar{x}$, whence
\be\label{qx}
q(x)\ge q(\bar{x})=-\psi_{\infty}(\bar{x})p_{\infty}''(\bar{x})>0\ \hbox{ for all }x\le\bar{x}.
\ee
Therefore, $\psi_{\infty}'(x)p_{\infty}'(x)\ge\psi_{\infty}(x)p_{\infty}''(x)>0$ for all $x\le\underline{x}\,(<\bar{x})$. In particular, $\psi_{\infty}'(x)>0$ for all $x\le\underline{x}$ and, since $\psi_{\infty}$ is positive, the limit $\psi_{\infty}(-\infty)\in[0,+\infty)$ exists. By~\eqref{eqpsiinfty}, the function $\psi_{\infty}''$ has a finite limit as $x\to-\infty$ and it follows then from elementary arguments that $\psi_{\infty}'(x)\to0$ as $x\to-\infty$. Lastly, since $p'_{\infty}(-\infty)=p_{\infty}''(-\infty)=0$, one gets that $q(x)\to0$ as $x\to-\infty$, which contradicts~\eqref{qx}. As a consequence, Case~3 is ruled out too and the proof of Lemma~\ref{unstablelar} is complete.
\end{proof}

\begin{proof}[Proof of Theorem~\ref{thlarge}]
For any fixed $L>L^*$, consider the semiflow $(Q_t)_{t\geq 0}$ generated by~\eqref{semiflow} with the period $L$, that is, by the equation~\eqref{eqL} with $L$-periodic coefficients. The properties~(A1)-(A4) used in the proof of Theorem~\ref{thhomobis} are easily verified. Because of Lemma~\ref{unstablelar}, $(Q_t)_{t\geq 0}$ satisfies the bistability condition~(A5) and the same analysis as that in Lemma~\ref{spreadingspeed} implies that $(Q_t)_{t\geq 0}$ satisfies the counter-propagation property~(A6). Thus, it follows from~\cite{fz} that for any $L>L^*$, equation~\eqref{eqL} admits a pulsating front $u_L(t,x)=\phi_L(x-c_Lt,x/L)$ with speed $c_L$. Furthermore, assumption~\eqref{conlarge} and Lemma~\ref{sign} imply that $c_L\ge0$. To end the proof, even if it means redefining~$L^*$, one needs to show that $c_L>0$ for all $L>L^*$ (large enough).

Assume by contraction that there is a sequence $(L_n)_{n\in\N}$ in $(L^*,+\infty)$ converging to~$+\infty$ and such that $c_{L_n}=0$ for all $n\in\N$. Namely, for each $n\in\N$, there is a $C^2(\R)$ solution~$\phi_{n}$~of
\begin{equation}\label{stationary}
\left\{\baa{l}
(a_{L_n}\phi_n')'+f_{L_n}(x,\phi_n)=0\hbox{ in }\R,\vspace{3pt}\\
\phi_n(-\infty)=1,\ \ \phi_n(+\infty)=0\ \hbox{ and}\ 0<\phi_n<1\hbox{ in }\R.\eaa\right.
\end{equation}
Since $\int_0^1f(x,u)\,du>0$ for all $x\in\R$ and $f$ is bounded in $\R\times[0,1]$, there is $\tau\in\R$ such that
\begin{equation}\label{upbound}
1-\delta < \tau <1\quad \hbox{and}\quad \int_0^s f(x,u)\,du>0\hbox{ for all }x\in\R\hbox{ and }s\in [\tau,1],
\end{equation}
where $\delta>0$ is the constant in~\eqref{asspars}. For every $n\in\N$, there is $y_n\in\R$ such that $\phi_n(y_n)=\tau$. Write~$y_n=y_n'+\tilde{y}_n$, with $y_n'\in L_n\Z$ and $\tilde{y}_n\in [0,L_n]$, and set $v_n(x)=\phi_n(x+y_n)$ for $x\in\R$ and~$n\in\N$. Since both $a_{L_n}$ and $f_{L_n}$ are $L_n$-periodic in $x$, each function $v_n$ obeys
$$\left\{\baa{l}
(a_{L_n}(x+\tilde{y}_n)v_n')'+f_{L_n}(x+\tilde{y}_n,v_n)=0\hbox{ in }\R,\vspace{3pt}\\
v_n(0)=\tau,\ \ v_n(-\infty)=1,\ \ v_n(+\infty)=0\ \hbox{ and}\ 0<v_n<1\hbox{ in }\R.\eaa\right.$$
Up to extraction of some subsequence, one can assume that  $\tilde{y}_n/L_n\to y_{\infty}\in[0,1]$ as $n\to+\infty$ and that, from standard elliptic estimates, $v_n\to v_{\infty}$ as $n\to+\infty$ in $C^2_{loc}(\R)$, where the function $0\le v_{\infty}\le 1$ solves
\begin{equation}\label{limiteq}
a(y_{\infty})\,v_{\infty}''+f(y_{\infty},v_{\infty})=0\hbox{ in }\R
\end{equation}
and $v_{\infty}(0)=\tau$, whence $0< v_{\infty}<1$ in $\R$ from the strong elliptic maximum principle. As for equation~\eqref{largelim} used in Lemma~\ref{unstablelar}, it follows from~\eqref{conlarge} that there are three types of solutions to equation~\eqref{limiteq}: the constant solutions (equal to $\theta_{y_{\infty}}$), the non-constant periodic solutions and the non-periodic ground state solutions converging to $0$ at $\pm\infty$.  In all cases, by multiplying the equation~\eqref{limiteq} by $v_{\infty}'$ and integrating on suitable intervals, it follows easily that~$\int_{\underline{s}}^{\bar{s}}f(y_{\infty},u)du=0$, where $0\le\inf_{\R}v_{\infty}=\underline{s}\le\overline{s}=\max_{\R}v_{\infty}<1$. It follows then from~\eqref{bistable} and~\eqref{limiteq} that $\underline{s}\le\theta_{y_{\infty}}\le\bar{s}$ and that $\int_0^sf(y_{\infty},u)du\le0$ for all $0\le s\le\bar{s}$. In particular, since~$v_{\infty}(0)=\tau\in[0,\bar{s}]$, one gets $\int_0^{\tau}f(y_{\infty},u)du\le0$. One has then reached a contradiction with~\eqref{upbound} and the proof of Theorem~\ref{thlarge} is thereby complete.
\end{proof}

%%%%%%%%%%%%%%%%%%%%%%%%%%%%%%%%%%%%%%%%
%%%%%%%%%%%%%%%%%%%%%%%%%%%%%%%%%%%%%%%%

\SE{The set $E$ of periods $L$ for which~\eqref{eqL} admits pulsating fronts with nonzero speed}\label{sec3}

This section is devoted to the proof of  Theorems~\ref{thE1} and~\ref{thE2} on the set $E$ of periods $L$ for which pulsating fronts with nonzero speed exist. Theorem~\ref{thE1} is similar to Theorem~\ref{thhomo} in the sense that they are both concerned with the existence and convergence of pulsating fronts as $L$ converges to a fixed $L_0\ge0$, given that~\eqref{eqL} when $L_0>0$ (resp.~\eqref{homogenized} when $L_0=0$) admits a non-stationary pulsating front. Namely, to prove Theorem~\ref{thE1},  we apply the implicit function theorem for some suitable function space as in Theorem~\ref{thhomo}, and the arguments are actually simpler since no singularity occurs when $L$ converges to $L_0>0$. The proof of Theorem~\ref{thE2} relies on several passages to the limit and the analysis of the stability of the limiting solutions.

%%%%%%%%%%%%%%%%%%%%%%%%%%%%%%%%%%%%%%%%

\subsection{Proof of Theorem~\ref{thE1}}\label{sec31}

Throughout the proof, we assume that equation~\eqref{eqL} with $L=L_0>0$ admits a pulsating front
$$U(t,x)=\phi_{L_0}(x-c_{L_0}t,x/L_0)$$
with a nonzero speed $c_{L_0}$. From parabolic regularity applied to the equations satisfied by $u$ and $u_t$, the function $u$ is of class $C^2$ in $\R^2$ and so is the function $\phi_{L_0}$. As in the proof of Theorem~\ref{thhomo}, one can assume without loss of generality that $(\phi_{L_0}(\xi,y),c_{L_0})$ solves~\eqref{pulsolu} with $c_{L_0}>0$.

We use here the same notations $\mathcal{D}_L$, $\T$, $L^2(\R\times\T)$ and $H^1(\R\times\T)$ as in Section~\ref{smalllhomo}. A positive real number $\beta>0$ is given. For any $c>0$ and $L>0$, the linear operator $M_{c,L}: \mathcal{D}_L\mapsto L^2(\R\times\T)$ defined in~\eqref{linhomo} is invertible by Lemma~\ref{invertible}. Now for~$v\in H^1(\R\times\T)$, $c>0$ and $L>0$, we define $K(v,c,L)=f(y,v+\phi_{L_0})+\beta v + \tilde{\partial}_{L}(a\tilde{\partial}_{L}\phi_{L_0})+c\partial_{\xi}\phi_{L_0}$, where $\tilde{\partial}_L=\partial_{\xi}+L^{-1}\partial_y$, and
$$\tilde{G}(v,c,L)=\Big(v+M_{c,L}^{-1}(K(v,c,L)),\int_{\R^+\times \T}\big(\phi_{L_0}(\xi,y)+v(\xi,y)\big)^2-\phi_{L_0}^2(\xi,y)\Big).$$
Note that~$\tilde{G}(0,c_{L_0},L_0)=(0,0)$. Moreover, as done in the proof of Theorem~\ref{thhomo} and using also parabolic regularity, a pair~$(\phi_L,c_L)\in\big(\phi_{L_0}+H^1(\R\times\T)\big)\times(0,+\infty)$ solves~\eqref{pulsolu} for $L\neq 0$ with the normalization $\int_{\R^+\times \T} \phi_L^2=\int_{\R^+\times\T}\phi_{L_0}^2$ if and only if $\tilde{G}(\phi_L-\phi_{L_0},c_L,L)=(0,0)$. On the other hand, using Lemma~\ref{continuem2} and similar arguments as in the proof of Lemma~\ref{continueg}, it follows that, in $H^1(\R\times\T)\times(0,+\infty)\times(0,+\infty)$, the function $\tilde{G}$ is continuous with respect to $(v,c,L)$ and continuously differentiable with respect to $(v,c)$ with derivative given by
\begin{equation*}
\begin{split}
&\partial_{(v,c)}\tilde{G}(v,c,L)(\tilde{v},\tilde{c})\vspace{3pt}\\
=&\Big(\tilde{v}+M_{c,L}^{-1}\big((\partial_uf(y,v+\phi_{L_0})+\beta)\tilde{v}\big)-\tilde{c}M_{c,L}^{-1}\big(\partial_{\xi}(M_{c,L}^{-1}(K(v,c,L))-\phi_{L_0})\big),2\displaystyle\int_{\R^+\times \T}(\phi_{L_0}+v)\tilde{v}\Big)
\end{split}
\end{equation*}
for all $(\tilde{v},\tilde{c})\in H^1(\R\times\T)\times\R$. In particular, $\tilde{Q}:=\partial_{(v,c)}\tilde{G}(0,c_{L_0},L_0)$ is given by
$$\tilde{Q}(\tilde{v},\tilde{c})=\Big(\tilde{v}+M_{c_{L_0},L_0}^{-1}\big((\partial_uf(y,\phi_{L_0})+\beta)\tilde{v}\big)+\tilde{c}M_{c_{L_0},L_0}^{-1}(\partial_{\xi}\phi_{L_0}),2\displaystyle\int_{\R^+\times \T}\phi_{L_0}\tilde{v}\Big).$$

Now, in order to apply the implicit function theorem for $\tilde{G}$ around the point $(0,c_{L_0},L_0)$, one needs to show that the operator $\tilde{Q}$ is invertible as a map from $H^1(\R\times\T)\times\R$ to itself. The method used in Lemma~\ref{continueg} can be adapted to prove this property under the condition that the linearization of equation~\eqref{pulsolu} at $(\phi_{L_0},c_{L_0})$ satisfies similar properties as the operators~$H$ and~$H^*$ given by~\eqref{defH}. More precisely, let
\be\label{defHL0}
H_{L_0}(u)= \tilde{\partial}_{L_0}(a\tilde{\partial}_{L_0}u)+c_{L_0}\partial_{\xi}u +\partial_uf(y,\phi_{L_0})u\ \hbox{ for }u\in\mathcal{D}_{L_0}
\ee
and let the adjoint operator $H_{L_0}^*$ be defined by $H_{L_0}^*(u)= \tilde{\partial}_{L_0}(a\tilde{\partial}_{L_0}u)-c_{L_0}\partial_{\xi}u +\partial_uf(y,\phi_{L_0})u$ for $u\in\mathcal{D}_{L_0}$. Let us now work with complex valued functions. Namely, for $u=v+iw$ with $v,w\in\mathcal{D}_{L_0}$ and $i^2=-1$, we set $H_{L_0}(u)=H_{L_0}(v)+iH_{L_0}(w)$ and similarly for $H_{L_0}^*(u)$. One has $\big<H_{L_0}^*(v),u\big>_{L^2(\R\times\T,\C)}=\big<v,H_{L_0}(u)\big>_{L^2(\R\times\T,\C)}$ for all $u,v\in \mathcal{D}_{L_0}+i\mathcal{D}_{L_0}$, with $\big<w,z\big>_{L^2(\R\times\T,\C)}=\int_{\R\times\T}w\overline{z}$ for $w,z\in L^2(\R\times\T,\C)$.

\begin{lem}\label{degenerate}
The operators $H_{L_0}$ and  $H_{L_0}^*$ have algebraically simple eigenvalue $0$ and the range of~$H_{L_0}$ is closed in $L^2(\R\times\T,\C)$. Furthermore, if $\lambda\in\C^*$ is an eigenvalue of $H_{L_0}$, then $Re(\lambda)<0$.
\end{lem}

Notice that $H_{L_0}$ and $H_{L_0}^*$ are not elliptic operators in the variables $(\xi,y)$, but by the change of variable defined in~\eqref{changev} they are equivalent to standard parabolic operators in the variables~$(t,x)$ and the parabolic theory helps overcome this degeneracy. As a matter of fact, the simplicity of the eigenvalue $0$ is highly dependent on the following maximum principle for $H_{L_0}$ (similar results can be obtained for $H_{L_0}^*$).

\begin{lem}\label{smp}
Let $\phi$ be a $C^2(\R\times\T,\R)$ solution of $H_{L_0}(\phi)\leq 0$ on $\R\times\T$ with $\phi\geq 0$ in $\R\times\T$. Then either $\phi\equiv 0$ in $\R\times\T$, or $\phi(\xi,y)>0$ for all $(\xi,y)\in\R\times\T$.
\end{lem}

\begin{proof}
This conclusion follows from the strong parabolic maximum principle applied to the function $u(t,x)=\phi(x-c_{L_0}t,x/L_0)$ and from the periodicity of $\phi(\xi,y)$ in the $y$-variable (see also Proposition~3.1 of~\cite{x5}).
\end{proof}

Now one is ready to prove Lemma~\ref{degenerate}. Note that similar conclusions were obtained in~\cite{x3} for the linearized operator of an equation with combustion nonlinearity. Special weighted spaces (requiring functions to decay to zero at a certain exponential rate as $|x|\to\infty$) are introduced in that paper, whereas they are not needed here due to the bistable assumption~\eqref{asspars}. The strategy for the poof of Lemma~\ref{degenerate} is actually a little bit different from that used in Section 2 of~\cite{x3}.

\begin{proof}[Proof of Lemma~\ref{degenerate}]  We proceed with five steps.

{\it Step 1: $0$ is a geometrically simple eigenvalue of $H_{L_0}$ in $\mathcal{D}_{L_0}+i\mathcal{D}_{L_0}$.} First, by parabolic regularity, the time-derivative $U_t$ of the function $U(t,x)=\phi_{L_0}(x-c_{L_0}t,x/L_0)$ is of class~$C^{1,2}(\R^2)$. Thus, one can differentiate the equation~\eqref{pulsolu} (with $L=L_0$) satisfied by $\phi_{L_0}$ with respect to $\xi$. More precisely, the function $\partial_{\xi}\phi_{L_0}$ satisfies
$$\tilde{\partial}_{L_0}\big(a(y)\tilde{\partial}_{L_0}(\partial_{\xi}\phi_{L_0})\big)+c_{L_0}\partial_{\xi}(\partial_{\xi}\phi_{L_0})+\partial_uf(y,\phi_{L_0})\partial_{\xi}\phi_{L_0}=0\ \hbox{ for all }(\xi,y)\in\R\times\T,$$
where $\partial_{\xi}\phi_{L_0}(\xi,y)=-c_{L_0}^{-1}U_t((L_0y-\xi)/c_{L_0},L_0y)$, $\tilde{\partial}_{L_0}(\partial_{\xi}\phi_{L_0})(\xi,y)=-c_{L_0}^{-1}U_{tx}((L_0y-\xi)/c_{L_0},L_0y)$, $\partial_{\xi}(\partial_{\xi}\phi_{L_0})(\xi,y)=c_{L_0}^{-2}U_{tt}((L_0y-\xi)/c_{L_0},L_0y)$ and
$$\tilde{\partial}_{L_0}\big(a(y)\tilde{\partial}_{L_0}(\partial_{\xi}\phi_{L_0})\big)(\xi,y)=-\frac{a(y)}{c_{L_0}}U_{txx}\Big(\frac{L_0y-\xi}{c_{L_0}},L_0y\Big)-\frac{a'(y)}{L_0c_{L_0}}U_{tx}\Big(\frac{L_0y-\xi}{c_{L_0}},L_0y\Big).$$
Therefore, it follows from Lemma~\ref{exponential} that $\partial_{\xi}\phi_{L_0}\in \mathcal{D}_{L_0}$. On the other hand, Theorem~\ref{thqual} and the strong parabolic maximum principle imply that $U_t>0$ in $\R^2$, whence $\partial_{\xi}\phi_{L_0}$ is a negative eigenfunction of $H_{L_0}$ for the eigenvalue~$0$.

Next, suppose that $v\in\mathcal{D}_{L_0}+i\mathcal{D}_{L_0}$ satisfies~$H_{L_0}(v)=0$ and $v\not\equiv 0$. Without loss of generality, one can assume that $v$ is real valued. By rewriting the equation~$H_{L_0}(v)=0$ in its weak form in the variables $(t,x)$, it follows from parabolic regularity theory and bootstrap arguments that~$v$ is of class $C^2(\R^2)$ and is a bounded classical solution of~$H_{L_0}(v)=0$ such that $v(\pm\infty,\cdot)=0$ uniformly in $\T$. For any $\mu\in\R$ and $(\xi,y)\in \R\times\T$, let
$$w_{\mu}(\xi,y)=v(\xi,y)-\mu\,\partial_{\xi}\phi_{L_0}(\xi,y). $$
Each $w_{\mu}$ is a classical solution of $H_{L_0}(w_{\mu})=\tilde{\partial}_{L_0}\big(a(y)\tilde{\partial}_{L_0}w_{\mu}\big)+c_{L_0}\partial_{\xi}w_{\mu}+\partial_uf(y,\phi_{L_0})w_{\mu}=0$ in~$\R\times\T$ with $w_{\mu}(\pm\infty,\cdot)=0$ uniformly in $\T$. One sees from~\eqref{asspars} and the uniform continuity of~$\partial_uf$ in~$\R\times[0,1]$ that there is $N>0$ large enough such that
\be\label{defN}
\partial_uf(y,\phi_{L_0}(\xi,y))\le-\frac{\gamma}{2}<0\hbox{ for all }|\xi|\geq N\hbox{ and }y\in\T.
\ee
Since $\partial_{\xi}\phi_{L_0}$ is negative and continuous in $\R\times\T$, it follows that, for this chosen $N$, there exists $\mu_0>0$ such that $w_{\mu}>0$ in $[-N,N]\times\T$ for all $\mu\ge\mu_0$. We claim that, for such $\mu$, $w_{\mu}(\xi,y)>0$ for all $(\xi,y)\in \R\times\T$. Indeed, otherwise, $w_{\mu}$ achieves a nonpositive minimum at a point $(\xi_0,y_0)\in\R\times\T$ such that $|\xi_0|>N$. If $w_{\mu}(\xi_0,y_0)<0$, then evaluating all terms of $H_{L_0}(w_{\mu})$ at $(\xi_0,y_0)$ yields
$$\baa{l}
\tilde{\partial}_{L_0}(a\tilde{\partial}_{L_0}w_{\mu})(\xi_0,y_0)+c_{L_0}\partial_{\xi}w_{\mu}(\xi_0,y_0)+ \partial_uf(y,\phi_{L_0}(\xi_0,y_0)) w_{\mu}(\xi_0,y_0)\vspace{3pt}\\
\qquad\qquad\qquad\qquad\ge\partial_uf(y,\phi_{L_0}(\xi_0,y_0)) w_{\mu}(\xi_0,y_0)\ge-\displaystyle\frac{\gamma}{2} w_{\mu} (\xi_0,y_0) >0,\eaa$$
which contradicts the equation $H_{L_0}(w_{\mu})=0$ in $\R\times\T$ (notice here that $H_{L_0}(w_{\mu})\le0$ would have been sufficient to conclude, that is $H_{L_0}(v)\le0$ would have been sufficient). If  $w_{\mu}(\xi_0,y_0)=0$, then strong maximum principle in Lemma~\ref{smp} shows that $w_{\mu}\equiv 0$, which is also impossible. Thus, one gets that $w_{\mu}>0$ in $\R\times\T$.

Now define
$$\nu =\inf\big\{\mu\in\R\, \big|\, w_{\mu}=v-\mu\partial_{\xi}\phi_{L_0}\geq 0\hbox{ in }\R\times\T\big\}.$$
Obviously, $-\infty <\nu\leq \mu_0$ and $w_{\nu}=v-\nu\partial_{\xi}\phi_{L_0}\geq 0\hbox{ in }\R\times\T$. If $w_{\nu}>0$ in the compact set~$[-N,N]\times\T$, then $w_{\nu+\epsilon}>0$ in $[-N,N]\times\T$ for $\epsilon>0$ small enough. Hence, as in the previous paragraph, it follows that $w_{\nu+\epsilon}>0$ in $\R\times\T$ for all $\epsilon>0$ small enough, which contradicts the definition of $\nu$. Therefore, $w_{\nu}$ vanishes somewhere in $[-N,N]\times\R$ and Lemma~\ref{smp} again implies that $v-\nu\partial_{\xi}\phi_{L_0}=w_{\nu}=0$ in $\R\times\T$. That is, $v=\nu\partial_{\xi}\phi_{L_0}$ in $\R\times\T$ (with $\nu\neq 0$ since $v\not\equiv0$).

{\it Step 2: $0$ is an algebraically simple eigenvalue of $H_{L_0}$.} Suppose that $H_{L_0}^m(v)=0$ for some integer $m\ge2$ and $v\in\mathcal{D}_{L_0}+i\mathcal{D}_{L_0}$ such that $H_{L_0}(v),\ldots,H^{m-1}_{L_0}(v)\in\mathcal{D}_{L_0}+i\mathcal{D}_{L_0}$. Without loss of generality, one can assume that $v$ is real valued. Since $\ker(H_{L_0})=\C\partial_{\xi}\phi_{L_0}$ and $v$ is real valued, it follows that $H_{L_0}^{m-1}(v)=C_1\partial_{\xi}\phi_{L_0}$ with some constant $C_1\in\R$. Without loss of generality, even if it means changing $v$ into $-v$, one can assume that $C_1\geq 0$. On the other hand, parabolic regularity theory implies that $H_{L_0}^{m-2}(v)$ is a bounded $C^2(\R\times\T)$ solution of $H_{L_0}\big(H_{L_0}^{m-2}(v)\big)=C_1\partial_{\xi}\phi_{L_0}\le0$ in $\R\times\T$. By considering functions of the type $H_{L_0}^{m-2}(v)-\mu\partial_{\xi}\phi_{L_0}$ with $\mu\in\R$, it follows then as in Step~1 that $H_{L_0}^{m-2}(v)=C_2\partial_{\xi}\phi_{L_0}$ for some constant $C_2$, whence $H_{L_0}^{m-1}(v)=0$. By an immediate induction, one concludes that $v=C_3\partial_{\xi}\phi_{L_0}$ for some constant $C_3$.

{\it Step 3: if $\lambda\in\C^*$ is an eigenvalue of $H_{L_0}$, then $Re(\lambda)<0$.} Let $\lambda\in\C$ be an eigenvalue of~$H_{L_0}$, with an eigenfunction $\psi\in\mathcal{D}_{L_0}+i\mathcal{D}_{L_0}$, and assume that $Re(\lambda)\ge0$. By standard parabolic estimates applied to its real and imaginary parts, the function $u(t,x)=\psi(x-c_{L_0}t,x/L_0)$ is a classical solution of
\be\label{equlambda}
u_t-(a_{L_0}(x)u_x)_x-\partial_uf_{L_0}(x,U(t,x))u=-\lambda u\ \hbox{ in }\R^2.
\ee
Furthermore, $u\in W^{1,\infty}(\R^2,\C)$ and then $u(t,x)\to0$ as $|x-c_{L_0}t|\to+\infty$. Denote $\rho=|u|$ the modulus of $u$. In the open set $\Omega:=\big\{(t,x)\in\R^2\ |\ \rho(t,x)>0\big\}$, one can write $u(t,x)=\rho(t,x)\,\me^{i\vartheta(t,x)}$ where the real-valued functions $\rho$ and $\vartheta$ are of class $C^1$ with respect to $t$ and $C^2$ with respect to~$x$ in~$\Omega$. By putting $u=\rho\,\me^{i\vartheta}$ in~\eqref{equlambda} and taking the real part, one infers that
\be\label{eqrho}
\rho_t-(a_{L_0}(x)\rho_x)_x-\partial_uf_{L_0}(x,U(t,x))\rho=-(Re(\lambda)+a_{L_0}\vartheta_x^2)\,\rho\le0\ \hbox{ in }\Omega.
\ee
By~(\ref{bistable}-\ref{asspars}), there is $N>0$ such that $\partial_uf_{L_0}(x,U(t,x))\le-\gamma/2<0$ for all $|x-c_{L_0}t|\ge N$. Since~$U_t$ is positive and continuous in $\R^2$ and since $U_t(t+L_0/c_{L_0},x+L_0)=U_t(t,x)$ in $\R^2$, one has~$\inf_{|x-c_{L_0}t|\le N}U_t(t,x)>0$ and there is $\sigma>0$ such that
$$\rho\le\sigma U_t\ \hbox{ for all }(t,x)\in\R^2\hbox{ with }|x-c_{L_0}t|\le N.$$
It follows then as in the end of the proof of Lemma~\ref{exponential} that $\rho\le\sigma U_t$ for all $x-c_{L_0}t\ge N$ (otherwise, there would exist $\epsilon^*>0$ with $z:=\sigma U_t-\rho\ge-\epsilon^*$ in $\{x-c_{L_0}t\ge N\}$ and a point~$(t^*,x^*)$ such that $x^*-c_{L_0}t^*>N$ and $z(t^*,x^*)=-\epsilon^*$; since $\rho(t^*,x^*)=\sigma U_t(t^*,x^*)+\epsilon^*>0$, there holds~$z_t-(a_{L_0}(x)z_x)_x-\partial_uf_{L_0}(x,U(t,x))z\ge0$ in a neighborhood of $(t^*,x^*)$, and one gets a contradiction at $(t^*,x^*)$, since $-\partial_uf_{L_0}(x^*\!,U(t^*\!,x^*))\,z(t^*\!,x^*)\!=\!\epsilon^*\partial_uf_{L_0}(x^*\!,U(t^*\!,x^*))\!\le\!-\epsilon^*\gamma/2\!<\!0$). Similarly,~$\rho\le\sigma U_t$ for all $x-c_{L_0}t\le-N$, whence $\rho\le\sigma U_t$ in $\R^2$.

Set now $\sigma^*=\inf\big\{\varsigma\ge 0\ |\ \rho\le\varsigma U_t$ in $\R^2\big\}$. One has $\sigma^*>0$ since $\rho\not\equiv0$, and~$\rho\le\sigma^*U_t$ in~$\R^2$. If $\inf_{|x-c_{L_0}t|\le N}\big(\sigma^*U_t(t,x)-\rho(t,x)\big)>0$, then there would exist $\sigma_*\in(0,\sigma^*)$ such that~$\rho(t,x)\le\sigma_*U_t(t,x)$ for all $|x-c_{L_0}t|\le N$ (since $U_t$ is bounded) and it would follow as in the previous paragraph that $\rho\le\sigma_*U_t$ in $\R^2$, contradicting the minimality of $\sigma^*$. Consequently,~$\inf_{|x-c_{L_0}t|\le N}\big(\sigma^*U_t(t,x)-\rho(t,x)\big)=0$. By continuity and the properties
\be\label{Utrho}
U_t(t+L_0/c_{L_0},x+L_0)=U_t(t,x),\ \ \rho(t+L_0/c_{L_0},x+L_0)=\rho(t,x)\ \hbox{ for all }(t,x)\in\R^2,
\ee
there is $(t_0,x_0)\in\R^2$ such that $\rho(t_0,x_0)=\sigma^*U_t(t_0,x_0)$. Hence, $\rho>0$ in (at least) a neighborhood of $(t_0,x_0)$ and the strong parabolic maximum principle implies actually that $\rho=\sigma^*U_t>0$ in~$(-\infty,t_0]\times\R$ and then in $\R^2$ by~\eqref{Utrho}. Therefore, $Re(\lambda)=0$ and $\vartheta_x=0$ in $\R^2$, by~\eqref{eqrho}. On the other hand, by taking the imaginary part of~\eqref{equlambda}, one infers that $\vartheta_t=-Im(\lambda)$ in $\R^2$. Finally, since $\vartheta(t+L_0/c_{L_0},x+L_0)=\vartheta(t,x)$ in $\R^2$, one gets that $\vartheta$ is constant in $\R^2$ and that $\lambda=0$. As a conclusion, $Re(\lambda)<0$ if $\lambda\neq0$.

{\it Step 4: the range of $H_{L_0}$ is closed in $L^2(\R\times\T,\C)$.} Let $(v_n)_{n\in\N}$ in $\mathcal{D}_{L_0}+i\mathcal{D}_{L_0}$ and $(g_n)_{n\in\N}$ in $L^2(\R\times\T,\C)$ be some sequences such that $H_{L_0}(v_n)=g_n\to g$ in $L^2(\R\times\T,\C)$ as $n\to+\infty$. Without loss of generality, one can assume that all functions $v_n$, $g_n$ and $g$ are real valued and that $v_n$ is orthogonal to $\partial_{\xi}\phi_{L_0}$ in $L^2(\R\times\T,\R)$.

Let us now show that the sequence $(\|v_n\|_{L^2(\R\times\T)})_{n\in\N}$ is bounded. Suppose the contrary, let~$w_n=v_n/\|v_n\|_{L^2(\R\times\T)}$ with $\|w_n\|_{L^2(\R\times\T)}=1$ and observe that $H_{L_0}(w_n)=g_n/\|v_n\|_{L^2(\R\times\T)}\to 0$ as $n\to+\infty$. Notice that
$$M_{c_{L_0},L_0} (w_n)=H_{L_0}(w_n)-(\partial_uf(y,\phi_{L_0})+\beta)w_n\ \hbox{ for all }n\in\N.$$
By Lemma~\ref{invertible}, the sequence $(w_n)_{n\in\N}$ is then bounded in $H^1(\R\times\T)$ and then a subsequence converges in $H^1(\R\times\T)$ weakly and in $L^2_{loc}(\R\times\T)$ strongly to some $w_0\in H^1(\R\times\T)$. Furthermore,~$w_0$ is orthogonal to $\partial_{\xi}\phi_{L_0}$ in $L^2(\R\times\T)$ and $\int_{\R\times\T}\!a(y)(\tilde{\partial}_{L_0}w_0)(\tilde{\partial}_{L_0}\varphi)\!-\!c_{L_0}\varphi\partial_{\xi}w_0\!-\!\partial_uf(y,\phi_{L_0})w_0\varphi\!=\!0$ for all $\varphi\in H^1(\R\times\T)$, whence $w_0\in\mathcal{D}_{L_0}$ and $H_{L_0}(w_0)=0$. Since $\ker(H_{L_0})=\C(\partial_{\xi}\phi_{L_0})$, it follows that~$w_0=0$. Let $N>0$ be as in~\eqref{defN} and let $\rho:\R\to[0,1]$ be the piecewise affine function defined by $\rho(\xi)=0$ for all $\xi\le N$, $\rho(\xi)=\xi-N$ for all $\xi\in[N,N+1]$ and $\rho(\xi)=1$ for all~$\xi\ge N+1$. Then, by integrating the equation $H_{L_0}(w_n)=g_n/\|v_n\|_{L^2(\R\times\T)}$ against $w_n\rho$, one gets that
$$\baa{rcl}
\displaystyle-\int_{(N,+\infty)\times\T} a(y)\,\rho(\xi)\,(\tilde{\partial}_{L_0}w_n)^2-\int_{(N,N+1)\times\T}a(y)\,w_n\,\tilde{\partial}_{L_0}w_n & \!\!\!\! & \vspace{3pt}\\
\displaystyle-\int_{(N,N+1)\times\T}\frac{c_{L_0}}{2}w_n^2+\int_{(N,+\infty)\times\T}\rho(\xi)\,\partial_uf(y,\phi_{L_0})\,w_n^2 & \!\!=\!\! & \displaystyle\int_{(N,+\infty)\times\T}\frac{\rho(\xi)\,g_nw_n}{\|v_n\|_{L^2(\R\times\T)}}\ \mathop{\longrightarrow}_{n\to+\infty}\ 0.\eaa$$
Since the sequence $(w_n)_{n\in\N}$ is bounded in $H^1(\R\times\T)$, since $w_n\to0$ in $L^2_{loc}(\R\times\T)$ and since both terms $-\int_{(N,+\infty)\times\T}a(y)\,\rho(\xi)\,(\tilde{\partial}_{L_0}w_n)^2$ and $\int_{(N,+\infty)\times\T}\rho(\xi)\,\partial_uf(y,\phi_{L_0})\,w_n^2$ are nonpositive, it follows that they both converge to $0$ as $n\to+\infty$. In particular, by~\eqref{defN}, $\|w_n\|_{L^2((N+1,+\infty)\times\T)}\to0$ as~$n\to+\infty$. Using the same analysis over $(-\infty,-N)$ implies that $\|w_n\|_{L^2((-\infty,-N-1)\times\T)}\to0$ as~$n\to+\infty$. Finally the sequence~$(w_n)_{n\in\N}$ tends to 0 strongly in $L^2(\R\times\T)$ as $n\to+\infty$, which contradicts the fact that~$\|w_n\|_{L^2(\R\times\T)}=1$. Hence, the sequence $(v_n)_{n\in\N}$ is bounded in $L^2(\R\times\T)$.

Since $M_{c_{L_0},L_0}(v_n)=g_n-(\partial_uf(y,\phi_{L_0})+\beta)v_n$, Lemma~\ref{invertible} again implies that $(v_n)_{n\in\N}$ is bounded in $H^1(\R\times\T)$. Therefore, a subsequence converges weakly in $H^1(\R\times\T)$ to some $v\in\mathcal{D}_{L_0}$ such that $H_{L_0}(v)=g$.

{\it Step 5: $0$ is an algebraically simple eigenvalue of $H^*_{L_0}$.}  Choose a sufficient large real number~$\lambda_0$ such that $\lambda_0> \partial_uf(x,u)$  for all $(x,u)\in\R\times[0,1]$. Denote $\tilde{H}_{L_0}(v)=H_{L_0}(v)-\lambda_0v$ for $v\in\mathcal{D}_{L_0}+i\mathcal{D}_{L_0}$. The adjoint operator $\tilde{H}^*_{L_0}$ of $\tilde{H}_{L_0}$ is given by $\tilde{H}^*_{L_0}(v)=H_{L_0}^*(v)-\lambda_0v$ for $v\in \mathcal{D}_{L_0}+i\mathcal{D}_{L_0}$, in such a way that $\big<\tilde{H}^*_{L_0}(v),u\big>_{L^2(\R\times\T,\C)}=\big<v,\tilde{H}_{L_0}(u)\big>_{L^2(\R\times\T,\C)}$ for all $u,v\in \mathcal{D}_{L_0}+i\mathcal{D}_{L_0}$. As in the proof of Lemma~\ref{invertible}, it is easy to see that the kernels of $\tilde{H}_{L_0}$ and $\tilde{H}_{L_0}^*$ are reduced to $\{0\}$. In addition, similar arguments as in Step~4 above imply that the range of $\tilde{H}_{L_0}$ is closed in~$L^2(\R\times\T,\C)$. Thus, the operator $\tilde{H}_{L_0}:\mathcal{D}_{L_0}+i\mathcal{D}_{L_0}\to L^2(\R\times\T,\C)$ is invertible. Then the arguments in p.~220 of~\cite{x3} imply that there is a strictly positive function $v^*\in \mathcal{D}_{L_0}$ such that $H^*_{L_0}(v^*)=0$. That is,~$0$ is an eigenvalue of~$H^*_{L_0}$ with a positive eigenfunction $v^*$. Applying the above analysis in Step 1-2 to $H^*_{L_0}$ provides the algebraic simplicity of the eigenvalue~$0$. The proof of Lemma~\ref{degenerate} is thereby complete.
\end{proof}

\begin{proof}[Proof of Theorem~\ref{thE1}]
Given the above preparations, it follows as in Lemma~\ref{continueg} that the ope\-rator~$\tilde{Q}=\partial_{(v,c)}\tilde{G}(0,c_0,L_0):H^1(\R\times\T)\times \R\to H^1(\R\times\T)\times \R$ is invertible. The proof of Theorem~\ref{thE1} is then almost the same as that of Theorem~\ref{thhomo}, so we omit the details.
\end{proof}

%%%%%%%%%%%%%%%%%%%%%%%%%%%%%%%%%%%%%%%%

\subsection{Proof of Theorem~\ref{thE2}}\label{sec32}

This section is concerned with the proof of Theorem~\ref{thE2}. We first show by contradiction in Lemma~\ref{lemspeeds} below that the speeds $c_{L_n}$ are bounded when $L_n\in E$ approaches $L\in\partial E\cap(0,+\infty)$. This property could actually be viewed as a consequence of the more general boundedness pro\-perty~\eqref{bounded}, which follows from an even more general boundedness result on the global mean speeds of transition fronts, see~\cite{dhz2}. Here, for the sake of completeness, Lemma~\ref{lemspeeds} is proved. Then, the strategy of the proof of Theorem~\ref{thE2} is the following: if for some sequence the speeds~$c_{L_n}$ converge to a nonzero real number as $L_n\to L$ with $L_n\in E$, then equation~\eqref{eqL} admits some pulsating fronts connecting $0$, resp. $1$, to some $L$-periodic steady states. On the other hand, if the speeds $c_{L_n}$ converge to $0$, then equation~\eqref{eqL} admits either a semistable stationary front connecting $0$ and $1$, or some semistable stationary fronts connecting~$0$, resp. $1$, to some semistable $L$-periodic steady states.

Before doing so, we first state an elementary lemma which will be used several times.

\begin{lem}\label{lemliouville}
Let $\delta\in(0,1/2)$ be as in~\eqref{asspars} and let $L>0$ be arbitrary. If $u$ is a classical stationary solution of~\eqref{eqL} such that $0\le u\le\delta$ in $\R$, then $u\equiv 0$ in $\R$. Similarly, if $u$ is a classical stationary solution of~\eqref{eqL} such that $1-\delta\le u\le1$ in $\R$, then $u\equiv 1$ in $\R$.
\end{lem}

\begin{proof}
We only prove the first assertion, since the second one is similar. Let $u$ be a classical steady state of~\eqref{eqL} such that $0\le u\le\delta$. From~\eqref{asspars}, the function $a_Lu'$ is nondecreasing. Assume by contradiction that $u$ is not constant in $\R$. Then there exists~$x_0\in\R$ such that $u'(x_0)\neq0$. If~$u'(x_0)>0$, then $a_L(x_0)u'(x_0)\le a_L(x)u'(x)$ for all~$x_0\le x$ and $\inf_{x\ge x_0}u'(x)>0$, contradicting the boundedness of~$u$. Similarly, if $u'(x_0)<0$, then $\sup_{x\le x_0}u'(x)<0$, which is impossible too. Finally, $u$ is a constant, between $0$ and $\delta$, and assumption~\eqref{asspars} yields $u\equiv 0$ in~$\R$.
\end{proof}

\begin{lem}\label{lemspeeds}
If $(L_n)_{n\in\N}$ is a sequence in $E$ such that $0<\inf_{n\in\N}L_n\le\sup_{n\in\N}L_n<+\infty$, then the sequence $(c_n)_{n\in\N}$ of the front speeds associated with~\eqref{eqL} and the periods $L_n$ is bounded.
\end{lem}

\begin{proof}
Assume first by contradiction that, up to extraction of a subsequence, one has $0\!<\!c_n\!\to\!+\infty$ and $L_n\to L\in(0,+\infty)$ as $n\to+\infty$. For each $n\in\N$, let $u_n(t,x)=\phi_{L_n}(x-c_nt,x/L_n)$ be a pulsating front associated with~\eqref{eqL} and the period $L_n$. By Theorem~\ref{thqual}, each function $u_n$ is increasing in $t$. Since $u_n(t,\cdot)\to0$ as $t\to-\infty$ and $u_n(t,\cdot)\to1$ as $t\to+\infty$, locally uniformly in~$\R$, there is by continuity a unique $t_n\in\R$ such that
\be\label{normun}
\max_{[0,L_n]}u_n(t_n,\cdot)=\delta.
\ee
By standard parabolic estimates the functions $(t,x)\mapsto u_n(t+t_n,x)$ converge in~$C^{1,2}_{loc}(\R^2)$, up to extraction of a subsequence, to a classical solution $0\le u(t,x)\le 1$ of~\eqref{eqL} such that $\max_{[0,L]}u(0,\cdot)=\delta$ and $u$ is nondecreasing with respect to $t$. Furthermore, since
\be\label{unpuls}
u_n\Big(t+t_n+\frac{L_n}{c_n},x+L_n\Big)=u_n(t+t_n,x)\hbox{ for all }(t,x)\in\R^2\hbox{ and }n\in\N,
\ee
one infers that $u(t,x+L)=u(t,x)$ for all $(t,x)\in\R^2$. In other words, $u$ is $L$-periodic in~$x$. By monotonicity in $t$ and from standard parabolic estimates, one has $u(t,x)\to u^-(x)$ as $t~\!\!\to~\!\!-~\!\!\infty$ uniformly in $x\in\R$, where $0\le u^-\le1$ is an $L$-periodic steady state of~\eqref{eqL} such that~$\max_{[0,L]}u^-(\cdot)\le\max_{[0,L]}u(0,\cdot)=\delta$, whence $u^-\le\delta$ in $\R$ by $L$-periodicity. Lemma~\ref{lemliouville} implies that $u^-=0$ in~$\R$. As a consequence, there is $t_0<0$ such that $u(t_0,\cdot)\le\delta/2$ in $\R$ and, since $\delta/2$ is a supersolution of~\eqref{eqL}, it follows necessarily that $u(t,\cdot)\le\delta/2$ in $\R$ for all $t\ge t_0$, contradicting in particular~$\max_{[0,L]}u(0,\cdot)=\delta$.\par
Lastly, if there is a subsequence such that $0>c_n\to-\infty$ and $L_n\to L\in(0,+\infty)$, one reaches a similar contradiction by changing the normalization condition~\eqref{normun} into
\be\label{normun2}
\min_{[0,L_n]}u_n(\tau_n,\cdot)=1-\delta,
\ee
with $\tau_n\in\R$. The proof of Lemma~\ref{lemspeeds} is thereby complete.
\end{proof}

\begin{proof}[Proof of Theorem~\ref{thE2}]
Let $(L_n)_{n\in\N}$ be a sequence in $E$ such that $L_n\to L\in\partial  E\cap(0,+\infty)$ as~$n\to+\infty$. Thus, for each $n\in\N$, equation~\eqref{eqL} with the period $L_n$ admits
a pulsating front~$u_n(t,x)=\phi_{L_n}(x-c_nt,x/L_n)$ with speed $c_n\neq0$. It follows from Lemma~\ref{lemspeeds} that, up to extraction of a subsequence, there is $c\in\R$ such that $c_n\to c$ as $n\to+\infty$. According to the sign of $c$, four cases may occur.

{\it Case (i): $c>0$.} In that case, by Theorem~\ref{thqual}, there holds necessarily $c_n>0$ for each~$n\in\N$ and $\int_0^1\overline{f}(u)du>0$. Furthermore, each function $u_n$ is increasing in $t$. As in the proof of Lemma~\ref{lemspeeds}, for each $n\in\N$, there is a unique $t_n\in\R$ such that the normalization condition~\eqref{normun} holds. By standard parabolic estimates, up to extraction of a subsequence, the functions $(t,x)\mapsto u_n(t+t_n,x)$ converge in~$C^{1,2}_{loc}(\R^2)$ to a classical solution $0\le u(t,x)\le 1$ of~\eqref{eqL} such that $\max_{[0,L]}u(0,\cdot)=\delta$ and~$u$ is nondecreasing with respect to $t$. The strong maximum principle implies that~$0\!<\!u(t,x)\!<\!1$ for all~$(t,x)\in\R^2$. Furthermore, by passing to the limit as $n\to+\infty$ in~\eqref{unpuls}, one infers that
\be\label{upuls2}
u\Big(t+\frac{L}{c},x+L\Big)=u(t,x)\hbox{ for all }(t,x)\in\R^2,
\ee
that is, $u$ can be written as
$$u(t,x)=\phi\Big(x-ct,\frac{x}{L}\Big),$$
where $\phi(\xi,y)=u((Ly-\xi)/c,Ly)$ is~$1$-periodic in $y$. Furthermore, by monotonicity in $t$ and standard parabolic estimates, it follows by passing to the limit as $t\to\pm\infty$ in~\eqref{upuls2} that there exist two $L$-periodic steady states~$0\le\tilde{u}(x)\le \bar{u}(x)\le 1$ such that
$$u(t,x)\to\tilde{u}(x)\hbox{ as }t\to-\infty\ \hbox{ and }\ u(t,x)\to\bar{u}(x)\hbox{ as }t\to+\infty,\ \hbox{ locally uniformly in }x\in\R.$$
The steady state $\tilde{u}\ge0$ is $L$-periodic and satisfies~$\max_{[0,L]}\tilde{u}(\cdot)\le\max_{[0,L]}u(0,\cdot)=\delta$, whence~$\tilde{u}=0$ by Lemma~\ref{lemliouville}. In other words,~$\phi(+\infty,\cdot)=0$. Furthermore, the steady state~$\bar{u}(x)=\phi(-\infty,x/L)\ (\ge u(t,x)>0)$ cannot be equal to $1$, otherwise~$u$ would be a pulsating front (connecting~$0$ and~$1$) with the speed $c\neq0$ and~$L$ would then belong to~$E$, contradicting Theorem~\ref{thE1} and the assumption $L\in\partial E\cap(0,+\infty)$. Therefore, from the strong elliptic maximum principle, $0<\bar{u}<1$ is a non-trivial $L$-periodic steady state of~\eqref{eqL}. Notice also, from the strong parabolic maximum principle, that the nonnegative function~$u_t$ is positive in $\R^2$, whence $\phi(\xi,y)$ is decreasing in~$\xi$ and $0<\phi(\xi,y)<\bar{u}(Ly)$ for all $(\xi,y)\in\R^2$.

Let us now show that $\bar{u}$ is semistable. Suppose on the contrary that $\bar{u}$ is unstable, that is~$\lambda_1(L,\bar{u})>0$, where $\lambda_1(L,\bar{u})$ is the principal eigenvalue of~\eqref{prineigen} corresponding to the steady state~$\bar{u}$. By Definition~\ref{semistable}, one can choose $R$ large enough such that $\lambda_{1,R}(L,\bar{u})>0$, where~$\lambda_{1,R}(L,\bar{u})$ is the principal eigenvalue in~\eqref{truncted}, associated with a positive principal eigenfunction~$\psi_R$. For any $\epsilon >0$, define $v_{\epsilon}$ as in~\eqref{defveps}. Since $u(t,x)<\bar{u}(x)$ for all $(t,x)\in\R^2$ and both functions are continuous, one can choose $\epsilon >0$ so small that $u(0,x)<v_{\epsilon}(x)$ for all~$x\in\R$ and, as in~\eqref{supersol},~$v_{\epsilon}$ is a supersolution of the elliptic equation associated with~\eqref{eqL}. It then follows from the parabolic maximum principle that $u(t,x)<v_{\epsilon}(x)$ for all $t\ge0$ and $x\in\R$. Hence, $\bar{u}(x)\le v_{\epsilon}(x)$ for all $x\in\R$, which is clearly impossible. Finally, $0<\bar{u}<1$ is a semistable $L$-periodic steady state of~\eqref{eqL}.

Now, instead of the normalization~\eqref{normun}, one can choose $\tau_n\in\R$ such that~\eqref{normun2} holds. Since $\delta<1-\delta$ and each function $u_n$ is increasing in $t$, one infers that $t_n<\tau_n$. As above, up to extraction of a subsequence, the functions $(t,x)\mapsto u_n(t+\tau_n,x)\ (>u_n(t+t_n,x))$ converge in~$C^{1,2}_{loc}(\R^2)$ to a classical solution $0<v(t,x)<1$ of~\eqref{eqL} such that $\min_{[0,L]}v(0,\cdot)=1-\delta$,~$v$ is nondecreasing with respect to $t$, $v\ge u$ in $\R^2$, and $v$ satisfies~\eqref{upuls2} that is $v$ can be written as~$v(t,x)=\psi(x-ct,x/L)$, where $\psi(\xi,y)=v((Ly-\xi)/c,Ly)$ is~$1$-periodic in $y$. Furthermore, there are two $L$-periodic steady states~$0\le\bar{v}(x)\le \tilde{v}(x)\le 1$ such that $v(t,x)\to\bar{v}(x)$ as $t\to-\infty$ and $v(t,x)\to\tilde{v}(x)$ as $t\to+\infty$ locally uniformly in $x\in\R$. The steady state $\tilde{v}\le1$ is $L$-periodic and satisfies~$\min_{[0,L]}\tilde{v}(\cdot)\ge\min_{[0,L]}v(0,\cdot)=1-\delta$, whence $\tilde{v}=1$ by Lemma~\ref{lemliouville}. In other words,~$\psi(-\infty,\cdot)=1$. In particular, since $\bar{u}(x)=u(+\infty,x)<1=v(+\infty,x)$ and $u\le v$, it follows then from the strong maximum principle that
$$u(t,x)<v(t,x)\hbox{ for all }(t,x)\in\R^2.$$
Lastly, the steady state~$\bar{v}(x)=\psi(+\infty,x/L)\ (\le v(t,x)<1)$ cannot be equal to $0$, otherwise $v$ would be a pulsating front (connecting~$0$ and~$1$) with the speed $c\neq0$, contradicting Theorem~\ref{thE1} and the assumption $L\in\partial E\cap(0,+\infty)$. Therefore, $0<\bar{v}<1$ is a non-trivial $L$-periodic steady state of~\eqref{eqL}. Notice that, in this case, $\bar{v}$ may not be semistable in general, the maximum principle leading to no obvious contradiction if $\bar{v}$ were assumed to be unstable.

{\it Case (ii): $c<0$.} The analysis is similar to that done in case~(i), but now $\int_0^1\overline{f}(u)du<0$ and the functions $u_n$ and their limits $u$ and $v$ are nonincreasing in $t$. One has $\lim_{t\to+\infty}u(t,x)=0<\bar{u}(x)=\lim_{t\to-\infty}u(t,x)<1$ and $0<\bar{v}(x)=\lim_{t\to+\infty}v(t,x)<1=\lim_{t\to-\infty}v(t,x)$. Lastly, the functions $\bar{u}$ and $\bar{v}$ are $L$-periodic steady states of~\eqref{eqL} and $\bar{v}$ is semistable.

{\it Case (iii): $c=0$ and $\int_0^1\overline{f}(u)du>0$.} Hence, $c_n>0$ for all~$n\in\N$, by Theorem~\ref{thqual}. Let~$t_n\in\R$ be as in~\eqref{normun}. Up to extraction of a subsequence, the functions $(t,x)\mapsto u_n(t+t_n,x)$ converge in~$C^{1,2}_{loc}(\R^2)$ to a classical solution $0<u_{\infty}(t,x)<1$ of~\eqref{eqL} such that $\max_{[0,L]}u_{\infty}(0,\cdot)=\delta$ and~$u_{\infty}$ is nondecreasing with respect to $t$. It also follows that $u_{\infty}(t,x)\to u(x)$ as $t\to+\infty$ locally uniformly in $x\in\R$, where $0\le u\le 1$ is a steady state of~\eqref{eqL}. One has $\max_{[0,L]}u\ge\max_{[0,L]}u_{\infty}(0,\cdot)=\delta$, whence $u>0$ in $\R$ from the strong maximum principle. Furthermore, for every $x\in\R$ and $t\in\R$, since $u_n$ is increasing in time, there holds
$$u_n(t+t_n,x+L_n)=u_n\Big(t+t_n-\frac{L_n}{c_n},x\Big)<u_n(t+t_n,x),$$
whence $u_{\infty}(t,x+L)\le u_{\infty}(t,x)$ and
\be\label{Lmonotone}
u(x+L)\le u(x)\ \hbox{ for all }x\in\R.
\ee

Now, for any fixed $t\in\R$ and $x>L$, and for all $n$ large enough, one has $t-L_n/c_n\le 0$ and there is $k_n\in\N$ such that $k_n\ge1$ and $k_nL_n\le x\le(k_n+1)L_n$, whence $t-k_nL_n/c_n\le 0$ and
$$u_n(t+t_n,x)=u_n\Big(t+t_n-\frac{k_nL_n}{c_n},x-k_nL_n\Big)\le u_n(t_n,x-k_nL_n)\le\max_{[0,L_n]}u_n(t_n,\cdot)=\delta$$
by~\eqref{normun} and the monotonicity of $u_n$ with respect to $t$. Therefore, $u_{\infty}(t,x)\le\delta$ for all $t\in\R$ and~$x>L$ (and $x\ge L$ by continuity), whence $u(x)\le\delta$ for all $x\ge L$. In particular, the strong maximum principle yields $u<1$ in $\R$. Furthermore, by~\eqref{Lmonotone} and standard elliptic estimates, the functions $x\mapsto u(x+kL)$ (with parameter $k\in\N$) converge decreasingly as $k\to+\infty$ in $C^2_{loc}(\R)$ to an $L$-periodic steady state $\underline{u}(x)$ of~\eqref{eqL} such that $0\le\underline{u}(x)\le\delta$ for all $x\in\R$. Lemma~\ref{lemliouville} implies that $\underline{u}\equiv0$ in $\R$, whence $u(x)\to0$ as $x\to+\infty$. Since $u$ is (strictly) positive in $\R$, it then follows from~\eqref{Lmonotone} and the strong maximum principle that
\be\label{strict}
u(x+L)<u(x)\ \hbox{ for all }x\in\R.
\ee
Similarly, the functions $x\mapsto u(x-kL)$ (with parameter $k\in\N$) converge increasingly as $k\to+\infty$ in~$C^2_{loc}(\R)$ to an $L$-periodic steady state $\bar{u}(x)$ of~\eqref{eqL} such that $0<\bar{u}(x)\le1$ in $\R$. In particular,
\be\label{ubaru}
u(x)-\bar{u}(x)\to0\ \hbox{ as }x\to-\infty.
\ee
Notice that~\eqref{strict} and~\eqref{ubaru}, together with the $L$-periodicity of $\bar{u}$, imply that $u(x)<\bar{u}(x)$ for all $x\in\R$. We also point out that, if $\bar{u}$ is equal to $1$, then $u$ is a stationary front (connecting~$0$ and~$1$).

Let us now prove that, whether $\bar{u}$ be equal to $1$ or less than $1$, both steady states $0<u(x)<1$ and $0<\bar{u}(x)\le 1$ are semistable. We will actually first prove that $u$ is semistable, and $\bar{u}$ will then immediately be semistable too by~\eqref{ubaru}. So, assume first by contradiction that $u$ is unstable. By Definition~\ref{semistable}, there is $R>0$ large enough such that $R>L$ and $\lambda_{1,R}(L,u)>0$, where $\lambda_{1,R}(L,u)$ denotes the principal eigenvalue of~\eqref{truncted} in $[-R,R]$, associated with a principal eigenfunction $\psi$. As in~\eqref{supersol}, there is then $\epsilon_0>0$ such that, for all $\epsilon\in(0,\epsilon_0]$, the function
$$v_{\epsilon}(x)=\left\{\baa{ll}u(x)-\epsilon\psi(x) & \hbox{if }|x|<R,\vspace{3pt}\\
u(x) & \hbox{if }|x|\ge R,\eaa\right.$$
satisfies $0<v_{\epsilon}\le u$ in $\R$ and is a supersolution of the elliptic equation associated with~\eqref{eqL}. On the other hand, since $u_{\infty}(t,x)\le u(x)$ for all $(t,x)\in\R^2$, the strong parabolic maximum principle implies that either $u_{\infty}(t,x)=u(x)$ for all $(t,x)\in\R^2$, or $u_{\infty}(t,x)<u(x)$ for all $(t,x)\in\R^2$. In the latter case, by continuity, there would be $\epsilon\in(0,\epsilon_0]$ such that $u_{\infty}(0,x)\le v_{\epsilon}(x)$ for all~$x\in\R$, whence $u_{\infty}(t,x)\le v_{\epsilon}(x)$ for all $t\ge0$ and $x\in\R$ from the maximum principle. By passing to the limit as $t\to+\infty$, one would infer that $u\le v_{\epsilon}$ in $\R^2$, which is clearly impossible. Therefore,~$u_{\infty}(t,x)=u(x)$ for all $(t,x)\in\R^2$. Now, since $0<v_{\epsilon_0}<1$ is continuous and since, for each $n\in\N$, $u_n$ is continuous and increasing in $t$ with $u_n(-\infty,\cdot)=0$ and $u_n(+\infty,\cdot)=1$, there is a unique $t'_n\in\R$ such that $u_n(t'_n,\cdot)\le v_{\epsilon_0}$ in $[-R,R]$ with equality somewhere in $[-R,R]$, that is
$$\max_{[-R,R]}\big(u_n(t'_n,\cdot)-v_{\epsilon_0}\big)=0.$$
Since $R>L$ and $v_{\epsilon_0}-u=-\epsilon_0\psi$ is continuous and negative in $(-R,R)$, one has $L_n<R$ for $n$ large enough, and there is $\eta>0$ such that
$$\max_{[0,L_n]}u_n(t'_n,\cdot)\le\max_{[0,L_n]}v_{\epsilon_0}<\max_{[0,L_n]}u\,-\,\eta=\max_{[0,L_n]}u_{\infty}(0,\cdot)\,-\,\eta,$$
whence $\max_{[0,L_n]}u_n(t'_n,\cdot)<\max_{[0,L_n]}u_n(t_n,\cdot)$ for $n$ large enough. Therefore, $t'_n<t_n$ for $n$ large enough and the functions $(t,x)\mapsto u_n(t+t'_n,x)$ converge in~$C^{1,2}_{loc}(\R^2)$, up to extraction of a subsequence, to a classical solution $\tilde{u}_{\infty}$ of~\eqref{eqL} such that $0\le\tilde{u}_{\infty}(t,x)\le u_{\infty}(t,x)=u(x)<1$ for all $(t,x)\in\R^2$ and $\max_{[-R,R]}\big(\tilde{u}_{\infty}(0,\cdot)-v_{\epsilon_0}\big)=0$. In particular, $\tilde{u}_{\infty}(0,\cdot)\le v_{\epsilon_0}$ in $\R$ with equality somewhere in $[-R,R]$. The maximum principle yields $\tilde{u}_{\infty}(t,x)\le v_{\epsilon_0}(x)$ for all $t\ge0$ and $x\in\R$ and since $\tilde{u}_{\infty}$ is nondecreasing in $t$, the function $\tilde{u}(x)=\lim_{t\to+\infty}\tilde{u}_{\infty}(t,x)$ is a classical steady state of~\eqref{eqL} such that
$$\tilde{u}\le v_{\epsilon_0}\hbox{ in }\R$$
with equality at a point $x_0\in[-R,R]$. If $|x_0|=R$, one has $\tilde{u}\le u$ in $\R$ with equality at $x_0$, whence~$\tilde{u}\equiv u$ in $\R$ by the strong maximum principle which is clearly impossible since $\tilde{u}\le v_{\epsilon_0}<u$ in~$(-R,R)$. If $|x_0|<R$, then $\tilde{u}\equiv v_{\epsilon_0}$ in $(-R,R)$ from the strong maximum principle, whence $\tilde{u}(\pm R)=v_{\epsilon_0}(\pm R)=u(\pm R)$ by continuity, which is again impossible. Finally, one has reached a contradiction and one has shown that~$0<u<1$ is a semistable steady state of~\eqref{eqL}.

Let us now conclude that the $L$-periodic steady state $0<\bar{u}\le1$ is also semistable. Remember that $0<u<\bar{u}$ in $\R$. Actually, if $\bar{u}$ were unstable, then the arguments of Step~1 of the proof of Lemma~\ref{noexstate} would imply that $\sup_{x\in\R}\big(u(x)-\bar{u}(x)\big)<0$. This is clearly impossible by~\eqref{ubaru}. Therefore, $\bar{u}$ is semistable.

{\it Case (iv): $c=0$ and $\int_0^1\overline{f}(u)du<0$.} Here, $c_n<0$. For each $n\in\N$, let $\tau_n$ be the unique real number such that~\eqref{normun2} holds. As in case~(iii), up to extraction of a subsequence, the functions $(t,x)\mapsto u_n(t+\tau_n,x)$ converge in $C^{1,2}_{loc}(\R^2)$ to a classical solution $0<v_{\infty}(t,x)<1$ of~\eqref{eqL} such that $\min_{[0,L]}v_{\infty}(0,\cdot)=1-\delta$ and $v_{\infty}$ is nonincreasing with respect to $t$. Therefore, $v_{\infty}(t,x)\to v(x)$ as $t\to+\infty$ locally uniformly in $x\in\R$, where $0\le v<1$ is a steady state of~\eqref{eqL}. As in case~(iii), there holds $v_{\infty}(t,x+L)\le v_{\infty}(t,x)$ and $v(x+L)\le v(x)$ for all $(t,x)\in\R^2$, while $v_{\infty}(t,x)\ge1-\delta$ for all $t\in\R$ and $x\le0$. As a consequence, $v(x)\ge1-\delta$ for all $x\le0$ and $v(x)\to1$ as $x\to-\infty$ by using Lemma~\ref{lemliouville}. Hence $v(x+L)<v(x)$ for all $x\in\R$ by the strong maximum principle. Lastly, there is an $L$-periodic steady state $0\le\bar{v}(x)<1$ of~\eqref{eqL} such that $\bar{v}(x)<v(x)<1$ for all $x\in\R$ and $v(x)-\bar{v}(x)\to0$ as $x\to+\infty$. The semistability of~$v$ and~$\bar{v}$ can then be proved as it was done in case~(iii) for $u$ and $\bar{u}$. The proof of Theorem~\ref{thE2} is thereby complete.
\end{proof}

%%%%%%%%%%%%%%%%%%%%%%%%%%%%%%%%%%%%%%%%
%%%%%%%%%%%%%%%%%%%%%%%%%%%%%%%%%%%%%%%%

\SE{Exponential stability of pulsating fronts}\label{sec4}

In this section, we first prove Theorem~\ref{gstability} on the exponential stability of the non-stationary pulsating fronts of~\eqref{eqL}. The proof is divided into two parts. In the first part, in Section~\ref{sec41}, we present a dynamical systems approach to the global stability of the fronts. Namely, the solution of the Cauchy problem~\eqref{inieqL} with an initial value satisfying~\eqref{initialv} converges at large time to a translate of the pulsating front. The uniqueness results stated in Theorem~\ref{thqual} can then be viewed as consequences of this global stability. In the second part, in Section~\ref{sec42}, by using spectral analysis we show that this convergence admits an exponential rate that is independent of the initial values. Lastly, in Section~\ref{sec43}, initial conditions satisfying assumptions of the type~\eqref{initialv2} are considered and Theorem~\ref{gstability2} is proved. For the sake of simplicity, throughout this section, even if it means rescaling the variables and renormalizing the reaction, one assumes that $L=1$ and that equation~\eqref{eqL} admits a pulsating front $U(t,x)=\phi(x-ct,x)$ with a nonzero speed $c$. Without loss of generality, as explained in Sections~\ref{sec2} and~\ref{sec3}, one can assume that $c>0$.

%%%%%%%%%%%%%%%%%%%%%%%%%%%%%%%%%%%%%%%%

\subsection{Global stability of pulsating fronts}\label{secgstability}\label{sec41}

Consider the moving coordinates
\begin{equation}\label{movcoor}
(\xi,t)=(x-ct,t),
\end{equation}
and write the solution of~\eqref{eqL} as $v(t,\xi)=u(t,x)$, so that $v(t,\xi)$ satisfies the following~$T$-periodic parabolic equation with $T=1/c$:
\begin{equation}\label{changeq}
v_t=(a(\xi+ct)v_{\xi})_{\xi}+cv_{\xi}+f(\xi+ct,v).
\end{equation}
Clearly,  the assumption \eqref{asspars} implies that $0$ and $1$ are two stable $T$-periodic solutions of~\eqref{changeq}. Note that for any $\tau\in\R$,
\begin{equation}\label{varmono}
V^{\tau}(t,\xi):=\phi(\xi+\tau,\xi+ct)
\end{equation}
is also a $T$-periodic solution of~\eqref{changeq}. Let $P$ be the Poincar\'{e} map of the $T$-periodic equation~\eqref{changeq}, that is,
\begin{equation}\label{poin}
P(g)=v(T,\cdot;g),
\end{equation}
where $v(t,\xi;g)$ is the unique solution of the Cauchy problem of ~\eqref{changeq} with initial condition $v(0,\cdot;g)=g\in C(\R,[0,1])$. Throughout this section, we denote $ \|\cdot\|=\|\cdot\|_{L^{\infty}(\R)}$. It easily follows that $0$, $1$ and $V^{\tau}(0,\cdot)$ (for any $\tau\in\R$) are fixed points of $P$ in $C(\R,[0,1])$. Since~$\phi(\xi,x)$ is decreasing in $\xi$ by Theorem~\ref{thqual}, there holds $V^{\tau_1}(0,\xi)>V^{\tau_2}(0,\xi)$ for all $\tau_1<\tau_2$ and~$\xi\in\R$. Hence, the set $\{V^{\tau}(0,\cdot)\,\big|\,\tau\in\R\}$ is totally ordered in $C(\R,[0,1])$. In order to prove the global stability of the pulsating front $U(t,x)=\phi(x-ct,x)$ with phase shift in time, we will apply the following convergence theorem to the Poincar\'{e} map $P$.

\begin{lem}{\sc (\cite[Theorem 2.2.4]{zh})}\label{convlem}
Let $C$ be a closed and ordered convex subset of an ordred Banach space $E$ and let $F:C\to C$ be a continuous and monotone map. Assume that there exists an increasing homeomorphism $h$ from $[0,1]$ onto a subset of $C$ such that
\begin{itemize}
\item[(1)] for each $s\in[0,1]$, $h(s)$ is a stable fixed point of $F$;
\item[(2)] for each $x\in[h(0),h(1)]_E=\big\{x\in E\ |\ h(0)\le_E x\le_E h(1)\big\}\subset C$, the forward orbit $\gamma^+(x)=\big\{F^n(x)\,|\,n\in\N\big\}$ is precompact;
\item[(3)] if $\omega(x)>_Eh(s_0)$ for some $x\in [h(0),h(1)]_E$ and $s_0\in[0,1)$, then there exists $s_1\in(s_0,1)$ such that $\omega(x)\geq_E h(s_1)$. Here $\omega(x)=\cap_{k\in\N}\overline{\gamma^+(F^k(x))}$ denotes the $\omega$-limit set of $\{x\}$
    for $F$.
\end{itemize}
Then  for any precompact forward orbit $\gamma^+(y)$ of $F$ in $C$ with $\omega(y)\cap [h(0),h(1)]_E \neq \emptyset$, there is $s^*\in[0,1]$ such that $\omega(y)=\{h(s^*)\}$.
\end{lem}

This abstract convergence result and its continuous-time analog were used, respectively, to prove the global attractiveness and uniqueness of bistable traveling waves for two classes of time-periodic reaction-diffusion equations in \cite{zh,zz} and an autonomous reaction-diffusion system in~\cite{xzh}. Here we should point out that the arguments used there are dependent on the property that both planar and time-periodic traveling fronts are monotone in the spatial variable. For equation~\eqref{eqL}, a pulsating front $\phi(x-ct,x)$ is no longer monotone in $x$ in general, but it is in the first variable $\xi=x-ct$. Thus, to make use of this monotonicity, we introduced the new variable~$\tau$ in~\eqref{varmono}. In what follows, we provide a series of lemmas to verify that the strategy in~\cite{zh} works for the functions $V^{\tau}(0,\cdot)$.

\begin{lem}\label{vagres}
For any $g\in C(\R,[0,1])$ satisfying~\eqref{initialv} and any $\epsilon>0$, there exist some integers $\tilde{k}=\tilde{k}(g,\epsilon)$ and $\tilde{m}=\tilde{m}(g,\epsilon)$ such that $V^{\tilde{k}}(0,\xi)-\epsilon\leq v(\tilde{m}T,\xi;g)\leq V^{-\tilde{k}}(0,\xi)+\epsilon$ for all $\xi\in\R$.
\end{lem}

\begin{proof}
Let us first set a few notations. Recall that $\delta\in(0,1/2)$ and $\gamma>0$ are given in~\eqref{asspars}. Denote $w_1^+(t)=1+(1-\delta)\me^{-\gamma t}$ and $w_2^+(t)=\delta\me^{-\gamma t}$ for $t\ge0$, and notice that $(w_1^+)' (t)=\gamma(1-w_1^+(t))$, $(w_2^+)'(t) = -\gamma w_2^+(t)$ and $0<w^+_2(t)<w^+_1(t)\le2-\delta$ for all $t\ge0$. Set $\eta(\xi)=(1+\tanh(-\xi/2))/2$ for $\xi\in\R$ and notice that $\eta'=-\eta(1-\eta)$ and $\eta''=\eta(1-\eta)(1-2\eta)$ in $\R$. Furthermore, since the function~$f(x,u)$ is of class $C^{1,1}$ in~$u$ uniformly for $x\in\R$, there exists $K>0$ such that
\begin{equation}\label{lipcons}
|f(x,s_1)-f(x,s_2)|+|\partial_uf(x,s_1)-\partial_uf(x,s_2)|\leq K|s_1-s_2|\ \hbox{ for all }x,\,s_1,\,s_2\in\R.
\end{equation}
Lastly, define $c^+=c-\| a\|-\|a'\|-2K$ and
$$w^+(t,\xi)=w_1^+(t)\,\eta(\xi+c^+t)+w_2^+(t)\,(1-\eta(\xi+c^+t))\ \hbox{ for }t\geq 0\hbox{ and }\xi\in\R.$$

Let now $g\in C(\R,[0,1])$ satisfy~\eqref{initialv}. There is $\xi_0\in\N$ such that $g(\xi+\xi_0)\leq \delta$ for all~$\xi\geq 0$. Without loss of generality, even if it means working with the shifted function $g(\cdot+\xi_0)$, one can assume that $\xi_0=0$. Next, one shows that $w^+$ is a supersolution of~\eqref{changeq}. To do so, observe first that $v(0,\xi;g)=g(\xi)\le w^+(0,\xi)$ for all $\xi\in\R$. On the other hand, for all $t>0$ and $\xi\in\R$,
\begin{equation*}
\begin{split}
\mathcal{L}(w^+):&=w^+_t-cw^+_{\xi}-(a(\xi+ct)w^+_{\xi})_{\xi}-f(\xi+ct,w^+)\vspace{3pt}\\
&=\gamma(1-w^+_1)\eta-\gamma w^+_2(1-\eta)-f(\xi+ct,w^+)+c^+(w^+_1\eta'-w^+_2\eta')\vspace{3pt}\\
&\quad - c(w^+_1\eta'-w^+_2\eta')-a'(\xi+ct)(w^+_1\eta'-w^+_2\eta')-a(\xi+ct)(w^+_1\eta''-w^+_2\eta''),
\end{split}
\end{equation*}
where $\eta$, $\eta'$ and $\eta''$ are taken at $\xi+ct$, while $w^+_1$ and $w^+_2$ are taken at $t$. Remember that $f(x,u)=\partial_uf(x,1)\,(u-1)$ for all $u\ge1$ and $x\in\R$, with $\partial_uf(x,1)\le-\gamma$. It follows from~\eqref{asspars} and~\eqref{lipcons} that
\begin{equation*}
\begin{split}
&\gamma(1-w^+_1)\eta-\gamma w^+_2(1-\eta)-f(\xi+ct,w^+)\vspace{3pt}\\
\geq &\ f(\xi+ct,w^+_1)\eta+ f(\xi+ct,w^+_2)(1-\eta)-f(\xi+ct,w^+)\eta-f(\xi+ct,w^+)(1-\eta)\vspace{3pt}\\
= & \ \big(\partial_uf(\xi+ct,\vartheta_1w^+_1+(1-\vartheta_1)w^+)-\partial_uf(\xi+ct,\vartheta_2w^+_2+(1-\vartheta_2)w^+)\big)(w^+_1-w^+_2)\eta(1-\eta)\vspace{3pt}\\
\geq &\ -K(w^+_1-w^+_2)^2\eta(1-\eta),
\end{split}
\end{equation*}
where $\vartheta_1=\vartheta_1(t,\xi),\,\vartheta_2=\vartheta_2(t,\xi)\in[0,1]$. Then, owing to the definitions of $\eta$ and $c^+$, one has
\begin{equation*}
\begin{split}
\mathcal{L}(w^+)&\,\geq -K(w^+_1-w^+_2)^2\eta(1-\eta)+(-c^++c+a'(\xi+ct))(w^+_1-w^+_2)\eta(1-\eta)\vspace{3pt}\\
&\quad\,-a(\xi+ct)(w^+_1-w^+_2)\eta(1-\eta)(1-2\eta)\vspace{3pt}\\
&\,= \big(2K-K(w^+_1-w^+_2)\big)(w^+_1-w^+_2)\eta(1-\eta)+\big(\|a'\|+a'(\xi+ct)\big)(w^+_1-w^+_2)\eta(1-\eta)\vspace{3pt}\\
&\quad\,+\big(\|a\|-a(\xi+ct)(1-2\eta)\big)(w^+_1-w^+_2)\eta(1-\eta)\vspace{3pt}\\
& \, \geq 0
\end{split}
\end{equation*}
in $(0,+\infty)\times \R$. The parabolic maximum principle implies that $v(t,\xi;g)\leq w^+(t,\xi)$ for all $(t,\xi)\in[0,+\infty)\times \R$.

Finally, let $\epsilon>0$ be any positive real number. There exist an integer $\tilde{m}=\tilde{m}(g,\epsilon)$ and a positive real number $C=C(g,\epsilon)$ such that $1<w^+_1(\tilde{m}T)\leq 1+\epsilon/2$, $0< w^+_2(\tilde{m}T)\leq \epsilon/2$ and~$(w^+_1(\tilde{m}T)-w^+_2(\tilde{m}T))\eta(\xi+c^+\tilde{m}T)\leq\epsilon/2$ for all $\xi\geq C$. Thus, it follows that, for all $\xi\ge C$,
$$v(\tilde{m}T,\xi;g)\leq w^+(\tilde{m}T,\xi)=(w^+_1(\tilde{m}T)-w^+_2(\tilde{m}T))\eta(\xi+c^+\tilde{m}T)+w^+_2(\tilde{m}T)\leq \epsilon \leq V^0(0,\xi)+\epsilon.$$
In the case where $\xi<C$, since $\lim_{s\to+\infty}V^0(0,\xi-s)=1$ uniformly for $\xi<C$, there exists an integer $\tilde{k}=\tilde{k}(g,\epsilon)$ such that $v(\tilde{m}T,\xi;g)\leq 1 \leq V^0(0,\xi-\tilde{k})+\epsilon=V^{-\tilde{k}}(0,\xi)+\epsilon$ for all $\xi<C$. Since $V^0(0,\xi)=\phi(\xi,\xi)\le\phi(\xi-\tilde{k},\xi)=V^{-\tilde{k}}(0,\xi)$ for all $\xi\in\R$ by Theorem~\ref{thqual}, one finally gets that $v(\tilde{m}T,\xi;g)\leq V^{-\tilde{k}}(0,\xi)+\epsilon$ for all $\xi\in\R$.

Similarly, even if it means increasing the integers $\tilde{m}$ and $\tilde{k}$, one can show that $V^{\tilde{k}}(0,\xi)-\epsilon\leq v(\tilde{m}T,\xi;g)$ for all $\xi\in\R$, by constructing an analogous subsolution of equation~\eqref{changeq}. More precisely, letting $w^-_1(t)=1-\delta\me^{-\gamma t}$, $w^-_2(t)=-(1+\delta)\me^{-\gamma t}$ and $c^-=c+\| a\|+\|a'\|+2K$, the function $w^-(t,\xi)=w^-_1(t)\eta(\xi+c^-t)+w^-_2(t)(1-\eta(\xi+c^-t))$ defined in $[0,+\infty)\times\R$ is a subsolution of~\eqref{changeq} and one can conclude as in the previous paragraph. Therefore, the proof of Lemma~\ref{vagres} is complete.
\end{proof}

The above lemma reveals that a ``vaguely resembling wave initial condition" (i.e. $g$ satisfying~\eqref{initialv}) evolves into a ``resembling wave front" (i.e. close to $0$ and $1$ as $\xi\to\pm\infty$) after a certain time. Next by constructing similar super- and subsolutions of~\eqref{changeq} as in~\cite{fm,x4}, one shows that a ``resembling wave front" preserves its structure uniformly at later times.

\begin{lem}\label{resemwave}
(i) There exist some positive real numbers $\epsilon_0$ and $k_0$ such that for any $g\in C(\R,[0,1])$ satisfying $g\le V^{\tau_0}(0,\cdot)+\epsilon$ $($resp. $g\ge V^{\tau_0}(0,\cdot)-\epsilon)$ in $\R$ for some $\epsilon\in(0,\epsilon_0]$ and $\tau_0\in\R$, then $v(t,\xi;g)\leq V^{\tau_0-k_0\epsilon}(t,\xi)+\epsilon \me^{-\gamma t/2}$ $($resp. $v(t,\xi;g)\geq V^{\tau_0+k_0\epsilon}(t,\xi)-\epsilon \me^{-\gamma t/2})$ for all $t\ge0$ and $\xi\in\R$.\hfill\break
(ii) There exists a positive real number $k_1$ such that if $\|g-V^{\tau_0}(0,\cdot)\|\leq \epsilon$ for some $\epsilon\in(0,\epsilon_0]$ and~$\tau_0\in\R$, then $\| v(t,\cdot;g)-V^{\tau_0}(t,\cdot)\|\leq k_1\epsilon$ for all $t\geq 0$.
\end{lem}

\begin{proof}
(i) We only consider the case $g\le V^{\tau_0}(0,\cdot)+\epsilon$, since the case $g\ge V^{\tau_0}(0,\cdot)-\epsilon$ can be treated similarly. The general strategy can be described as follows. We set
$$ \overline{V}(t,\xi):=V^{\tau(t)}(t,\xi)+q(t)\ \hbox{ for }t\geq 0\hbox{ and }\xi\in\R,$$
where $\tau$ and $q$ are $C^1([0,+\infty))$ functions such that $\tau(0)=\tau_0$, $\tau'(t)<0$ for all $t\ge0$, $q(0)=\epsilon$ and~$0<q(t)\leq \epsilon$ for all $t\ge0$. Here $\epsilon>0$ will be chosen small enough and $\tau_0\in\R$ is arbitrary. The assumption $g\le V^{\tau_0}(0,\cdot)+\epsilon$ in $\R$ means that $g\le\overline{V}(0,\cdot)$ in $\R$. By choosing appropriate functions~$\tau(t)$ and $q(t)$, one will actually show that $\overline{V}(t,\xi)$ is a supersolution of equation~\eqref{changeq}.

To do so, let $\zeta(t,\xi)=\xi+\tau(t)$ and $x(t,\xi)=\xi+ct$, so that $\overline{V}(t,\xi)=\phi(\zeta(t,\xi),x(t,\xi))+q(t)$. To avoid any confusion, let us denote here $\partial_1\phi$ and~$\partial_2\phi$ the partial derivatives of $\phi$ with respect to its first and second arguments. A straightforward calculation gives, for all $t>0$ and $\xi\in\R$,
\begin{equation*}
\begin{split}
\mathcal{L}\overline{V}&\,=\overline{V}_t-c\overline{V}_{\xi}- (a(\xi+ct)\overline{V}_{\xi})_{\xi}-f(\xi+ct,\overline{V})\vspace{3pt}\\
&\,=\tau'(t)\partial_1\phi+c\partial_2\phi+q'(t)-c(\partial_1\phi\!+\!\partial_2\phi)-a'(x)(\partial_1\phi\!+\!\partial_2\phi)-a(x)(\partial_1\!+\!\partial_2)^2\phi-f(x,\phi\!+\!q(t))\vspace{3pt}\\
&\,=\tau'(t)\partial_1\phi+q'(t)-f(x,\phi+q(t))+f(x,\phi)\vspace{3pt}\\
& \ \ \ \ -\big(c\partial_1\phi+a'(x)(\partial_1\phi+\partial_2\phi)+a(x)(\partial_1+\partial_2)^2\phi+f(x,\phi)\big),
\end{split}
\end{equation*}
where $\phi$, $\partial_1\phi$ and $\partial_2\phi$ are taken at $(\zeta(t,\xi),x(t,\xi))$, while $x=x(t,\xi)$. Since $(\phi,c)$ is a pulsating front of equation~\eqref{eqL}, it follows that, for all $t>0$ and $\xi\in\R$,
\be\label{LbarV}
\mathcal{L}\overline{V}= \tau'(t)\partial_1\phi+q'(t)+f(x,\phi)-f(x,\phi+q(t)).\ee

Now, since $f(y,u)$ satisfies~\eqref{asspars} and is of class $C^{1,1}$ in $u$ uniformly in $y\in\R$, there exists a positive real number $\epsilon_0$ such that, if $(\epsilon,t,\xi)\in(0,\epsilon_0]\times[0,+\infty)\times\R$ and $\phi(\zeta(t,\xi),x(t,\xi))\in[0,\epsilon_0]\cup[1-\epsilon_0,1]$, then
\begin{equation}\label{resemwavein1}
f(x(t,\xi),\phi(\zeta(t,\xi),x(t,\xi)))-f(x(t,\xi),\phi(\zeta(t,\xi),x(t,\xi))\!+\!q(t))\geq\frac{\gamma q(t)}{2}.
\end{equation}
On the other hand, since the derivative $\partial_1\phi$ is continuous and negative in $\R\times\T$ from Theorem~\ref{thqual} and the strong parabolic maximum principle applied to the function~$U_t$, there is~$\beta>0$ such that, if $(\epsilon,t,\xi)\in(0,\epsilon_0]\times[0,+\infty)\times\R$ and $\phi(\zeta(t,\xi),x(t,\xi))\in[\epsilon_0,1-\epsilon_0]$, then
\begin{equation}\label{resemwavein3}
\partial_1\phi(\zeta(t,\xi),x(t,\xi))\leq-\beta .
\end{equation}

Given these positive parameters $\epsilon_0$ and $\beta$, let us now consider $\epsilon\in(0,\epsilon_0]$, $g\in C(\R,[0,1])$ and~$\tau_0\in\R$ such that $g\le V^{\tau_0}(0,\cdot)+\epsilon$ in $\R$, and let us then choose $q(t)$ and $\tau(t)$ so that
\begin{equation}\label{resemwavein4}
q(0)=\epsilon,\ q'(t)=-\frac{\gamma q(t)}{2}\hbox{ for all }t\ge0,\ \tau(0)=\tau_0\hbox{ and }\tau'(t)=-\frac{2K+\gamma}{2\beta}q(t)\hbox{ for all }t\ge0,
\end{equation}
that is $q(t)=\epsilon \me^{-\gamma t/2}$ and $\tau(t)=\tau_0-\epsilon(2K+\gamma)(1-\me^{-\gamma t/2})/(\gamma\beta)$ for all $t\ge0$. It is then easy to check from~(\ref{lipcons}-\ref{resemwavein4}) that $\mathcal{L}\overline{V}\geq 0$ for all $t>0$ and $\xi\in\R$, while $g\le\overline{V}(0,\cdot)$ in $\R$. That is, $\overline{V}$ is a supersolution of~\eqref{changeq}. As a consequence, by the comparison principle, one infers that
$$v(t,\xi;g)\le V^{\tau(t)}(t,\xi)+q(t)\ \hbox{ for all }t\geq 0\hbox{ and }\xi\in\R.  $$
Since $V^s(t,\xi)$ is decreasing in $s\in\R$, letting $k_0=(2K+\gamma)/(\gamma\beta)>0$, one has $\tau(t)\ge\tau_0-k_0\epsilon$ for all $t\ge0$, whence $v(t,\xi;g)\leq V^{\tau_0-k_0\epsilon}(t,\xi)+\epsilon \me^{-\gamma t/2}$ for all $t\ge 0$ and $\xi\in\R$. The proof of assertion~(i) is thereby complete.

(ii) Since $V^{\tau_0}(0,\xi)-\epsilon\leq g(\xi)\leq V^{\tau_0}(0,\xi)+\epsilon$ for all $\xi\in \R$, one sees from assertion~(i) that $V^{\tau_0+k_0\epsilon}(t,\xi)-\epsilon \me^{-\gamma t/2} \leq v(t,\xi;g)\leq V^{\tau_0-k_0\epsilon}(t,\xi)+\epsilon \me^{-\gamma t/2}$ for all $t\ge0$ and $\xi\in \R$. Therefore,
\begin{equation*}
\begin{split}
v(t,\xi;g)-V^{\tau_0}(t,\xi)&\leq V^{\tau_0-k_0\epsilon}(t,\xi)-V^{\tau_0}(t,\xi)+\epsilon \me^{-\gamma t/2}\vspace{3pt}\\
&=\phi(\xi+\tau_0-k_0\epsilon,\xi+ct)-\phi(\xi+\tau_0,\xi+ct)+\epsilon \me^{-\gamma t/2},
\end{split}
\end{equation*}
and, similarly, $v(t,\xi;g)-V^{\tau_0}(t,\xi)\geq \phi(\xi+\tau_0+k_0\epsilon,\xi+ct)-\phi(\xi+\tau_0,\xi+ct)-\epsilon \me^{-\gamma t/2}$ for all~$t\geq 0$ and~$\xi\in\R$. Since $\phi$ is globally Lipschitz-continuous, there exists a positive real number~$k_1$ depending only on~$\phi$ and~$k_0$ such that $\|v(t,\cdot;g)-V^{\tau_0}(t,\cdot)\|\le k_1\epsilon$ for all $t\ge0$. The proof of Lemma~\ref{resemwave} is thereby complete.
\end{proof}

Lemma~\ref{resemwave} implies in particular that for each $\tau\in\R$, $V^{\tau}(0,\cdot)$ is a Lyapunov stable fixed point of the Poincar\'e map $P$ defined in~\eqref{poin}. Now we are in a position to employ Lemma~\ref{convlem} to prove that the non-stationary pulsating fronts of~\eqref{eqL} are globally stable. Due to Lemmas~\ref{vagres}-\ref{resemwave} and the monotonicity of $V^{\tau}(t,\xi)$ in $\tau$, the arguments in \cite[Theorem 10.2.1]{zh} carry through with minor modification. For the sake of completeness, we include the details below.

\begin{pro}\label{nonlocals}
Let $\phi(x-ct,x)$ be a pulsating front of equation~\eqref{eqL} with $L=1$, $c\neq 0$ and let $u(t,x;g)$ be the solution of~\eqref{inieqL} with $u(0,\cdot;g)=g\in L^{\infty}(\R,[0,1])$ satisfying~\eqref{initialv}. Then there exists $\tau_g\in\R$ such that
$$\sup_{x\in\R}\big|u(t,x;g)-U(t+\tau_g,x)\big|=\sup_{x\in\R}\big|u(t,x;g)-\phi(x-ct-c\tau_g,x)\big|\to0\ \hbox{ as }t\to+\infty,$$
that is, $\|v(t,\cdot;g)-V^{-c\tau_g}(t,\cdot)\|\to0$ as $t\to+\infty$.
\end{pro}

\begin{proof}
Recall that $P:C(\R,[0,1])\to C(\R,[0,1])$ is the Poincar\'e map defined in~\eqref{poin}, that is, $P(\varphi)=v(T,\cdot;\varphi)$ for all $\varphi\in C(\R,[0,1])$. We are going to apply Lemma~\ref{convlem} with $E=C(\R,\R)$ endowed with the norm $\|\ \|=\|\ \|_{L^{\infty}(\R)}$ and the standard order between real-valued functions, $C=C(\R,[0,1])$ and $F=P$. We first notice that $C$ is a closed and ordered convex subset of $E$ and that the map $P:C\to C$ is monotone and continuous, by the parabolic maximum principle.

Consider now any $g\in L^{\infty}(\R,[0,1])$ satisfying~\eqref{initialv}. Since $v(t,\cdot;g)\in C(\R,[0,1])$ for all $t>0$ and since $\limsup_{\xi\to-\infty}(1-v(t,\xi;g))\le\big(\limsup_{\xi\to-\infty}(1-g(\xi))\big)\me^{Kt}$ and $\limsup_{\xi\to+\infty}v(t,\xi;g)\le\big(\limsup_{\xi\to+\infty}g(\xi)\big)\me^{Kt}$ for all $t>0$ with $K$ as in~\eqref{lipcons}, one can assume without loss of generality that $g\in C(\R,[0,1])$ satisfies~\eqref{initialv}, even if it means replacing $g$ by $v(\rho,\cdot;g)$ for some small enough~$\rho>0$. Under the notations $\epsilon_0$ and $k_0$ of Lemma~\ref{resemwave}~(i), fix any $\epsilon\in(0,\epsilon_0]$. Lemmas~\ref{vagres} and~\ref{resemwave} provide the existence of some integers $\tilde{k}=\tilde{k}(g,\epsilon)$ and $\tilde{m}=\tilde{m}(g,\epsilon)$ such that
\begin{equation}\label{nonlocalseq1}
V^{\tilde{k}+k_0\epsilon}(t,\xi)-\epsilon\me^{-\gamma t/2}\leq v(\tilde{m}T+t,\xi;g)\leq V^{-\tilde{k}-k_0\epsilon}(t,\xi)+\epsilon\me^{-\gamma t/2}\ \hbox{ for all }t\geq 0\hbox{ and }\xi\in\R.
\end{equation}
Since the sequence $(P^n(g))_{n\ge1}=(v(nT,\cdot;g))_{n\ge1}$ is bounded in $C^1(\R,\R)$ by standard parabolic estimates and since $V^{\tau}(t,-\infty)=1$ and $V^{\tau}(t,+\infty)=0$ uniformly in $t\in\R$, it then follows from~\eqref{nonlocalseq1} that the forward orbit $\gamma^+(g)=\{P^n(g)\ |\ n\in\N\}$ is precompact in $C(\R,[0,1])$. Hence, the $\omega$-limit set $\omega(g)$ is nonempty, compact and invariant by $P$. Letting $p=\tilde{k}+k_0\epsilon>0$ and~$t=nT$ in~\eqref{nonlocalseq1}, one then concludes that
$$\omega(g)\subset I:=\big\{\varphi\in C(\R,[0,1])\ |\ V^p(0,\cdot)\leq \varphi \leq V^{-p}(0,\cdot)\hbox{ in }\R\big\}\ \subset\ C. $$
Now define $h(s)=V^{p-2ps}(0,\cdot)$ for $s\in[0,1]$. Thus, $I=[h(0),h(1)]_E$ and $h$ is an increasing homeomorphism from $[0,1]$ onto a subset of $C$. On the other hand, by Lemma~\ref{resemwave}~(ii), $h(s)$ is a stable fixed point for $P$ for each $s\in[0,1]$. As done for $g$, one can also observe that for each~$\varphi\in I=[h(0),h(1)]_E$, the forward orbit $\gamma^+(\varphi)$ is included in $I$ and precompact in~$C(\R,[0,1])$.

Let us finally check the last condition in Lemma~\ref{convlem}. To do so, assume that $h(s_0)<_E\omega(\varphi_0)$ for some $s_0\in[0,1)$ and $\varphi_0\in I$. In other words, one has $h(s_0)<_E\varphi$, that is, $V^{p-2ps_0}(0,\cdot)\le\varphi$ in $\R$ and $V^{p-2ps_0}(0,\cdot)\not\equiv\varphi$, for all $\varphi\in\omega(\varphi_0)$. By the strong maximum principle, there holds $V^{p-2ps_0}(t,\xi)<v(t,\xi;\varphi)$ for all $t>0$ and $\xi\in\R$, whence $V^{p-2ps_0}(T,\cdot)<P(\varphi)$ in $\R$. By the $T$-periodicity of $V^{p-2ps_0}$ and the invariance of $\omega(\varphi_0)$ for $P$, it then follows that
\be\label{varphi0}
V^{p-2ps_0}(0,\cdot)<\varphi\hbox{ in }\R\hbox{ for all }\varphi\in \omega(\varphi_0).
\ee
Furthermore, since $V^{\tau}(0,\xi)=\phi(\xi+\tau,\xi)$ and $\lim_{\xi\to\pm\infty}\partial_{\xi}\phi(\xi,x)=0$ uniformly in $x\in\R$, there is~$M>0$ such that
\be\label{defM}
0<\tilde{\delta}:=\sup_{s,s'\in[0,3/2],\,s\neq s',\,|\xi|\geq M}\!\!\frac{|V^{p-2ps}(0,\xi)-V^{p-2ps'}(0,\xi)|}{|s-s'|}\le\min\Big(\frac{2\epsilon_0}{5},\frac{2p}{5k_0}\Big).
\ee
On the other hand, since $\omega(\varphi_0)$ is compact, it follows then from~\eqref{varphi0} that there exists a real number $s_1\in (s_0,1)$ such that $V^{p-p(3s_1-s_0)}(0,\xi)<\varphi(\xi)$ for all $\xi\in[-M,M]$ and $\varphi\in\omega(\varphi_0)$.

 Let $\varphi\in\omega(\varphi_0)$ be given. Then there is a sequence $n_j\to+\infty$ such that $\|P^{n_j}(\varphi_0)-\varphi\|\to0$ as~$j\to+\infty$. Let $n_k\in\N$ be such that $\|P^{n_k}(\varphi_0)-\varphi \|\leq\tilde{\delta}\,(s_1-s_0).$ Since $\varphi(\xi)- V^{p-p(3s_1-s_0)}(0,\xi)>0$ for all $\xi\in [-M,M]$ and $\varphi(\xi)- V^{p-p(3s_1-s_0)}(0,\xi)>V^{p-2ps_0}(0,\xi)-V^{p-p(3s_1-s_0)}(\xi)$ for all $\xi\in\R$, one infers that
\begin{equation*}
\begin{split}
P^{n_k}(\varphi_0)(\xi)-V^{p-p(3s_1-s_0)}(0,\xi)& \geq -\|P^{n_k}(\varphi_0)-\varphi \| + \varphi(\xi)-V^{p-p(3s_1-s_0)}(0,\xi)\vspace{3pt}\\
& \geq -\tilde{\delta}\,(s_1-s_0)-\sup_{|\xi|\geq M}|V^{p-2ps_0}(0,\xi)- V^{p-p(3s_1-s_0)}(0,\xi)|\\
& \displaystyle\geq-\frac{5\tilde{\delta}\,(s_1-s_0)}{2}\,\ge-\,\epsilon_0
\end{split}
\end{equation*}
for all $\xi\in\R$, by~\eqref{defM}. Thus, by Lemma~\ref{resemwave}~(i), there holds
$$v(t,\xi;P^{n_k}(\varphi_0))\geq V^{p-p(3s_1-s_0)+5k_0\tilde{\delta}(s_1-s_0)/2}(t,\xi)-\frac{5\tilde{\delta}(s_1-s_0)}{2}\,\me^{-\gamma t/2}\hbox{ for all }t>0\hbox{ and }\xi\in\R.$$
Letting $t=(n_j-n_k)T$ and $j\to+\infty$ yields $\varphi\geq V^{p-p(3s_1-s_0)+5k_0\tilde{\delta}(s_1-s_0)/2}(0,\cdot)\geq V^{p-2ps_1}(0,\cdot)$ in $\R$ since $p-p(3s_1-s_0)+5k_0\tilde{\delta}(s_1-s_0)/2\le p-2ps_1$ by~\eqref{defM} and $V^{\tau}$ is decreasing with respect to~$\tau$. Hence, $\omega(\varphi_0)\geq_EV^{p-2ps_1}(0,\cdot)=h(s_1)$.

Finally, since $\emptyset\neq\omega(g)\subset I=[h(0),h(1)]_E$, it follows from Lemma~\ref{convlem} that there is $s_g\in[0,1]$ such that $\omega(g)=\{h(s_g)\}=\{V^{p-2ps_g}(0,\cdot)\}$. That is, $\lim_{n\to+\infty}\|P^{n}(g)-V^{p-2ps_g}(0,\cdot)\|=0$. It then follows from Lemma~\ref{resemwave}~(ii) that $\lim_{t\to+\infty} \|v(t,\cdot;g)-V^{p-2ps_g}(t,\cdot)  \|=0$. Since $v(t,x-ct;g)=u(t,x;g)$ and $V^{p-2ps_g}(t,x-ct)=\phi(x-ct+p-2ps_g,x)$ for all $t\ge0$ and $x\in\R$, one gets the conclusion of Proposition~\ref{nonlocals} with $\tau_g=(2ps_g-p)/c$.
\end{proof}

Once the global stability of pulsating fronts is established, the uniqueness results in Theorem~\ref{thqual} is an easy corollary.

\begin{proof}[Proof of the uniqueness result in Theorem~\ref{thqual}]
First of all, as in Proposition~\ref{nonlocals}, one can assume without loss of generality that $L=1$. One has to show that if $U(t,x)=\phi(x-ct,x)$ and $\tilde{U}(t,x)=\tilde{\phi}(x-\tilde{c}t,x)$ are two pulsating fronts of equation~\eqref{eqL} with $c\neq0$, then $\tilde{c}=c$, and $U$ and~$\tilde{U}$ are equal up to shift in time. Since $\tilde{U}(0,\cdot)\in C(\R,[0,1])$ and $\tilde{U}(0,-\infty)=1$, $\tilde{U}(0,+\infty)=0$, Proposition~\ref{nonlocals} yields the existence of $\tau\in\R$ such that $\sup_{x\in\R}|\tilde{U}(t,x)-\phi(x-ct-c\tau,x)|\to0$ as~$t\to+\infty$, whence
\begin{equation}\label{unin}
\sup_{\xi\in\R}\big|\phi(\xi-c\tau,\xi+ct)- \tilde{\phi}(\xi+(c-\tilde{c})t,\xi+ct)\big|\to0\hbox{ as }t\to+\infty.
\end{equation}
Remember that $\tilde{\phi}(+\infty,x)=\phi(+\infty,x)=0$, $\tilde{\phi}(-\infty,x)=\phi(-\infty,x)=1$ uniformly in~$x\in\R$, that~$\phi(\xi,x)$ is continuous in $\R^2$, $1$-periodic in $x$ and decreasing in $\xi$, whence $0<\min_{x\in\R}\phi(\xi,x)\le\max_{x\in\R}\phi(\xi,x)<1$ for all $\xi\in\R$. Therefore, if $c\neq\tilde{c}$, by fixing $\xi\in\R$ and letting $t\to+\infty$ in~\eqref{unin}, one derives a contradiction. Thus, $c=\tilde{c}$ and one infers from~\eqref{unin} that, for every~$(\xi,x)\in\R^2$ and~$k\in\N$,
$$\phi(\xi-c\tau,x)-\tilde{\phi}(\xi,x)=\phi\Big(\xi-c\tau,\xi+c\big(\frac{x-\xi}{c}+\frac{k}{|c|}\big)\Big)-\tilde{\phi}\Big(\xi,\xi+c\big(\frac{x-\xi}{c}+\frac{k}{|c|}\big)\Big)\mathop{\longrightarrow}_{k\to+\infty}0.$$
Thus, $\phi(\xi-c\tau,x)=\tilde{\phi}(\xi,x)$ for all $(\xi,x)\in\R^2$, that is, $\tilde{U}(t,x)=U(t+\tau,x)$ for all $(t,x)\in\R^2$. Hence, the proof of the uniqueness result in Theorem~\ref{thqual} is complete.
\end{proof}

%%%%%%%%%%%%%%%%%%%%%%%%%%%%%%%%%%%%%%%%

\subsection{Exponential stability of pulsating fronts}\label{sec42}

Here, we prove Theorem~\ref{gstability} by using the general theory of exponential stability of invariant manifolds with asymptotic phase. This theory was first established by Henry~\cite{henry} in the context of reaction-diffusion equations, and then applied to bistable time-periodic equations in~\cite{abc}. Once again, one assumes that $L=1$ and that~\eqref{eqL} admits a pulsating front $U(t,x)=\phi(x-ct,x)$ with a speed $c>0$. Consider the moving coordinates $(\xi,t)$ defined in~\eqref{movcoor} and the resulting time-periodic parabolic equation~\eqref{changeq}. One uses the notations $T=1/c$ and $V^{\tau}$ given in~\eqref{varmono}, and $P:C(\R,[0,1])\to C(\R,[0,1])$ is the Poincar\'e map defined in~\eqref{poin}. Note that $\tilde{\mathcal{M}}:=\{V^{\tau}\,|\,\tau\in\R\}\subset C(\R^2,[0,1])$ is a one-dimensional manifold of special solutions of~\eqref{changeq}. By Proposition~\ref{nonlocals},~$\tilde{\mathcal{M}}$ attracts the solutions of the Cauchy problem
of ~\eqref{changeq} with initial values $g$ satisfying~\eqref{initialv}. In order to show that this convergence is also exponential in time, it is sufficient to prove the local exponential stability of the manifold
$$\mathcal{M}:=\{V^{\tau}(0,\cdot)\ |\ \tau\in\R\}\subset C(\R,[0,1]).$$
Notice that since each element in $\mathcal{M}$ is a fixed point of $P$, the manifold $\mathcal{M}$ is invariant
under ~$P$.

One is thus interested in the linearization of $P$ around the points of $\mathcal{M}$. Without loss of generality, consider the point $V_0:=V^0(0,\cdot)$. Let $\mathcal{X}= BUC(\R,\C)$ be the Banach space of bounded and uniformly continuous functions on $\R$, with the norm $\|\ \|=\|\ \|_{L^{\infty}(\R,\C)}$. It is easy to check that the derivative of $P$ at $V_0$ is given by $P'(V_0)(w)=W(T,\cdot;w)$ for $w\in\mathcal{X}$, where~$W(t,\xi;w)$ obeys
$$\left\{\baa{l}
W_t=(a(\xi+ct)W_{\xi})_{\xi}+cW_{\xi}+\partial_uf(\xi+ct,V^0(t,\xi))W,\quad t>0,\,\,\xi\in\R,\vspace{3pt}\\
W(0,\xi)=w(\xi),\quad\xi\in\R.
\eaa\right.
$$
Denote $W^0(t,\xi):=\partial_{\tau}V^{\tau}(t,\xi)_{|\tau=0}=\partial_1\phi(\xi,\xi+ct)$ and notice that $1$ is an eigenvalue of $P'(V_0)$ with eigenfunction $w_0:=W^0(0,\cdot)\in\mathcal{X}$. As in \cite{abc}, in order to prove the local exponential stability of~$\mathcal{M}$, we need to show that $1$ is algebraically simple, and that the rest of the spectrum is contained in a disk of radius strictly less than $1$. Namely, we prove the following two lemmas.

\begin{lem}\label{spectral}
The value $1$ is an algebraically simple eigenvalue of $P'(V_0)$ with eigenfunction $w_0$ and, if $\lambda\in\C$ is an eigenvalue of $P'(V_0)$ with eigenfunction $w\not\in\C w_0$, then $|\lambda |<1.$
\end{lem}

\begin{proof}
Assume that $\lambda\in\C^*$ is an eigenvalue of $P'(V_0)$ with eigenfunction $w\in \mathcal{X}$. Let $\mu\in\C$ be such that $\me^{\mu T}=1/\lambda$ and set $h(t,\xi)=\me^{\mu t} W(t,\xi;w)$ for $t\ge0$ and $\xi\in\R$. The function $h$ satisfies
\begin{equation}\label{lineigen2}
\left\{\baa{l}
h_t-(a(\xi+ct)h_{\xi})_{\xi}-ch_{\xi}-\partial_uf(\xi+ct,V^0(t,\xi))h=\mu h\ \hbox{ for all }t>0\hbox{ and }\xi\in\R,\vspace{3pt}\\
h(0,\xi)=h(T,\xi)\ \hbox{ for all }\xi\in\R.
\eaa\right.
\end{equation}
Hence, the eigenvalue problem $P'(V_0)(w)=\lambda w$ can be recast as the spectral problem~\eqref{lineigen2} in the space of time periodic functions $\big\{h\in C(\R^2,\C)\ |\ h(t,\cdot)=h(t+T,\cdot)$ for all $t\in\R\big\}$. Clearly,~$W^0$ solves~\eqref{lineigen2} with $\mu=0$. In addition, following similar arguments as in Step~3 of the proof of Lemma~\ref{degenerate} or in \cite[Lemma A.2]{abc}, one infers that if $h\not\in\C W^0$ solves~\eqref{lineigen2} with some $\mu\in\C$, then $Re(\mu)>0$. In other words, if~$w\not\in\C w_0$, then $|\lambda |<1.$

On the other hand, since $t=(x-\xi)/c$, setting $\psi(\xi,x)=h(t,\xi)$ in~\eqref{lineigen2} yields
\begin{equation*}
\left\{\baa{l}
\tilde{\partial}(a(x)\tilde{\partial}\psi)+c\partial_{\xi}\psi +\partial_uf\big(x,\phi(\xi,x)\big)\psi=-\mu \psi\ \hbox{ for all }(\xi,x)\in\R^2,\vspace{3pt}\\
\psi(\xi,x+1)=\psi(\xi,x)\ \hbox{ for all }(\xi,x)\in\R^2,
\eaa\right.
\end{equation*}
where $\tilde{\partial}=\partial_{\xi}+\partial_x$. But, from Lemma~\ref{degenerate} and standard parabolic estimates, it follows that $0$ is an algebraically simple eigenvalue of the operator $H_1$ defined in~\eqref{defHL0} with $L_0=1$ in the space of spatially periodic functions $\mathcal{E}:=\big\{\psi\in BUC(\R^2,\C)\ |\ \psi(\xi,x)=\psi(\xi,x+1)$ for all $(\xi,x)\in\R^2\big\}$.
\footnote{We point out that for any $\psi\in\mathcal{E}$ such that $H_1(\psi)=0$ in the sense of distributions, one can assume without loss of generality that $\psi$ is real valued, and the function $\psi$ is a classical solution by parabolic estimates. Hence $\psi(\pm\infty,x)=0$ uniformly in $x\in\R$ because of~\eqref{asspars} and $\psi$ and its derivatives decay exponentially as $\xi\to\pm\infty$, as in Lemma~\ref{exponential}. Thus, $\psi\in\mathcal{D}_1$.} Thus,~$1$ is an algebraically simple eigenvalue of $P'(V_0)$ and the proof of Lemma~\ref{spectral} is complete.
\end{proof}

\begin{lem}\label{essspectral}
The essential spectrum of $P'(V_0)$ is contained in the disk $\big\{\lambda\in \mathbb{C}\ |\ |\lambda|\leq \me^{-\gamma T/2}\big\}$. Thus, if $\lambda$ is in the spectrum of $P'(V_0)$ and $|\lambda|>\me^{-\gamma T/2}$, then $\lambda$ is an eigenvalue, and for any $r>\me^{-\gamma T/2}$, there are only finitely many eigenvalues of $P'(V_0)$ in $\{\lambda\in \mathbb{C}\ |\ |\lambda|\geq r\}$.
\end{lem}

\begin{proof}
Firstly,~\eqref{asspars} yields the existence of $N>0$ such that $\partial_uf(\xi+ct,V^0(t,\xi)) \leq -\gamma/2$ for all~$t\in\R$ and $|\xi|\geq N$. Let $\kappa_1(t,\xi)$ and $\kappa_2(t,\xi)$ be the continuous $T$-periodic functions defined by
$$\kappa_1(t,\xi)=\!\left\{\begin{array}{lll}
\!\!\partial_uf(\xi+ct,V^0(t,\xi)) & \!\!\hbox{if }\xi\leq -N,\vspace{3pt}\\
\!\!\partial_uf(-N+ct,V^0(t,-N)) & \!\!\hbox{if }\xi>-N,\end{array}\right.
\kappa_2(t,\xi)=\!\left\{\begin{array}{lll}
\!\!\partial_uf(\xi+ct,V^0(t,\xi)) & \!\!\hbox{if }\xi\geq N,\vspace{3pt}\\
\!\!\partial_uf(N+ct,V^0(t,N)) & \!\!\hbox{if }\xi<N.\end{array}\right.$$
Let $\varsigma\in C(\R,[0,1])$ be a function satisfying $\varsigma(\xi)=0$ for $\xi\leq -N$ and $\varsigma(\xi)=1$ for $\xi\geq N$, and define $\kappa(t,\xi)=(1-\varsigma(\xi))\kappa_1(t,\xi)+\varsigma(\xi)\kappa_2(t,\xi)$ for $(t,\xi)\in\R^2$. The function $\kappa$ is continuous in~$\R^2$, bounded, $T$-periodic with respect to $t$, and $\kappa\leq -\gamma/2$ in $\R^2$. Consider now the operator $\mathcal{K}$ defined on $\mathcal{X}$ by $\mathcal{K}(w)=\hat{W}(T,\xi;w)$ for $w\in\mathcal{X}$, where $\hat{W}(t,\xi;w)$ solves the equation
$$\left\{\baa{l}
\hat{W}_t=(a(\xi+ct)\hat{W}_{\xi})_{\xi}+c\hat{W}_{\xi} +\kappa(t,\xi)\hat{W},\quad t>0,\,\,\xi\in\R\vspace{3pt}\\
\hat{W}(0,\cdot;w)=w.
\eaa\right.$$
By the maximum principle, $\|\hat{W}(t,\cdot;w)\|\leq \me^{-\gamma t/2}\| w\|$ for all $t>0$, whence $\|\mathcal{K}(w)\|\leq \me^{-\gamma T/2}\| w\|$ for all $w\in\mathcal{X}$. Thus, the spectral radius of $\mathcal{K}$ is at most $\me^{-\gamma T/2}$.

Next, one shows that $\mathcal{K}-P'(V_0):\mathcal{X}\to\mathcal{X}$ is a compact operator. To do so, consider a bounded sequence $(w_n)_{n\in\N}$ in $\mathcal{X}$ and denote $\eta_n=(\mathcal{K}-P'(V_0))(w_n)$ for $n\in\N$. From standard parabolic estimates, up to extraction of some subsequence, the functions~$\eta_n$ converge to a function~$\eta_{\infty}$, locally uniformly as $n\to+\infty$, with $\eta_{\infty}\in\mathcal{X}$. To show the uniform convergence in $\R$, one needs to estimate the values of $\eta_n(\xi)$ for $|\xi|\geq N$ uniformly in $n\in\N$. Without loss of generality, one can assume that the functions $w_n$ (and $\eta_n$) are real-valued and one only considers the case~$\xi\le-N$ since the case~$\xi\ge N$ can be dealt with similarly. For each $n\in\N$, let $\Phi_n(t,\xi):=\hat{W}(t,\xi;w_n)-W(t,\xi;w_n)$ for~$t\ge0$ and $\xi\leq -N$. Owing to the definition of~$\kappa$, each function~$\Phi_n$ solves
\begin{equation}\label{eqlarxi}
\left\{\baa{l}
(\Phi_n)_t=(a(\xi+ct)(\Phi_n)_{\xi})_{\xi}+c(\Phi_n)_{\xi} +\kappa(t,\xi)\Phi_n,\ \ t>0,\ \xi\leq -N,\vspace{3pt}\\
\Phi_n(0,\xi)=0,\ \ \xi\leq -N.\eaa\right.
\end{equation}
As in Lemma~\ref{exponential}, for any $\nu\in\R$, the linear operator $\mathcal{T}_{\nu}$ defined on $\big\{\varphi\in C^2(\R)\ |\ \varphi=\varphi(\cdot+1)$ in~$\R\big\}$ by $\mathcal{T}_{\nu}[\varphi]:=(a\varphi')'+2\nu a\varphi'+(\nu a'+c\nu+a\nu^2-\gamma/2)\varphi$ admits a principal eigenvalue $\lambda_1(\nu)$, and there is $\nu_1>0$ such that $\lambda_1(\nu_1)=0$. Let $\varphi_{\nu_1}$ be a positive corresponding eigenfunction. Since~$\sup_{t\in[0,T],\,n\in\N}\|\Phi_n(t,\cdot)\|<+\infty$ by the parabolic maximum principle and the boundedness of the sequence $(w_n)_{n\in\N}$ in $\mathcal{X}$, there is a constant $A>0$ such that
$$A\me^{-\nu_1 N} \varphi_{\nu_1}(-N+ct)\geq|\Phi_n(t,-N)|\ \hbox{ for all }t\in[0,T]\hbox{ and }n\in\N. $$
An immediate computation shows that $A\me^{\nu_1 \xi}\varphi_{\nu_1}(\xi+ct)$ is a supersolution of~\eqref{eqlarxi} in~$[0,T]\times(-\infty,-N]$, whence $|\eta_n(\xi)|=|\Phi_n(T,\xi)|\leq A  \me^{\nu_1 \xi}\varphi_{\nu_1}(\xi+cT)$ for all $\xi\leq -N$ and $n\in\N$ by the maximum principle. Thus, $\eta_n(\xi)$ converges to $0$ as $\xi\to \pm\infty$ uniformly in $n\in\N$. Together with the local convergence of the functions $\eta_n$ to $\eta_{\infty}$, one concludes that $\|\eta_n-\eta_{\infty}\|\to0$ as $n\to+\infty$.

Finally, there are then finitely many eigenvalues of $P'(V_0)$ in $\{\lambda\in\C\ |\ |\lambda|\ge r\}$ for any $r>e^{-\gamma T/2}$ and the essential spectrum of $P'(V_0)$ is the same as that of $\mathcal{K}$, whence the radius of essential spectrum of $P'(V_0)$ is not larger than $\me^{-\gamma T/2}$.
\end{proof}

\begin{proof}[End of the proof of Theorem~\ref{gstability}]
Consider any $g\in L^{\infty}(\R,[0,1])$ satisfying~\eqref{initialv}. As in the proof of Proposition~\ref{nonlocals}, one can assume that $g\in C(\R,[0,1])$. By Proposition~\ref{nonlocals}, there is $\tau_g\in\R$ such that $\|u(t,\cdot;g)-U(t+\tau_g,\cdot)\|=\|v(t,\cdot;g)-V^{-c\tau_g}(t,\cdot)\|\to0$ as $t\to+\infty$. Without loss of generality, even if it means shifting $U$ in time, one can assume that $\tau_g=0$. From Lemmas~\ref{spectral} and~\ref{essspectral}, $1$ is an algebraically simple eigenvalue of $P'(V_0)$ and the rest of the spectrum of $P'(V_0)$ is contained in a disk with radius $\tilde{r}(P'(V_0))\in(0,1)$. Therefore, the manifold~$\mathcal{M}$ is locally exponentially stable with asymptotic phase for~$P$, see \cite[Chapter 9]{henry}. Since $P(V_0)\!=\!V_0\!=\!V^0(0,\cdot)$ and $\|P^n(g)-V_0\|=\|v(nT,\cdot;g)-V^0(nT,\cdot)\|\to0$ as~$n\to+\infty$, there are then some real numbers $\mu=-\ln((\tilde{r}(P'(V_0))+1)/2)>0$ independent of $g$, and~$C$ depending on $g$, such that
$$\|v(nT,\cdot;g)-V_0\|=\|P^n(g)-V_0\|\leq C\me^{-\mu Tn}\ \hbox{ for all }n\in\N.$$
Under the notations of Lemma~\ref{resemwave}, there is then $n_0\in\N$ such that $\|v(nT,\cdot;g)-V_0\|\leq C\me^{-\mu Tn}\le\epsilon_0$ for all $n\in\N$ with $n\ge n_0$, whence $\|v(nT+t,\cdot;g)-V^0(t,\cdot)\|\le k_1C\me^{-\mu Tn}$ for all $n\ge n_0$ and~$t\in[0,T]$. Since $V^0$ is $T$-periodic in $t$, one infers that
$$\|v(t,\cdot;g)-V^0(t,\cdot)\|\le k_1C\me^{\mu T}\me^{-\mu t}\ \hbox{ for all }t\ge n_0T.$$
Finally, since $0\le v(\cdot,\cdot;g),\,V^0\le 1$ in $[0,+\infty)\times\R$, there is a constant $C_g$ (depending on $g$ and on the parameters of~\eqref{eqL}) such that $\|v(t,\cdot;g)-V^0(t,\cdot)\|\le C_g\me^{-\mu t}$ for all $t\ge0$. Since $v(t,x-ct;g)=u(t,x;g)$ and $V^0(t,x-ct)=\phi(x-ct,x)=U(t,x)$, one concludes that $\|u(t,\cdot;g)-U(t,\cdot)\|\le C_g\me^{-\mu t}$ for all $t\ge0$ and the proof of Theorem~\ref{gstability} is thereby complete.
\end{proof}

%%%%%%%%%%%%%%%%%%%%%%%%%%%%%%%%%%%%%%%%

\subsection{Proof of Theorem~\ref{gstability2}}\label{sec43}

This section is devoted to the proof of Theorem~\ref{gstability2}. That is, we show that any solution of~\eqref{inieqL} with initial value satisfying~\eqref{initialv2} converges to a translate (in time) of the pulsating front of equation~\eqref{eqL}, and this convergence is exponential in time. Thanks to the assumption that all $L$-periodic stationary states of equation~\eqref{eqL} are unstable, we will prove that such solution of~\eqref{inieqL} will evolve into a ``front-like'' wave (i.e., satisfying~\eqref{initialv}) after a certain time. Then the conclusion of Theorem~\ref{gstability2} will follow easily from Theorem~\ref{gstability}.

\begin{proof}[Proof of Theorem~\ref{gstability2}]
Since the initial value $g$ satisfies~\eqref{initialv2}, there is $\sigma>0$ such that
\begin{equation}\label{initialg}
\liminf_{x\to-\infty}\big(g(x)-\bar{u}_-(x)\big)\geq 2\sigma,\quad \bar{u}_-+2\sigma< 1\hbox{ in }\R,
\end{equation}
and $\limsup_{x\to+\infty}\big(g(x)-\bar{u}_+(x)\big)\leq -2\sigma$ with $\bar{u}_+-2\sigma> 0$ in $\R$. Let $u^{\pm}(t,x)$ denote the solutions of the Cauchy problems
$$\left\{\baa{ll}
u^{\pm}_t=(a_L(x)u^{\pm}_x)_x+f_L(x,u^{\pm}), & t>0,\ x\in\R,\vspace{3pt}\\
u^{\pm}(0,x)=\bar{u}_{\pm}\mp\sigma, & x\in\R.\eaa\right.$$
Since by assumption $\bar{u}_{\pm}$ are unstable, it follows as in Step~2 of Lemma~\ref{noexstate} that $u^-(t,x)\to1$ and~$u^+(t,x)\to0$ as $t\to+\infty$ uniformly in $x\in\R$.

Next, one shows that $\liminf_{x\to-\infty} \big(u(t,x)-u^-(t,x)\big)>0$ for all $t>0$. Assume by contradiction that there are $t_0>0$ and a sequence $(x_n)_{n\in\N}$ with $\lim_{n\to+\infty}x_n=-\infty$ and
\begin{equation}\label{leftcon2}
\lim_{n\to+\infty} \big(u(t_0,x_n)-u^-(t_0,x_n)\big)\leq 0.
\end{equation}
By~\eqref{initialg} and since $g$ ranges in $[0,1]$, there is a $C^{2,\alpha}(\R,[0,1])$ function $\underline{u}_0$ such that $\underline{u}_0\le g$ in $\R$, $\underline{u}_0(x)=0$ for $x\gg1$ and $\underline{u}_0(x)=\bar{u}_-(x)+3\sigma/2$ for $x\ll-1$. Let $\underline{u}$ be the solution of the Cauchy problem~\eqref{inieqL} with initial condition $\underline{u}(0,\cdot)=\underline{u}_0$. The maximum principle yields $0\le\underline{u}(t,x)\le u(t,x)\le1$ for all $t>0$ and $x\in\R$. Write now $x_n=x_n'+x_n''$ with $x_n'\!\in\!L\Z$ and~$x_n''\in [0,L]$, and set $\underline{u}_n(t,x)=\underline{u}(t,x+x_n')$. Since $a_L$ and $f_L$ are $L$-periodic in~$x$, the functions~$\underline{u}_n$ obey~\eqref{eqL} for $t\ge0$ and $x\in\R$. Up to extraction of some subsequence, one can assume that~$x_n''\to x_{\infty}\in[0,L]$ as $n\to+\infty$ and that, from standard parabolic estimates,~$\underline{u}_n(t,x)\to\underline{u}_{\infty}(t,x)$ as~$n\to+\infty$ locally uniformly in $[0,+\infty)\times\R$, where $\underline{u}_{\infty}$ solves~\eqref{eqL}. In particular, for any $x\in\R$, $\underline{u}_{\infty}(0,x)=\lim_{n\to+\infty}\underline{u}_n(0,x)=\bar{u}_-(x)+3\sigma/2>u^-(0,x)$, whence $\underline{u}_{\infty}(t,x)>u^-(t,x)$ for all $t\ge0$ and $x\in\R$ by the maximum principle. On the other hand, since $\underline{u}\le u$ and $u^-$ is $L$-periodic in~$x$,~\eqref{leftcon2} yields
$$\underline{u}_{\infty}(t_0,x_{\infty})\!=\!\lim_{n\to+\infty}\underline{u}_n(t_0,x_n'')\!=\!\lim_{n\to+\infty}\underline{u}(t_0,x_n)\!\le\!\limsup_{n\to+\infty}u(t_0,x_n)\!\le\!\limsup_{n\to+\infty} u^-(t_0,x_n)\!=\!u^-(t_0,x_{\infty}).$$
This contradicts $\underline{u}_{\infty}(t_0,x_{\infty})>u^-(t_0,x_{\infty})$. Hence $\liminf_{x\to-\infty} \big(u(t,x)-u^-(t,x)\big)>0$ for every $t>0$. Similarly, there holds $\lim_{t\to+\infty}u^+(t,x)=0$ uniformly in $x\in\R$ and $\limsup_{x\to+\infty} \big(u(t,x)-u^+(t,x)\big)<0$ for every $t>0$. Thus, there is $T>0$ such that $\liminf_{x\to-\infty}u(T,x)>1-\delta$ and $\limsup_{x\to+\infty} u(T,x)<\delta$, where $\delta$ is the constant given in~\eqref{asspars}. By Theorem~\ref{gstability}, as applied to initial value $u(T,\cdot)$, the conclusion of Theorem~\ref{gstability2} follows.
\end{proof}

%%%%%%%%%%%%%%%%%%%%%%%%%%%%%%%%%%%%%%%%
%%%%%%%%%%%%%%%%%%%%%%%%%%%%%%%%%%%%%%%%

\SE{Appendix}\label{sec5}

The appendix is devoted to the proof of some technical auxiliary lemmas which were stated in Section~\ref{sec2} and used in the
 proofs of Theorems~\ref{thhomo} and~\ref{thE1}. We first begin with the proofs of the properties of the operators $M_{c,L}$ and $M_{c,0}$ stated in Lemmas~\ref{invertible},~\ref{continuem} and~\ref{continuem2}.

\begin{proof}[Proof of Lemma~\ref{invertible}]
Consider only the case $L\neq0$, since the case $L=0$ follows the same lines, and is actually simpler. Let $\beta>0$, $c>0$ and $L>0$ be fixed, and let $v$ be in the kernel of~$M_{c,L}$. Integra\-ting~$M_{c,L}(v)=0$ against $v$ in $L^2(\R\times\T)$ gives $\int_{\R\times \T} \big(a(y)(\tilde{\partial}_Lv)^2 +\beta v^2\big)=0$, whence~$v=0$. The adjoint operator $M_{c,L}^*$ of $M_{c,L}$ is defined as $M_{c,L}^*(v)=\tilde{\partial}_L(a\tilde{\partial}_Lv)-c\partial_{\xi}v-\beta v$ with domain~$\mathcal{D}_L$, in such a way that $(M_{c,L}^*(v),u)_{L^2(\R\times \T)}=(v,M_{c,L}(u))_{L^2(\R\times \T)}$ for any $u,\,v\in \mathcal{D}_L$. The same arguments yield that the kernel of the operator $M_{c,L}^*$ is reduced to $\{0\}$. Then in order to get the invertibility of the linear operator $M_{c,L}$, it is sufficient to show that the range $R(M_{c,L})$ is closed in $L^2(\R\times \T)$. To do so, let $M_{c,L}(v_n)=g_n$ with $v_n\in \mathcal{D}_L$ and $g_n\to g$ in $L^2(\R\times \T)$ as~$n\to+\infty$. Integrating~$M_{c,L}(v_n)=g_n$ against $v_n$ gives $\int_{\R\times \T} \big((a(y)(\tilde{\partial}_Lv_n)^2 +\beta v_n^2)\big)=-\int_{\R\times \T} g_nv_n$, whence
\begin{equation}\label{hnorm}
\int_{\R\times \T}a(y)\,(\tilde{\partial}_Lv_n)^2 + \frac{\beta}{2} v_n^2\leq \frac{1}{2\beta}\int_{\R\times \T}g_n^2.
\end{equation}
For any $h>0$, let us define the symmetric difference quotient in $\xi$-direction as follows: $D_hv_n(\xi,y)=(v_n(\xi+h,y)-v_n(\xi-h,y))/(2h)$ for $(\xi,y)\in\R\times\T$ and notice that $D_hv_n\in H^1(\R\times\T)$. Taking it as a test function in $M_{c,L}(v_n)=g_n$, one gets that
$$\int_{\R\times \T}a(y)\tilde{\partial}_Lv_n\tilde{\partial}_LD_hv_n-c\tilde{\partial}_Lv_nD_hv_n+\beta v_nD_hv_n=-\int_{\R\times \T} g_nD_hv_n.$$
Observe that $\int_{\R\times \T} a(y)(\tilde{\partial}_LD_hv_n)\tilde{\partial}_Lv_n=\int_{\R\times \T}\beta (D_hv_n)v_n=0$, and that $D_hv_n\rightharpoonup\partial_{\xi}v_n$ in $L^2(\R\times\T)$ weakly as $h\to 0^+$. Hence, by letting $h\to 0$, one concludes that
\begin{equation}\label{quetients}
\int_{\R\times \T} c(\partial_{\xi}v_n)^2=\int_{\R\times \T} g_n\partial_{\xi}v_n,
\end{equation}
and then, by the Cauchy-Schwarz inequality,
\begin{equation}\label{partials}
\|\partial_{\xi}v_n\|_{L^2(\R\times \T)}\leq \frac{\|g_n\|_{L^2(\R\times \T)}}{c}.
\end{equation}
The uniform $L^2$-estimate for $\partial_y v_n$ can be obtained from~\eqref{hnorm} and~\eqref{partials} as follows:
\begin{equation}\label{partialy}
\begin{split}
\|\partial_y v_n\|_{L^2(\R\times \T)} = \|L(\tilde{\partial}_Lv_n-\partial_{\xi}v_n)\|_{L^2(\R\times \T)} &\leq|L|\big(\|\tilde{\partial}_Lv_n\|_{L^2(\R\times\T)}+\|\partial_{\xi}v_n\|_{L^2(\R\times \T)}\big)\vspace{3pt}\\
& \leq |L|\|g_n\|_{L^2(\R\times \T)}\Big(\frac{1}{c}+\frac{1}{\sqrt{2\beta a^-}}\Big)
\end{split}
\end{equation}
with $a^-:=\min_{y\in \R} a(y)>0$. Hence, there is a subsequence $(v_{n_i})_{i\in\N}$ such that $v_{n_i}\rightharpoonup v $ in $H^1(\R\times \T)$ weakly as~$i\to+\infty$. This implies that $\int_{\R\times \T}a(y)\tilde{\partial}_Lv\tilde{\partial}_L\varphi\!-\!c(\partial_{\xi}v)\varphi\!+\!\beta v\varphi=-\int_{\R\times \T} g\varphi$ for all $\varphi\!\in\!H^1(\R\times\T)$. It follows that $\tilde{\partial}_L(a\tilde{\partial}_Lv) \in L^2(\R\times \T)$. It follows that $v\in\mathcal{D}_L$, $M_{c,L}(v)=g$ and~$g\in R(M_{c,L})$. Since the kernel of $M_{c,L}$ is trivial, the limit $v$ is unique and the whole sequence~$(v_n)_{n\in\N}$ converges to $v$ weakly in~$H^1(\R\times\T)$ as $n\to+\infty$. Finally, the calculations in~\eqref{hnorm},~\eqref{partials} and~\eqref{partialy} imply that
\begin{equation}\label{inverineq}
\|v\|_{H^1(\R\times \T)}^2\leq \Big(\frac{1}{\beta^2}+\frac{1}{c^2}+|L|^2\big(\frac{1}{c}+\frac{1}{\sqrt{2\beta a^-}}\big)^2\Big)\|g\|_{L^2(\R\times \T)}^2
\end{equation}
and the proof of Lemma~\ref{invertible} is thereby complete.
\end{proof}

\begin{proof}[Proof of Lemma~\ref{continuem}]
We only give the proof of the first assertion in~\eqref{convinverse}, the second one being obtained similarly and being actually simpler. For the proof, we first begin with the special case where the sequences $(c_n)_{n\in\N}$ and $(g_n)_{n\in\N}$ are constant, and we then deal with the general case.

{\it Step 1: the case $c_n=c$ and $g_n=g$.} Consider any sequence $(L_n)_{n\in\N}$ in $\R^*$ converging to $0$ and let $v_n=M_{c,L_n}^{-1}(g)\in H^1(\R\times\T)$. By the estimate~\eqref{inverineq} and the Sobolev injections, the functions~$v_n$ converge, up to extraction of a subsequence, strongly in $L^2_{loc}(\R\times \T)$ and weakly in~$H^1(\R\times \T)$, to some $v_0\in H^1(\R\times\T)$ as $n\to+\infty$. Note that $\partial_yv_0=0$ because of the estimate~\eqref{partialy}. Then the function $v_0$ can be viewed as an $H^1(\R)$ function and we can set $v_0'=\partial_{\xi}v_0$. For any $\phi\in H^2(\R)$, taking $\psi(\xi,y)=\phi(\xi)+L_n\chi(y)\phi'(\xi)\in H^1(\R\times\T)$ (where $\chi\in C^2(\R)$ solves~\eqref{dualcell}) as a test function in~$M_{c,L_n}(v_n)=g$ gives
$$\int_{\R\times \T} a\,((\chi'+1)\phi'+L_n\chi\phi'')\,\tilde{\partial}_{L_n}v_n+\int_{\R\times \T}
(-c\partial_{\xi}v_n+\beta v_n)(\phi+L_n\chi\phi')=-\int_{\R\times \T} g(\phi+L_n\chi\phi'),$$
where $a$, $\chi$ and $\chi'$ are evaluated at $y$, while $\phi$, $\phi'$ and $\phi''$ are evaluated at $\xi$. The second term of the left hand side clearly converges to $\int_{\R}(-cv_{0}'+\beta v_0)\phi$ as $n\to+\infty$, and so does the right hand side to $-\int_{\R\times \T} g\phi=-\int_{\R} \overline{g}\phi$. The first term can be written as
\begin{equation*}
\begin{split}
\int_{\R\times \T} a\,((\chi'+1)\phi'+L_n\chi\phi'')\,\tilde{\partial}_{L_n}v_n
&=\int_{\R\times \T}-v_n\tilde{\partial}_{L_n}(a\,(\chi'+1)\,\phi') +L_na\,\chi\,\phi''\tilde{\partial}_{L_n}v_n\vspace{3pt}\\
&=\int_{\R\times \T}a\,(\chi'+1)\,\phi'\partial_{\xi}v_n+(\partial_yv_n+L_n\partial_{\xi}v_n)\,a\,\chi\,\phi'',
\end{split}
\end{equation*}
which converges to $\int_{\R\times \T} v_0'a\,(\chi'+1)\,\phi'=a_H\int_{\R}v_0'\phi'$ as $n\to+\infty$, since $\partial_yv_0=0$. Hence,
$$\int_{\R}a_Hv_0'\phi'-cv_{0}'\phi+\beta v_0\phi=-\int_{\R} \overline{g}\phi,$$
for all $\phi\in H^2(\R)$ and then for all $\phi\in H^1(\R)$ by density.  This implies that $v_0\in H^2(\R)$ and that~$M_{c,0}(v_0)=\overline{g}$.

Next we show that $v_n$ converges to $v_0$ in $L^2(\R\times\T)$ as $n\to+\infty$ (the convergence has only been known in $L^2_{loc}(\R\times\T)$ so far). Let $\zeta_n(\xi,y)=v_n(\xi,y)-v_0(\xi)-L_n\chi(y)v_0'(\xi)$. By Lemma~\ref{invertible}, the sequence $(v_n)_{n\in\N}$ is bounded in $H^1(\R\times\T)$ and so is the sequence $(\zeta_n)_{n\in\N}$. Taking $\zeta_n$ as a test function in $M_{c,L_n}(v_n)=g$ gives $\int_{\R\times \T}a\,\tilde{\partial}_{L_n}v_n\tilde{\partial}_{L_n}\zeta_n-c\zeta_n\partial_{\xi}v_n+\beta v_n\zeta_n=-\int_{\R\times \T}g\,\zeta_n$. Observe that
\begin{equation*}
\begin{split}
\int_{\R\times \T} a\,\tilde{\partial}_{L_n}v_n\tilde{\partial}_{L_n}\zeta_n & =\int_{\R\times \T}a\,(\tilde{\partial}_{L_n}\zeta_n)^2+av_0'\tilde{\partial}_{L_n}\zeta_n+a\chi'v_0'\tilde{\partial}_{L_n}\zeta_n +L_na\chi v_0''\tilde{\partial}_{L_n}\zeta_n\vspace{3pt}\\
&=\int_{\R\times \T}a(\tilde{\partial}_{L_n}\zeta_n)^2+a_Hv_0'\tilde{\partial}_{L_n}\zeta_n+L_na\chi v_0''\tilde{\partial}_{L_n}\zeta_n\vspace{3pt}\\
&=\int_{\R\times \T}a(\tilde{\partial}_{L_n}\zeta_n)^2-a_Hv_0''\zeta_n+L_na\chi v_0''\partial_{\xi}\zeta_n-(a\chi)'v_0''\zeta_n,\\
\end{split}
\end{equation*}
and that $\int_{\R\times \T}c\zeta_n\partial_{\xi}v_n=\int_{\R\times \T}cv_0'\zeta_n+L_nc\chi v_0''\zeta_n$ since $\int_{\R\times \T} \zeta_n\tilde{\partial}_{L_n}\zeta_n=0$, while $\int_{\R\times \T}\beta v_n\zeta_n=\int_{\R\times \T}\beta\zeta_n^2+\beta v_0\zeta_n+L_n\beta\chi v_0'\zeta_n$. Since $M_{c,0}(v_0)=\overline{g}$, one obtains that
$$\baa{rcl}
\displaystyle\int_{\R\times \T}a\,(\tilde{\partial}_{L_n}\zeta_n)^2+\beta\zeta_n^2 & = & \displaystyle\int_{\R\times \T}(\overline{g}-g)\zeta_n+(a\chi)'v_0''\zeta_n-L_n(a\chi v_0''\partial_{\xi}\zeta_n-c\chi v_0''\zeta_n+\beta\chi v_0'\zeta_n)\vspace{3pt}\\
& = & \displaystyle\int_{\R\times \T}(\overline{g}-g)(v_n-v_0)+(a\chi)'v_0''(v_n-v_0)\vspace{3pt}\\
& & \displaystyle-L_n\int_{\R\times \T}a\chi v_0''\partial_{\xi}\zeta_n-c\chi v_0''\zeta_n+\beta\chi v_0'\zeta_n+(\overline{g}-g)\chi v_0'+(a\chi)'v_0''\chi v_0'.\eaa$$
Since $v_n\rightharpoonup v_0$ in $L^2(\R\times \T)$ weakly as $n\to+\infty$, and since the sequence $(\zeta_n)_{n\in\N}$ is bounded in~$H^1(\R\times \T)$ and $v_0\in H^2(\R)$, it follows that $\zeta_n\to 0$ in $L^2(\R\times \T)$ and hence $v_n\to v_0$ in~$L^2(\R\times \T)$.

Finally, we prove that the sequence $(v_n)_{n\in\N}$ converges to $v_0$ in $H^1(\R\times\T)$. Now by using the symmetric difference quotient in $\xi$-direction for $v_n$ and $v_0$ as in~\eqref{quetients} and the weak convergence of $\partial_{\xi}v_n$ to $v_0'$ as $n\to+\infty$ in $L^2(\R\times \T)$, one concludes that
\begin{equation}\label{eqlims}
c\int_{\R\times \T} (\partial_{\xi} v_n )^2=\int_{\R\times \T} g\partial_{\xi} v_n\mathop{\longrightarrow}_{n\to+\infty} \int_{\R\times \T}gv_0'=\int_{\R} \overline{g}v_0'
=c\int_{\R}(v_0')^2.
\end{equation}
By the weak convergence of $\partial_{\xi}v_n$ again, this implies that $\partial_{\xi}v_n$ converges to $v_0'$ strongly in $L^2(\R\times\T)$ as $n\to+\infty$. From arguments similar to~\eqref{partialy}, one has $\partial_y v_n\to0$ in $L^2(\R\times \T)$ as $n\to+\infty$. As a conclusion, $v_n$ converges to $v_0$ in $H^1(\R\times\T)$ as $n\to+\infty$ and this holds for the whole sequence by uniqueness of the limit. Lastly, once one knows the convergence of $(v_n)_{n\in\N}$ to $v_0$ in the strong sense in $H^1(\R\times\T)$, it is easy to see that the above estimates show that the convergence is actually uniform with respect to $g$ in the ball $B_A=\big\{g\in L^2(\R\times\T)\,|\,\|g\|_{L^2(\R\times\T)}\le A\big\}$, for every $A>0$.

{\it Step 2: the general case}. We first claim that for any $r>0$ and any $\omega\in L^2(\R\times \T)$ such that~$\|\omega\|_{L^2(\R\times \T)}\leq r$, there holds
\begin{equation}\label{uniequa}
\|M_{c_n,L_n}^{-1}(\omega)- M_{c,L_n}^{-1}(\omega)\|_{H^1(\R\times \T)}\leq C|c_n-c|r\ \hbox{ for all }n\in\N,
\end{equation}
where $C$ is a constant independent of $n$ and $r$. In fact, let $p_n=M_{c_n,L_n}^{-1}(\omega)$ and $q_n=M_{c,L_n}^{-1}(\omega)$. Then $M_{c,L_n}(p_n-q_n)=(c-c_n)\partial_{\xi}p_n$. Testing it with the function $(p_n-q_n)\in H^1(\R\times \T)$ and using arguments similar to the ones used~\eqref{hnorm}, one obtains that
$$\int_{\R\times \T}a\,(\tilde{\partial}_{L_n}p_n-\tilde{\partial}_{L_n}q_n)^2 + \frac{\beta}{2} (p_n-q_n)^2\leq \frac{(c_n-c)^2}{2\beta}\int_{\R\times\T}(\partial_{\xi} p_n)^2.$$
Similar arguments to~\eqref{partials} and~\eqref{partialy} yield
$$\left\{\baa{rcl}
\|\partial_{\xi}(p_n-q_n)\|_{L^2(\R\times \T)} & \leq & |c_n-c|\,c^{-1}\|\partial_{\xi} p_n\|_{L^2(\R\times \T)},\vspace{3pt}\\
\|\partial_y(p_n-q_n)\|_{L^2(\R\times \T)} & \leq & |L_n||c_n-c|\big(c^{-1}+(2\beta a^-)^{-1/2}\big)\|\partial_{\xi} p_n\|_{L^2(\R\times\T)}.\eaa\right.$$
Note also that $\|\partial_{\xi} p_n\|_{L^2(\R\times \T)}\leq(1/c_n)\|\omega\|_{L^2(\R\times \T)}$ from the estimate~\eqref{partials}. Then the proof of our claim (\ref{uniequa}) is finished, since the sequences $(L_n)_{n\in\N}$ and $(1/c_n)_{n\in\N}$ are bounded.

Next, we observe that
\begin{equation*}
\begin{split}
\|M_{c_n,L_n}^{-1}(g_n)\!-\!M_{c,0}^{-1}(\overline{g})\|_{H^1(\R\times\T)}\leq\,\,&\|M_{c_n,L_n}^{-1}(g_n)\!-\!M_{c,L_n}^{-1}(g_n)\|_{H^1(\R\times\T)}\vspace{3pt}\\
&+\|M_{c,L_n}^{-1}(g_n)\!-\!M_{c,L_n}^{-1}(g)\|_{H^1(\R\times\T)}+\|M_{c,L_n}^{-1}(g)\!-\!M_{c,0}^{-1}(\overline{g})\|_{H^1(\R\times\T)}.
\end{split}
\end{equation*}
From~\eqref{uniequa}, the first term of the right hand side converges to $0$ as $n\to+\infty$, and this holds uniformly with respect to $g\in B_A$, for every $A>0$. By Lemma~\ref{invertible} and the boundedness of the sequences~$(L_n)_{n\in\N}$ and $(1/c_n)_{n\in\N}$, one has $\|M_{c,L_n}^{-1}(g_n-g)\|_{H^1(\R\times\T)}\to0$ as $n\to+\infty$, uniformly with respect to $g\in B_A$ for every $A>0$. Lastly, the convergence of the last term to 0 follows from Step~1. Hence, the proof of Lemma~\ref{continuem} is finished.
\end{proof}

\begin{proof}[Proof of Lemma~\ref{continuem2}]
By Lemma~\ref{invertible} and~\eqref{uniequa}, one only needs to prove that
\begin{equation}\label{converl}
M_{c,L_n}^{-1}(g)\to  M_{c,L}^{-1}(g) \quad \hbox{in}\,\, H^1(\R\times\T)\,\,\,\hbox{as}\,\,n\to+\infty,
\end{equation}
uniformly for $g\in B_A$, for every $A>0$, and for every sequence $(L_n)_{n\in\N}\in\R^*$ such that $L_n\to L$ as $n\to+\infty$. Given any $g\in L^2(\R\times\T)$, let $w_n=M_{c,L_n}^{-1}(g)$ and $w=M_{c,L}^{-1}(g)$. Then, by similar estimates in Lemma~\ref{invertible}, the functions $w_n$ converge to $w$ strongly in $L^2_{loc}(\R\times \T)$ and weakly in~$H^1(\R\times \T)$ as $n\to+\infty$. Note that
$$M_{c,L_n}(w_n-w)=\underbrace{M_{c,L_n}(w_n)}_{=g}-\underbrace{M_{c,L}(w)}_{=g}-\big(M_{c,L_n}(w)-M_{c,L}(w)\big)=-\tilde{\partial}_{L_n}(a\tilde{\partial}_{L_n}w)+\tilde{\partial}_L(a\tilde{\partial}_Lw).$$
Integrating the above equation against $w_n-w$ yields
$$\baa{rcl}
\displaystyle\int_{\R\times\T}a\,(\tilde{\partial}_{L_n}w_n-\tilde{\partial}_{L_n}w)^2+\beta(w_n-w)^2 &
= & \displaystyle\int_{\R\times\T}\big(\tilde{\partial}_L(a\tilde{\partial}_Lw)-\tilde{\partial}_{L_n}(a\tilde{\partial}_{L_n}w)\big)(w_n-w)\vspace{3pt}\\
& = & \displaystyle\int_{\R\times\T}(w_n-w)\,\tilde{\partial}_L(a\tilde{\partial}_Lw)+\frac{a\,\partial_yw\,(\tilde{\partial}_{L_n}w_n-\tilde{\partial}_{L_n}w)}{L_n}\vspace{3pt}\\
& & \displaystyle+\int_{\R\times\T}a\,\partial_{\xi}w\,(\tilde{\partial}_{L_n}w_n-\tilde{\partial}_{L_n}w),\eaa$$
which converges to $0$ as $n\to+\infty$, since $\tilde{\partial}_L(a\tilde{\partial}_Lw)\in L^2(\R\times\T)$ and $w\in H^1(\R\times\T)$, and since~$w_n\rightharpoonup w$ and $\tilde{\partial}_{L_n}w_n\rightharpoonup\tilde{\partial}_{L_n}w$ weakly in $L^2(\R\times\T)$. Thus, $\|w_n-w\|_{L^2(\R\times\T)}\to 0$ and $\|\tilde{\partial}_{L_n}w_n-\tilde{\partial}_{L_n}w\|_{L^2(\R\times\T)}\to 0$ as $n\to+\infty$. The convergence $\partial_{\xi}w_n\to\partial_{\xi}w$ in $L^2(\R\times\T)$ follows then from the weak convergence and the fact that $\|\partial_{\xi}w_n\|_{L^2(\R\times\T)}^2=(1/c_n)\int_{\R\times\T}g\partial_{\xi}w_n\to(1/c)\int_{\R\times\T}g\partial_{\xi}w=\|\partial_{\xi}w\|_{L^2(\R\times\T)}^2$, as in~\eqref{eqlims}. On the other hand, observe that
\begin{equation*}
\begin{split}
\|\partial_y(w_n-w)\|_{L^2(\R\times\T)}=&\,\|L_n(\tilde{\partial}_{L_n}w_n-\partial_{\xi}w_n)-L(\tilde{\partial}_Lw-\partial_{\xi}w)\|_{L^2(\R\times\T)}\vspace{3pt}\\
\leq&\,\|L_n\tilde{\partial}_{L_n}w_n-L\tilde{\partial}_{L_n}w\|_{L^2(\R\times\T)} + |L|\|\tilde{\partial}_{L_n}w-\tilde{\partial}_Lw\|_{L^2(\R\times\T)}\vspace{3pt}\\
&\quad+|L_n|\|\partial_{\xi}w_n-\partial_{\xi}w\|_{L^2(\R\times\T)} + |L_n-L|\|\partial_{\xi}w\|_{L^2(\R\times\T)},
\end{split}
\end{equation*}
which converges to $0$ as $n\to+\infty $. Hence, the proof of the convergence~\eqref{converl} is complete. Lastly, the above calculations together with Lemma~\ref{invertible} imply that~\eqref{converl} holds uniformly with respect to the functions $g$ such that $\|g\|_{L^2(\R\times\T)}\le A$, for every $A>0$.
\end{proof}

Next, we do the proof of the continuity and differentiability properties of the function $G$ defined in~\eqref{defG}, which were stated in Lemma~\ref{continueg}.

\begin{proof}[Proof of Lemma~\ref{continueg}]
{\it Step 1: the continuity of $G$.} The continuity of $G_2$ is obvious from Cauchy-Schwarz inequality. Next we consider the continuity of $G_1$ at $(v,c,0)$ with $(v,c)\in H^1(\R\times\T)\times(0,+\infty)$. To do so, we first prove that~$M_{c_n,L_n}^{-1}\big(K(v_n,c_n,L_n)\big) \rightarrow M_{c,0}^{-1}\big(\overline{K(v,c,0)}\big)$ in $H^1(\R\times\T)$ as $n\to+\infty$ for any sequences~$(v_n)_{n\in\N}$ in $L^2(\R\times\T)$, $(c_n)_{n\in\N}$ in $(0,+\infty)$ and $(L_n)_{n\in\N}$ in $\R^*$ such that $v_n\to v$ in $L^2(\R\times\T)$, $c_n\to c$ and~$L_n\to0$. By Lemma~\ref{continuem}, it is sufficient to show that $K(v_n,c_n,L_n)\to K(v,c,0)$ in $L^2(\R\times\T)$ as $n\to+\infty$. Since the function $f(y,u)$ is globally Lipschitz-continuous in $u$ uniformly for $y\in\T$, one has that
\begin{equation*}
\begin{aligned}
&\|K_{c_n,L_n}(v_n)- K(v,c,0)\|_{L^2(\R\times\T)}\vspace{3pt}\\
=\,\,& \big\|L_na\chi\phi_0^{(3)}+(c_n\!-\!c)\phi_0'+c_nL_n\chi\phi_0''+\beta(v_n\!-\!v)+f(y,v_n\!+\!\phi_0\!+\!L_n\chi\phi_0')-f(y,v\!+\!\phi_0)\big\|_{L^2(\R\times\T)}\vspace{3pt}\\
\leq\,\,& |L_n|\|a\chi\phi_0^{(3)}\|_{L^2(\R\times\T)}+ |c_n-c|\|\phi_0'\|_{L^2(\R)}+|c_nL_n|\|\chi\phi_0''\|_{L^2(\R\times\T)}+\beta\|v_n-v\|_{L^2(\R\times\T)}\vspace{3pt}\\
&+C\|v_n-v\|_{L^2(\R\times\T)}+C\,|L_n|\|\chi\phi_0'\|_{L^2(\R\times\T)},
\end{aligned}
\end{equation*}
where $C$ is a constant depending only on $f$. Thus, $K(v_n,c_n,L_n)\to K(v,c,0)$ in $L^2(\R\times\T)$ as~$n\to+\infty$. Similarly, the fact that $\overline{K(v_n,c_n,0)}\to\overline{K(v,c,0)}$ in $L^2(\R)$ as $n\to+\infty$ together with Lemma~\ref{continuem} implies that $M_{c_n,0}^{-1}\big(\overline{K(v_n,c_n,0)}\big)\rightarrow M_{c,0}^{-1}\big(\overline{K(v,c,0)}\big)$ in $H^1(\R)$ as $n\to+\infty$. Therefore, the arguments of this paragraph show that $G_1$ is continuous at $(v,c,0)$.

For the continuity of $G_1$ at $(v,c,L)\in H^1(\R\times\T)\times(0,+\infty)\times\R^*$, the arguments are similar to those for $L=0$, the only additional fact being the use of Lemma~\ref{continuem2} instead of Lemma~\ref{continuem}.

{\it Step 2: $G$ is continuously Fr\'echet differentiable with respect to $(v,c)$.} First, for $v\in H^1(\R\times\T)$, we set $g[v](\xi,y)=f\big(y,\phi_0(\xi)+v(\xi,y)\big)$ in $\R\times\T$ and we prove that the function $g$ is continuously Fr\'echet differentiable from $H^1(\R\times\T)$ to $L^2(\R\times\T)$. Since the function $f(y,u)$ is globally Lipschitz-continuous in $u$ uniformly for $y\in\T$, and since $\phi_0\in L^2(\R^+)$ and $1-\phi_0\in L^2(\R^-)$, there is constant $C_1>0$ independent of $v$ such that
\begin{equation*}
\begin{split}
\|g[v]\|_{L^2(\R\times\T)}^2=&\ \int_{\R^-\times\T}|f(y,\phi_0+v)-f(y,1)|^2+\int_{\R^+\times\T}|f(y,\phi_0+v)-f(y,0)|^2\vspace{3pt}\\
\leq&\ C_1\big(\|1-\phi_0\|_{L^2(\R^-)}^2+\|\phi_0\|_{L^2(\R^+)}^2+\|v\|_{L^2(\R\times\T)}^2\big).
\end{split}
\end{equation*}
Hence, $g:v\mapsto g[v]$ is a map from $H^1(\R\times\T)$ to $L^2(\R\times\T)$. Since $f(y,u)$ is of class $C^{1,1}$ in $u\in\R$ uniformly for $y\in\T$, it follows from Sobolev injections and Cauchy-Schwarz inequality that for any $v,\,h\in H^1(\R\times\T)$ and $t\in(0,1]$,
\begin{equation}\label{gateaux}
\begin{split}
&\,\,\Big\|\frac{g[v+th]-g[v]}{t}-\partial_uf(y,\phi_0+v)h\Big\|_{L^2(\R\times\T)}\vspace{3pt}\\
=&\,\,\|\partial_uf(y,\phi_0+v+t\eta h)h -\partial_uf(y,\phi_0+v)h\|_{L^2(\R\times\T)}\vspace{3pt}\\
\leq&\,\,\|\partial_uf(y,\phi_0+v+t\eta h)-\partial_uf(y,\phi_0+v)\|_{L^4(\R\times\T)}\|h\|_{L^4(\R\times\T)}\leq C_2t\|h\|^2_{H^1(\R\times\T)},
\end{split}
\end{equation}
where $\eta$ is a function from $\R\times\T$ to $[0,1]$, and $C_2>0$ is independent on $t$. Thus,
$$\frac{g[v+th]-g[v]}{t}\to\partial_uf(y,\phi_0+v)h\ \hbox{ in }L^2(\R\times\T)\hbox{ as }t\to0^+,$$
which implies that $g$ is G\^ateaux differentiable at any point $v\in H^1(\R\times\T)$ with its derivative given by $A(v)h=\partial_uf(y,\phi_0+v)h$ for any $h\in H^1(\R\times\T)$. Actually,~\eqref{gateaux} implies that $g$ is Fr\'echet differentiable and since the map $v\mapsto A(v)$ from $H^1(\R\times\T)$ to $\mathcal{L}(H^1(\R\times\T),L^2(\R\times\T))$ is continuous, the function $g$ is continuously Fr\'echet differentiable in $H^1(\R\times\T)$.

In order to show the Fr\'echet differentiability of $G$ with respect to $(v,c)$, pick now any point~$(v,c,L)\in H^1(\R\times\T)\times(0, +\infty)\times\R$. Consider first the case $L\neq0$. Since $g$ is Fr\'echet differentiable at~$v$,~$K(v,c,L)$ is also Fr\'echet differentiable with respect to $v$ with
$$\partial_v K(v,c,L)(\tilde{v})=\big(\partial_uf(y,v+\phi_0+L\chi\phi_0')+\beta\big)\tilde{v}\ \hbox{ for all }\tilde{v}\in H^1(\R\times\T).$$
Since the linear operator $M_{c,L}^{-1}: L^2(\R\times\T)\to \mathcal{D}_L\subset H^1(\R\times T)$ is bounded, one has~$\partial_w\big(M_{c,L}^{-1}(w)\big)=M_{c,L}^{-1}$. Then, by the chain rule,
$$\partial_v\big(M_{c,L}^{-1}(K(v,c,L))\big)(\tilde{v})=M_{c,L}^{-1}\big((\partial_uf(y,v+\phi_0+L\chi\phi_0')+\beta)\tilde{v}\big)\ \hbox{ for all }\tilde{v}\in H^1(\R\times\T).$$
On the other hand, $\partial_c K(v,c,L)(\tilde{c})=(\phi_0'+L\chi\phi_0'')\tilde{c}$ for all $\tilde{c}\in\R$ while, by Lemma~\ref{continuem2},
$$\partial_c M_{c,L}^{-1}(v)(\tilde{c})=\lim_{d\to c} M_{c,L}^{-1}\Big(\frac{M_{c,L}(M_{d,L}^{-1}(v))-v }{d-c}\Big)\tilde{c}=-M_{c,L}^{-1}\big(\partial_{\xi}M_{c,L}^{-1}(v)\big)\tilde{c}\ \hbox{ for all }\tilde{c}\in\R.$$
As a consequence,
\begin{equation*}
\begin{split}
\partial_c\big(M_{c,L}^{-1}(K(v,c,L))\big)(\tilde{c})=&-M_{c,L}^{-1}\big(\partial_{\xi}M_{c,L}^{-1}(K(v,c,L))\big)\tilde{c}+M_{c,L}^{-1}(\phi_0'+L\chi\phi_0'')\tilde{c}\vspace{3pt}\\
=&-\tilde{c}\,M_{c,L}^{-1}\Big(\partial_{\xi}\big(M_{c,L}^{-1}(K(v,c,L))-\phi_0-L\chi\phi_0'\big)\Big)\hbox{ for all }\tilde{c}\in\R.
\end{split}
\end{equation*}
Hence, the function $G(v,c,L)$ is Fr\'echet differentiable with respect to $(v,c)$ with derivative
\begin{equation}\label{partialc}
\begin{split}
&\partial_{(v,c)}G(v,c,L)(\tilde{v},\tilde{c})\vspace{3pt}\\
=&\left(\!\!\!\begin{array}{c}
\tilde{v}+M_{c,L}^{-1}\big((\partial_uf(y,v\!+\!\phi_0\!+\!L\chi\phi_0')\!+\!\beta)\tilde{v}\big)
-\tilde{c}M_{c,L}^{-1}\Big(\partial_{\xi}\big(M_{c,L}^{-1}(K(v,c,L))\!-\!\phi_0\!-\!L\chi\phi_0'\big)\Big)\vspace{3pt}\\
2\displaystyle\int_{\R^+\times \T}(\phi_0+v+L\chi\phi_0')\tilde{v} \end{array}\!\!\!\right)\!\!.
\end{split}
\end{equation}
Similarly, for $L=0$, $G$ is Fr\'echet differentiable with respect to $(v,c)$ and for every $(v,c)\in H^1(\R\times\T)\times(0,+\infty)$ and $(\tilde{v},\tilde{c})\in H^1(\R\times\T)\times\R$,
\begin{equation}\label{difflo}
\partial_{(v,c)}G(v,c,0)(\tilde{v},\tilde{c})\!=\!\!\left(\!\!\!\begin{array}{c}
\tilde{v}\!+\!M_{c,0}^{-1}\!\Big(\!\overline{(\partial_uf(y,v\!+\!\phi_0)\!+\!\beta)\tilde{v}}\Big)
\!-\!\tilde{c}M_{c,0}^{-1}\!\Big(\!\overline{\partial_{\xi}\big(M_{c,0}^{-1}\big(\overline{K(v,c,0)}\big)\!-\!\phi_0\big)\!}\Big)\vspace{3pt}\\
2\displaystyle\int_{\R^+\times \T}(\phi_0+v)\tilde{v} \end{array}\!\!\!\!\right)\!\!.
\end{equation}

Finally, we prove that $\partial_{(v,c)}G:H^1(\R\times\T)\times(0,+\infty)\times\R\to\mathcal{L}(H^1(\R\times\T)\times\R,H^1(\R\times\T)\times\R)$ is continuous. Since the continuity of $\partial_{(v,c)}G_2$ is obvious from the Cauchy-Schwarz inequality, we only need to show that $\partial_{(v,c)}G_1$ is continuous. Let $(v,c,L)$ be any point in $H^1(\R\times\T)\times(0,+\infty)\times\R$ and let $(v_n)_{n\in\N}$ in $H^1(\R\times\T)$, $(c_n)_{n\in\N}$ in $(0,+\infty)$ and $(L_n)_{n\in\N}$ in $\R$ be such that $\|v_n-v\|_{H^1(\R\times\T)}\to0$, $c_n\to c$ and $L_n\to L$ as $n\to+\infty$. In the case $L=0$, by Lemma~\ref{continuem} and by~\eqref{partialc} and~\eqref{difflo}, for the continuity of $\partial_cG_1$ at~$(v,c,0)$, it is sufficient to prove that, whether $L_n$ be $0$ or not,
$$\left\{\baa{rcll}
\partial_{\xi}\Big(M_{c_n,L_n}^{-1}\big(K(v_n,c_n,L_n)\big)-\phi_0-L_n\chi\phi_0'\Big) & \displaystyle\mathop{\longrightarrow}_{n\to+\infty} & \partial_{\xi}\Big(M_{c,0}^{-1}\big(\overline{K(v,c,0)}\big)-\phi_0\Big) & \hbox{in }L^2(\R\times\T),\vspace{3pt}\\
\overline{\partial_{\xi}\Big(M_{c_n,0}^{-1}\big(\overline{K(v_n,c_n,0)}\big)-\phi_0\Big)} & \displaystyle\mathop{\longrightarrow}_{n\to+\infty} & \overline{\partial_{\xi}\Big(M_{c,0}^{-1}\big(\overline{K(v,c,0)}\big)-\phi_0\Big)} & \hbox{in }L^2(\R).\eaa\right.$$
In fact, these limits follow from Lemma~\ref{continuem} and the fact that $K(v_n,c_n,L_n)\to K(v,c,0)$ in $L^2(\R\times\T)$ (whence $\overline{K(v_n,c_n,L_n)}\to\overline{K(v,c,0)}$ in $L^2(\R)$) as $n\to+\infty$. By Lemmas~\ref{invertible},~\ref{continuem} and by~\eqref{partialc} and~\eqref{difflo}, for the continuity of $\partial_vG_1$ at $(v,c,0)$, it is sufficient to show that
$$\big(\partial_uf(y,v_n+\phi_0+L_n\chi\phi_0')+\beta\big)\tilde{v}\mathop{\longrightarrow}_{n\to+\infty}\big(\partial_uf(y,v+\phi_0)+\beta\big)\tilde{v}\ \hbox{ in }L^2(\R\times\T)$$
(whence the same property for the averaged functions with respect to $y$), uniformly with respect to $\|\tilde{v} \|_{H^1(\R\times\T)}\leq1$. Since the function $f(y,u)$ is of class $C^{1,1}$ in $u\in\R$ uniformly for $y\in\T$, these convergences follow from similar arguments as~\eqref{gateaux}.

In the case $L\neq0$, as in the previous paragraph, one has
$$\left\{\baa{rcl}
\partial_{\xi}\big(M_{c_n,L_n}^{-1}(K(v_n,c_n,L_n))-\phi_0-L_n\chi\phi_0'\big) & \displaystyle\mathop{\longrightarrow}_{n\to+\infty} & \partial_{\xi}\big(M_{c,L}^{-1}(K(v,c,L))-\phi_0-L\chi\phi_0'\big)\vspace{3pt}\\
\big(\partial_uf(y,v_n+\phi_0+L_n\chi\phi_0')+\beta\big)\tilde{v} & \displaystyle\mathop{\longrightarrow}_{n\to+\infty} & \big(\partial_uf(y,v+\phi_0+L\chi\phi_0')+\beta\big)\tilde{v}\eaa\right.$$
in $L^2(\R\times\T)$, uniformly with respect to $\|\tilde{v} \|_{H^1(\R\times\T)}\leq1$. Thus, the continuity of $\partial_{(v,c)}G_1$ at~$(v,c,L)$ with $L\neq 0$ follows from Lemma~\ref{continuem2}.

{\it Step 3: the invertibility of $Q=\partial_{(v,c)}G(0,c_0,0)$.} We first observe that, from~\eqref{difflo} and~$\overline{K(0,c_0,0)}=0$, the operator $Q$ is given by
\be\label{difflo2}
Q(\tilde{v},\tilde{c})\!=\!\Big(\tilde{v}\!+\!M_{c_0,0}^{-1}\big(\overline{(\partial_uf(y,\phi_0)\!+\!\beta)\tilde{v}}\big)+\tilde{c}M_{c_0,0}^{-1}(\phi_0'),2\int_{\R^+\!\times\!\T}\!\!\!\!\!\!\phi_0\tilde{v}\Big)\!\hbox{ for all }(\tilde{v},\tilde{c})\!\in\!H^1(\R\!\times\!\T)\!\times\!\R.
\ee

Let us now show that $Q$ has a closed range $R(Q)$. Let $(\tilde{v}_n)_{n\in\N},\,(\tilde{g}_n)_{n\in\N}$ in $H^1(\R\times\T)$ and~$(\tilde{c}_n)_{n\in\N},\,(\tilde{d}_n)_{n\in\N}$ in $\R$ be such that $Q(\tilde{v}_n,\tilde{c}_n)=(\tilde{g}_n,\tilde{d}_n)\to (\tilde{g},\tilde{d})$ in $H^1(\R\times\T)\times\R$ as $n\to+\infty$. Define $v_n=\tilde{v}_n-\tilde{g}_n$. One can see from~\eqref{difflo2} that $v_n=-M_{c_0,0}^{-1}\big(\overline{(\partial_uf(y,\phi_0)+\beta)\tilde{v}_n}\big)-\tilde{c}_nM_{c_0,0}^{-1}(\phi_0')\in H^2(\R)$. Then, by definition of $H$ and since~$v_n$ is independent of $y$, it follows that
\begin{equation}\label{vn}
\begin{split}
H(v_n)=M_{c_0,0}v_n+\big(\overline{f}'(\phi_0)+\beta\big)v_n
&=-\overline{\big(\partial_uf(y,\phi_0)+\beta\big)\tilde{v}_n}-\tilde{c}_n\phi_0'
+\overline{\big(\partial_uf(y,\phi_0)+\beta\big)v_n}\vspace{3pt}\\
&=-\overline{\big(\partial_uf(y,\phi_0)+\beta\big)\tilde{g}_n} -\tilde{c}_n\phi_0'
\end{split}
\end{equation}
Testing it with $0\neq w \in \ker(H^*)$ implies that $\tilde{c}_n\int_{\R} w\phi_0'=-\int_{\R\times \T} \big(\partial_uf(y,\phi_0)+\beta\big)w\tilde{g}_n$. Notice now that $\int_{\R} w\phi_0'\neq 0$. Indeed, otherwise, $\phi_0'$ would be in the orthogonal of $\R w=\ker(H^*)$, that is in the closed range~$R(H)$. But the property $\phi'_0\in R(H)$ is impossible since $\ker(H)=\R \phi_0'$ and~$0$ is an algebraically simple eigenvalue of $H$. Then, since $\tilde{g}_n\to\tilde{g}$ in $H^1(\R\times \T)$ as $n\to+\infty$, it follows that~$\tilde{c}_n$ converges to some $\tilde{c}\in\R$ and then, by~\eqref{vn}, $H(v_n)$ converges to some $g$ in $L^2(\R)$ with
\begin{equation}\label{eqlinm}
-\overline{\big(\partial_uf(y,\phi_0)+\beta\big)\tilde{g}} -\tilde{c}\phi_0'=g.
\end{equation}
Since $R(H)$ is closed, there is $v\in H^2(\R)$ such that $H(v)=g$. Set $\tilde{v}=v+\eta\phi_0'+\tilde{g}\in H^1(\R\times\T)$, where $\eta\in\R$ is chosen as the unique real number satisfying $2\int_{\R^+\times \T}\phi_0\tilde{v}=\tilde{d}$. Since $\tilde{v}-\tilde{g}=v+\eta\phi'_0\in H^2(\R)$ and $H(\tilde{v}-\tilde{g})=H(v+\eta\phi_0')=g$, one has $M_{c_0,0}(\tilde{v}-\tilde{g})+\overline{\big(\partial_uf(y,\phi_0)+\beta\big)(\tilde{v}-\tilde{g})}=g$. Then, by composing by $M_{c_0,0}^{-1}$ and using~\eqref{eqlinm}, one gets that $\tilde{v}\!+\!M_{c_0,0}^{-1}\big(\overline{(\partial_uf(y,\phi_0)\!+\!\beta)\tilde{v}}\big)\!+\!\tilde{c}M_{c_0,0}^{-1}(\phi_0')\!=\!\tilde{g}$. Therefore, $(\tilde{g},\tilde{d})=Q(\tilde{v},\tilde{c})$, and $R(Q)$ is closed.

Next, we prove that $Q$ has a trivial kernel. Suppose that $Q(\tilde{v},\tilde{c})=(0,0)$ for some $(\tilde{v},\tilde{c})\in H^1(\R\times\T)\times\R$. Since $M_{c_0,0}^{-1}$ is a map from $L^2(\R)$ to $H^2(\R)$, then $\partial_y\tilde{v}=0$ and $\tilde{v}\in H^2(\R)$ by~\eqref{difflo}. Furthermore, $M_{c_0,0}(\tilde{v})+\overline{(\partial_uf(y,\phi_0)+\beta)\tilde{v}}=-\tilde{c}\phi_0'$, that is, $H(\tilde{v})=-\tilde{c}\phi_0'$. Since $0$ is an algebraically simple eigenvalue of $H$ with kernel $\R\phi'_0$, it follows that $\tilde{c}=0$ and $\tilde{v}=\sigma\phi_0'$ with~$\sigma\in\R$. Since~$0=2\sigma\int_{\R^+\times \T}\phi_0\phi_0'=-\sigma(\phi_0(0))^2$ and $\phi_0(0)\neq0$, one infers that $\sigma=0$. Thus, the kernel of $Q$ is reduced to $\{(0,0)\}$.

Finally, we prove that the kernel of the adjoint operator $Q^*$ is also reduced to $\{(0,0)\}$. Let~$(\tilde{v},\tilde{c})\in H^1(\R\times\T)\times \R$ such that $Q^*(\tilde{v},\tilde{c})=0$, that is,
\begin{equation}\label{diffadj}
\begin{split}
0&= \big<Q(\tilde{w},\tilde{d}),(\tilde{v},\tilde{c})\big>_{H^1(\R\times\T)\times\R}\vspace{3pt}\\
&= \Big( \tilde{w}+M_{c_0,0}^{-1}\big(\overline{\partial_uf(y,\phi_0)+\beta)\tilde{w}}\big)+\tilde{d}M_{c_0,0}^{-1}(\phi_0'),\tilde{v}\Big)_{H^1(\R\times\T)} + 2\tilde{c}\int_{\R^+\times \T}\phi_0\tilde{w}
\end{split}
\end{equation}
for all $(\tilde{w},\tilde{d})\in H^1(\R\times\T)\times \R$, where $\big(v_1,v_2\big)_{H^1(\R\times\T)}=\int_{\R\times \T}v_1v_2+\partial_{\xi}v_1\partial_{\xi}v_2+\partial_yv_1\partial_yv_2$ for all $v_1,\,v_2\in H^1(\R\times\T)$. We notice that
$$\phi_0'+M_{c_0,0}^{-1}\big(\overline{(\partial_uf(y,\phi_0)+\beta)\phi_0'}\big)=\phi_0'+M_{c_0,0}^{-1}\big((\overline{f}'(\phi_0)+\beta)\phi_0'\big)=M_{c_0,0}^{-1}(H(\phi_0'))=0.$$
Therefore, taking $(\tilde{w},\tilde{d})=(\phi_0',0)$ in~\eqref{diffadj} yields $0=2\tilde{c}\int_{\R^+\times \T}\phi_0\phi_0'=-\tilde{c}(\phi_0(0))^2$, whence $\tilde{c}=0$. Next, by choosing $(\tilde{w},\tilde{d})=(0,1)$ in~\eqref{diffadj}, it follows that
\begin{equation}\label{diffadj1}
0=\big(M_{c_0,0}^{-1}(\phi_0'),\tilde{v}\big)_{H^1(\R\times\T)}=\big(M_{c_0,0}^{-1}(\phi_0'),\overline{\tilde{v}}\big)_{H^1(\R)}
\end{equation}
since $M_{c_0,0}^{-1}(\phi'_0)\in H^1(\R)$. On the other hand, if $\tilde{w}\in H^2(\R)$ is independent of $y\in\T$ and $\tilde{d}=0$ in~\eqref{diffadj}, then
\begin{equation}\label{diffadj2}
\begin{split}
0&=\big( \tilde{w}+M_{c_0,0}^{-1}\big(\overline{(\partial_uf(y,\phi_0)+\beta)\tilde{w}}\big),\tilde{v}\big)_{H^1(\R\times\T)}\vspace{3pt}\\
&=\big(M_{c_0,0}^{-1}\big( M_{c_0,0}(\tilde{w})+(\overline{f}'(\phi_0)+\beta)\tilde{w}\big),\overline{\tilde{v}}\big)_{H^1(\R)}=\big(M_{c_0,0}^{-1}(H(\tilde{w})),\overline{\tilde{v}}\big)_{H^1(\R)}.
\end{split}
\end{equation}
Since the closed image $R(H)$ of $H$ is orthogonal to $\ker(H^*)$ and since $\phi_0'\not\in R(H)$ (as already emphasized), it follows from~\eqref{diffadj1} and~\eqref{diffadj2} that $\big(M_{c_0,0}^{-1}(P(\phi_0')),\overline{\tilde{v}}\big)_{H^1(\R)}\!=\!0$, where $P(\phi_0')~\!\!\neq~\!\!0$ is the orthogonal projection of $\phi_0'$ onto $\ker(H^*)$. Since $0$ is an algebraically simple whence geometrically simple eigenvalue of $H^*$, it follows that $\ker(H^*)=\R\,P(\phi'_0)$ and $\big(M_{c_0,0}^{-1}(w),\overline{\tilde{v}}\big)_{H^1(\R)}=0$ for all~$w\in \ker(H^*)$. This together with~\eqref{diffadj2} and the fact that $R(H)$ is orthogonal to $\ker(H^*)$ implies that $\big(M_{c_0,0}^{-1}(w),\overline{\tilde{v}}\big)_{H^1(\R)}=0$ for all $w\in L^2(\R)$. Then, back to equation~\eqref{diffadj}, one gets that
\begin{equation*}
\begin{split}
0&=\big(\tilde{w},\tilde{v}\big)_{H^1(\R\times\T)}+\big(M_{c_0,0}^{-1}\big(\overline{(\partial_uf(y,\phi_0)+\beta)\tilde{w}}\big),\tilde{v}\big)_{H^1(\R\times\T)}\vspace{3pt}\\
&= \big(\tilde{w},\tilde{v}\big)_{H^1(\R\times\T)} + \big(M_{c_0,0}^{-1}\big(\overline{(\partial_uf(y,\phi_0)+\beta)\tilde{w}}\big),\overline{\tilde{v}}\big)_{H^1(\R)}= \big(\tilde{w},\tilde{v}\big)_{H^1(\R\times\T)},
\end{split}
\end{equation*}
for all $\tilde{w}\in H^1(\R\times\T)$. Thus, $\tilde{v}=0$ and $Q^*$ has a trivial kernel. The proof of Lemma~\ref{continueg} is thereby complete.
\end{proof}

%%%%%%%%%%%%%%%%%%%%%%%%%%%%%%%%%%%%%%%%
%%%%%%%%%%%%%%%%%%%%%%%%%%%%%%%%%%%%%%%%


\begin{thebibliography}{AAA}

\footnotesize{
\bibitem{agn} V.S. Afraimovich, L.Y. Glebsky, V.I. Nekorkin, {\it Stability of stationary sates and topological spatial chaos in multidimensional lattice dynamical systems}, Random Comput. Dynam. {\bf 2} (1994), 287-303.
\bibitem{abc} N. D. Alikakos, P. W. Bates, X. Chen, {\it Periodic traveling waves and locating oscillating patterns in multidimensional domains}, Trans. Amer. Math. Soc. {\bf 351} (1999), 2777-2805.
\bibitem{amo} D.G. Aronson, N.V. Mantzaris, H.G. Othmer, {\it Wave propagation and blocking in inhomogeneous media}, Disc. Cont. Dyn. Syst.~A {\bf 13} (2005), 843-876.
\bibitem{aw} D.G. Aronson, H.F. Weinberger, {\it Multidimensional nonlinear diffusions arising in population genetics}, Adv. Math. {\bf 30} (1978), 33-76.
\bibitem{bc} J. Bell, C. Cosner, {\it Threshold behavior and propagation for nonlinear differential-difference systems motivated by modelling myelineated axons}, Q. Appl. Math. {\bf 42} (1984), 1-14.
\bibitem{bbc} H. Berestycki, J. Bouhours, G. Chapuisat, {\it Front blocking and propagation in cylinders with varying cross section}, preprint.
\bibitem{bh1} H. Berestycki, F. Hamel, {\it Non-existence of travelling fronts solutions for some bistable reaction-diffusion equations}, Adv. Diff. Equations {\bf 5} (2000), 723-746.
\bibitem{bh2} H. Berestycki, F. Hamel, {\it Front propagation in periodic excitable media}, Comm. Pure Appl. Math. {\bf 55} (2002), 949-1032.
\bibitem{bh3} H. Berestycki, F. Hamel, {\it Generalized transition waves and their properties}, Comm. Pure Appl. Math. {\bf 65} (2012), 592-648.
\bibitem{bhn} H. Berestycki, F. Hamel, N. Nadirashvili, {\it The speed of propagation for KPP type problems: I~-~Periodic framework}, J. Europ. Math. Soc. {\bf 7} (2005),  173-213.
\bibitem{bhr1} H. Berestycki, F. Hamel,  L. Roques, {\it Analysis of the periodically fragmented environment model: I~-~Species persistence}, J. Math. Biol. {\bf 51} (2005), 75-113.
\bibitem{bhr2} H.~Berestycki, F. Hamel, L.~Roques, {\it Analysis of the periodically fragmented environment model: II~-~Biological invasions and pulsating travelling fronts}, J. Math. Pures Appl. {\bf 84} (2005), 1101-1146.
\bibitem{cmv} J.W. Cahn, J. Mallet-Paret, E.S. van Vleck, {\it Travelling wave solutions for systems of ODE's on a two-dimensional spatial lattice}, SIAM J. Appl. Math. {\bf 59} (1999), 455-493.
\bibitem{cg} G. Chapuisat, E. Grenier, {\it Existence and non-existence of progressive wave solutions for a bistable reaction-diffusion equation in an infinite cylinder whose diameter is suddenly increased}, Comm. Part. Diff. Equations {\bf 30} (2005), 1805-1816.
\bibitem{cs} S.-N. Chow, W. Shen, {\it Dynamics in a discrete Nagumo equation: spatial topological chaos}, SIAM J. Appl. Math. {\bf 55} (1995), 1764-1781.
\bibitem{d} K. Deimling, {\it Nonliear Functional Analysis}, Springer-Verlag, Berlin-New York, 1985.
\bibitem{dhz2} W. Ding, F. Hamel, X.-Q. Zhao, {\it Propagation phenomena for periodic bistable reaction-diffusion equations}, preprint.
\bibitem{dm} Y. Du, H. Matano, {\it Convergence and sharp thresholds for propagation in nonlinear diffusion problems}, J.~Europ. Math. Soc. {\bf 12} (2010), 279-312.
\bibitem{dgm} A. Ducrot, T. Giletti, H. Matano, {\it Existence and convergence to a propagating terrace in one-dimensional reaction-diffusion equations}, Trans. Amer. Math. Soc. (2014), forthcoming.
\bibitem{e} M. El Smaily, {\it Min-max formulas for the speeds of pulsating travelling fronts in periodic excitable media}, Ann. Mat. Pura Appl. {\bf 189} (2010), 47-66.
\bibitem{ehr} M. El Smaily, F. Hamel, L. Roques, {\it  Homogenization and influence of fragmentation in a biological invasion model}, Disc. Cont. Dyn. Syst.~A {\bf 25} (2009), 321-342.
\bibitem{fz} J. Fang, X.-Q. Zhao, {\it Bistable traveling waves for monotone semiflows with applications}, J.~Europ. Math. Soc. (2014), forthcoming.
\bibitem{fm} P.C. Fife, J.B. McLeod, {\it The approach of solutions of non-linear diffusion equations to traveling front solutions}, Arch. Ration. Mech. Anal. {\bf 65} (1977), 335-361.
\bibitem{g} J. G\"artner, {\it Bistable reaction-diffusion equations and excitable media}, Math. Nachr. {\bf 112} (1983), 125-152.
\bibitem{h} F. Hamel, {\it Qualitative properties of monostable pulsating fronts: exponential decay and monotonicity}, J.~Math. Pures Appl. {\bf 89} (2008), 355-399
\bibitem{hfr} F. Hamel, J. Fayard, L. Roques, {\it Spreading speeds in slowly oscillating environments}, Bull. Math. Biol. {\bf 72} (2010), 1166-1191.
\bibitem{hnr} F. Hamel, G. Nadin, L. Roques, {\it A viscosity solution method for the spreading speed formula in slowly varying media}, Indiana Univ. Math.~J. {\bf 60} (2011), 1229-1247.
\bibitem{hr1} F. Hamel, L.~Roques, {\it Uniqueness and stability properties of monostable pulsating fronts}, J. Europ. Math. Soc. {\bf 13} (2011), 345-390.
\bibitem{hr2} F. Hamel, L. Roques, {\it Persistence and propagation in periodic reaction-diffusion models}, Tamkang J.~Math., forthcoming.
\bibitem{he1} S. Heinze, {\it Homogenization of flame fronts}, preprint IWR, Heidelberg, 1993.
\bibitem{he2} S. Heinze, {\it Wave solutions to reaction-diffusion systems in perforated domains}, Z. Anal. Anwendungen {\bf 20} (2001), 661-670.
\bibitem{hps} S. Heinze, G. Papanicolaou, A. Stevens, {\it Variational principles for propagation speeds in inhomogeneous media}, SIAM J. Appl. Math. {\bf 62} (2001), 129-148.
\bibitem{henry} D. Henry, {\it Geometric Theory of Semilinear Parabolic Equations, Lecture Notes in Mathematics}, {\bf 840}, Springer-Verlag, 1981.
\bibitem{he} P. Hess, {\it Periodic-parabolic Boundary Value Problems and Positivity}, Longman Scientific \& Technical, 1991.
\bibitem{hz} W. Hudson, B. Zinner, {\it Existence of travelling waves for reaction-diffusion equations of Fisher type in periodic media}, In: Boundary Value Problems for Functional-Differential Equations, J. Henderson (ed.), World Scientific, 1995, 187-199.
\bibitem{k} J.P. Keener, {\it Propagation and its failure in coupled systems of discrete excitable cells}, SIAM J. Appl. Math. {\bf 47} (1987), 556-572.
\bibitem{lk} T.J. Lewis, J.P. Keener, {\it Wave-block in excitable media due to regions of depressed excitability}, SIAM J.~Appl. Math. {\bf 61} (2000), 293-316.
\bibitem{lz1} X. Liang, X.-Q. Zhao, {\it Asymptotic speeds of spread and traveling waves for monotone semiflows with applications}, Comm. Pure Appl. Math. {\bf 60} (2007), 1-40.
\bibitem{lz2} X. Liang, X.-Q. Zhao, {\it Spreading speeds and traveling waves for abstract monostable evolution systems}, J.~Funct. Anal. {\bf 259} (2010), 857-903.
\bibitem{mp} J. Mallet-Paret, {\it The global structure of traveling waves in spatially discrete dynamical systems}, J.~Dyn. Diff. Equations {\bf 11} (1999), 49-127.
\bibitem{n} G. Nadin, {\it The effect of Schwarz rearrangement on the periodic principal eigenvalue of a nonsymmetric operator}, SIAM J. Math. Anal. {\bf 41} (2010), 2388-2406.
\bibitem{n2} G. Nadin, {\it Critical travelling waves for general heterogeneous one-dimensional reaction-diffusion equations}, Ann. Inst. H.~Poincar\'e, Analyse Non Lin\'eaire, forthcoming.
\bibitem{p} J.P. Pauwelussen, {\it Nerve impulse propagation in a branching nerve system: A simple model}, Physica D {\bf 4} (1981), 67-88.
\bibitem{s} W. Shen, {\it Traveling waves in time almost periodic structures governed by bistable nonlinearities, I.~Stability and uniqueness}, J.~Diff. Equations {\bf 159} (1999), 1-54.
\bibitem{skt} N. Shigesada, K. Kawasaki, E. Teramoto, {\it Traveling periodic waves in heterogeneous environments}, Theor. Pop. Bio. {\bf 30} (1986), 143-160.
\bibitem{w} H. F. Weinberger, {\it On spreading speeds and traveling waves for growth and migration models in a periodic habitat}, J. Math. Biol. \textbf{45} (2002), 511-548.
\bibitem{x1} X. Xin, {\it Existence and uniqueness of travelling waves in a reaction-diffusion equation with combustion nonlinearity}, Indiana Univ. Math. J. {\bf 40} (1991), 985-1008.
\bibitem{x2} J.X. Xin, {\it Existence and stability of travelling waves in periodic media governed by a bistable nonlinearity}, J.~Dyn. Diff. Equations {\bf 3} (1991), 541-573.
\bibitem{x3} J.X. Xin, {\it Existence of planar flame fronts in convective-diffusive-periodic media}, Arch. Ration. Mech. Anal. {\bf 121} (1992), 205-233.
\bibitem{x4} J.X. Xin, {\it Existence and nonexistence of traveling waves and reaction-diffusion front propagation in periodic media}, J.~Stat. Phys. {\bf 73} (1993), 893-926.
\bibitem{x5} J.X. Xin, {\it Front propagation in heterogeneous media}, SIAM Review {\bf 42} (2000), 161-230.
\bibitem{xz} J.X. Xin, J. Zhu, {\it Quenching and propagation of bistable reaction-diffusion fronts in multidimensional periodic media}, Physica~D {\bf 81} (1995), 94-110.
\bibitem{xzh} D. Xu, X.-Q. Zhao, {\it Bistable waves in an epidemic model}, J. Dyn. Diff. Equations, {\bf 16} (2004), 679-707.
\bibitem{zz} Y. Zhang, X.-Q. Zhao, {\it Bistable travelling waves for a reaction and diffusion model with seasonal succession}, Nonlinearity {\bf 26} (2013), 691-709.
\bibitem{zh}  X.-Q. Zhao, {\it Dynamical Systems in Population Biology}, Springer, New York, 2003.
\bibitem{zl} A. Zlato{\v{s}}, {\it Sharp transition between extinction and propagation of reaction}, J.~Amer. Math. Soc. {\bf 19} (2006), 251-263.

}
\end{thebibliography}
\end{document}